\documentclass{sl}
\usepackage{subfiles} 

\usepackage{slsec}     
\usepackage{stud_log} 
\usepackage{slfootNew}
\usepackage{slthm}     

\usepackage{etex}

\usepackage[T1]{fontenc}
\usepackage[utf8x]{inputenc}
\usepackage[english]{babel}
\usepackage{amssymb, amsmath} 
\usepackage{ebproof}
\usepackage{mathrsfs} 
\usepackage{stmaryrd} 
\usepackage{cmll} 

\usepackage{braket} 
\usepackage{mathtools}
\usepackage{multirow} 
\usepackage{multicol} 
\usepackage{color,soul} 
\usepackage{enumitem} 
\usepackage{microtype}
\usepackage{mathdots} 
\usepackage{url}

\usepackage[
hidelinks]{hyperref}
\hypersetup{unicode=true}

\newtheorem{defin}{Definition}
\newtheorem{prop}[defin]{Proposition}
\newtheorem{lemma}[defin]{Lemma}
\newtheorem{thm}[defin]{Theorem}
\newtheorem{cor}[defin]{Corollary}

\newtheorem*{exe*}{Exercise}

\theoremstyle{definition}
\newtheorem{rmk}[defin]{Remark}
\newtheorem*{note}{Notation}

\newtheorem{ex}[defin]{Example}
\newtheorem*{ex*}{Example}

{\end{enumerate}} 
\newenvironment{romanenum}{%
\begin{enumerate}[label=(\roman*),ref=(\textup{\roman*})]}%
{\end{enumerate}} 

\newcommand{\Nat}{\mathbb{N}}
\newcommand{\Natp}{\mathbb{N}^+}

\newcommand{\Ax}{\mathsf{ax}}
\newcommand{\Raa}{\mathsf{raa}}
\newcommand{\Efq}{\mathsf{efq}}
\newcommand{\RAA}{\mathsf{RAA}}
\newcommand{\RAAp}{\mathsf{RAA}^+
}
\newcommand{\Intro}{\mathsf{i}}
\newcommand{\IntroOne}{\Intro_\mathsf{1}}
\newcommand{\IntroTwo}{\Intro_\mathsf{2}}
\newcommand{\ElimOne}{\Elim_\mathsf{1}}
\newcommand{\ElimTwo}{\Elim_\mathsf{2}}
\newcommand{\Elim}{\mathsf{e}}
\newcommand{\Dist}[1]{\mathsf{dist}_{#1}}
\newcommand{\Size}{\mathsf{size}}

\newcommand{\Sub}[2]{\{#1/#2\}}
\newcommand{\Nk}{\mathsf{NK}}
\newcommand{\Nj}{\mathsf{NJ}}
\newcommand{\Nm}{\mathsf{NM}}
\newcommand{\Transl}{\mathsf{m}}
\newcommand{\Translj}{\mathsf{j}}
\newcommand{\Var}{P}
\newcommand{\Vartwo}{Q}
\newcommand{\Varthree}{R}
\newcommand{\Ie}{\textup{i.e.}}
\newcommand{\Defeq}{\coloneqq}
\newcommand{\Mstandard}{$\mathsf{m}$-standard}
\newcommand{\Jstandard}{$\mathsf{j}$-standard}

\let\Gamma\varGamma
\let\Delta\varDelta
\let\Theta\varTheta
\let\Lambda\varLambda
\let\Xi\varXi
\let\Pi\varPi
\let\Sigma\varSigma
\let\Upsilon\varUpsilon
\let\Phi\varPhi
\let\Psi\varPsi
\let\Omega\varOmega

\newcommand\TITLE{Postponement of $\Raa$ and Glivenko's theorem, revisited (extended version)}

\HeadingsInfo{Guerrieri and Naibo}\TITLE

\begin{document}

\setcounter{page}{1} 
\twoAuthorsTitleoneline {Giulio Guerrieri}{Alberto Naibo}\TITLE

\newcommand\AFFILIATIONS{
 \AuthorAdressEmail{Giulio Guerrieri}{
   Department of Computer Science, University of Oxford \\
   Oxford, United Kingdom
   }{giulio.guerrieri@cs.ox.ac.uk}

 \AdditionalAuthorAddressEmail{Alberto Naibo}{
    IHPST (UMR 8590), Universit\'e Paris 1 Panth\'eon--Sorbonne\\
    CNRS, ENS\\
    Paris, France
    }{alberto.naibo@univ-paris1.fr}
}

\PresentedReceived{Name of Editor}{\today{}}




  
\begin{abstract}
This article focuses on the technique of postponing the application of the \textit{reduction ad absurdum} rule ($\Raa$) in classical natural deduction. First, it is shown how this technique is connected with two normalization strategies for classical logic: one given by Prawitz, and the other by Seldin. Secondly, a variant of Seldin’s strategy for the postponement of $\Raa$ is proposed, and the similarities with Prawitz’s approach are investigated. In particular, it is shown that, as for Prawitz, it is possible to use this variant of Seldin's strategy in order to induce a negative translation from classical to intuitionistic and minimal logic, which is nothing but a variant of Kuroda's translation. Through this translation, Glivenko's theorem for intuitionistic and minimal logic is proven.

\end{abstract}


%

\section{Introduction}
\label{sect:intro}

Among the inference rules of classical natural deduction, the 
\emph{reductio ad absurdum} -- denoted by $\Raa$ -- formalizes the principle of a ``proof by contradiction'': if a contradiction is obtained 
from $\lnot A$, then $A$ can asserted and the hypothesis $\lnot A$ can be dropped, \Ie~discharged.
This principle is rejected by intuitionism and, in general, by constructive accounts of logic. 
More precisely, $\Raa$ is not an admissible inference rule in intuitionistic natural deduction, even if the latter contains a special case of $\Raa$, called \emph{ex falso quodlibet} and denoted by $\Efq$. This rule formalizes the ``principle of explosion'': from a contradiction anything can be asserted, without 
discharging any hypothesis. The rule $\Efq$, in turn, is not admissible in a more restrictive system of constructive logic, like minimal natural deduction.

For natural deduction of first-order classical logic (with the 
$\Raa$ rule) 
there are two general strategies for defining a weak normalization procedure: one due to Prawitz and one due to Seldin (see \cite[pp.~282--283]{PereiraHaeusler15} for a first comparison). 

Prawitz's idea \cite{Prawitz65} is to restrict to the fragment $\{\neg, \wedge, \to, \bot, \forall\}$, reduce all the applications of $\Raa$ to atomic formulas, and then apply whatever normalization strategy for intuitionistic logic one likes (see \cite[pp.~39--41]{Prawitz65}). 
Seldin's idea \cite{Seldin86}, on the other hand, is to restrict to the fragment $\{\neg, \land, \lor, \to,\allowbreak \bot, \exists\}$, reduce all the applications of $\Raa$ present in a derivation tree to at most one single application occurring as the last step of the tree (this is the \emph{postponement of $\Raa$}), and then apply whatever normalization strategy for intuitionistic logic one likes (see \cite[pp. 638--645]{Seldin86}).%
\footnote{Both Prawitz and Seldin consider that a normal derivation is a derivation in which there are no detours. However, the notion of detour specific to classical logic is not the same for them. We will analyze this point in detail in \S\ref{sect:overview}.}
Prawitz's and Seldin's strategies can be seen as ``dual'': the former breaks down classical reasoning into a number of atomic steps of $\Raa$, the latter compactifies classical reasoning into one single (possibly complex) step of $\Raa$, the final one. 
Moreover, a peculiarity of Seldin's strategy is that Glivenko's theorem (for intuitionistic logic) can be obtained as an immediate consequence: it is sufficient to drop the $\Raa$ rule of a normal (according to Seldin) derivation in classical logic, and replace it with a $\neg$-introduction rule discharging the same assumptions (see \cite[\S3]{Seldin86}).

Glivenko's theorem, in its original formulation \cite{Glivenko29}, states that if a formula is provable in classical propositional logic then its double negation is provable in intuitionistic propositional logic (the converse trivially holds). 
Thus, it allows propositional classical logic to be embedded into propositional intuitionistic logic.
Several refinements and generalizations of Glivenko's theorem are well-known (see \cite{Espindola13} for a partial survey): there are extensions of the result to first-order  \cite{Kleene52,Seldin86,Espindola13}, second-order \cite{Zdanowski09} and substructural logics \cite{GalatosOno06,Ono09,FarahaniOno12}, and there is an embedding of classical logic into minimal logic \cite{ErtolaSagastume08}.
All these results are obtained using different approaches, both syntactic and semantic.

Starting from a comparison of Prawitz's and Seldin's weak normalization strategies, we will 
modify Seldin's reduction steps for the postponement of $\Raa$ in order to 
achieve two goals: on the one hand, we induce two variants of Kuroda's negative translation \cite{Kuroda51,FerreiraOliva12,BrownRizkallah14} 
of first-order classical logic into intuitionistic and minimal logic; on the other hand, we give an elegant 
proof of Glivenko's theorem 
for both intuitionistic and minimal logic. 
We obtain these results through a completely proof-theoretic approach. 
The three main reasons for the interest of this approach are:

\begin{enumerate}
  \item We point out that the postponement of $\Raa$ not only is an interesting result in itself, with remarkable consequences such as weak normalization and Glivenko theorems (as first observed by Seldin \cite{Seldin86}), but it allows them to be proved in a \emph{uniform} way in, at least, a triple sense:
  \begin{enumerate}
    \item we prove the postponement of $\Raa$ and its consequences for first-order classical logic, but our methods and techniques can also be applied in other systems, such as second-order classical logic (see point \ref{point:second-order} below), and the modal logics S4 and S5 with quantifiers;
    \item the postponement of $\Raa$ allows derivation of Glivenko's theorem for both intuitionistic and minimal logic using the same proof-theoretic approach based on our reduction steps; and
    \item the reduction steps we have defined are essentially variants of the ones used by Seldin, Prawitz and others (see \S\ref{sect:overview}) to prove  weak normalization for classical natural deduction.
  \end{enumerate}

  \item\label{point:second-order} Our proof of the postponement of $\Raa$ is proof-theoretic in a ``geometric'' way, in the sense that it relies on a notion of size for a derivation based only on the distance of the instances of $\Raa$ from the conclusion of the derivation; the complexity of formulas play no role in this definition of size.
  This approach has two immediate consequences:
  \begin{enumerate}
    \item it allows us to generalize the postponement of $\Raa$ and its corollaries to second-order classical logic, since the substitution of formulas for propositional variables does not impact the size we have defined;
    \item we prove the postponement of $\Raa$ in a weak form, in the sense that if we apply our reduction steps to \emph{suitable} instances of $\Raa$, then the size of the derivation decreases; but we strongly conjecture that, by refining this notion of size, our reduction steps allow the postponement of $\Raa$ in a strong sense, \Ie~they can be applied following whatever strategy one likes.
  \end{enumerate}

  \item Besides the atomization of $\Raa$ proposed by Prawitz's weak normalization strategy for classical natural deduction, which is deeply related to the Gödel--Gentzen negative translation, we show that the postponement of $\Raa$ induces (a variant of) Kuroda's negative translation
  .
\end{enumerate}

\subparagraph{Outline}
The article will be structured into two parts. In the first part, namely \S\ref{sect:overview}, we will present a survey of the relevant literature concerning the (weak) normalization strategies for classical natural deduction. 
The second part presents our technical contributions.
More precisely, \S\ref{sect:preliminairies} is devoted to basic definitions for first-order language and natural deduction.
In \S\ref{sect:rewrite} we will introduce our reduction steps, and in \S\ref{sect:postponement} we will use them to prove the postponement of $\Raa$.
Finally, in \S\ref{sect:Glivenko}, we will expose the relationship with Kuroda's translation, and the Glivenko theorems for intuitionistic and minimal logic.


%
\section{\texorpdfstring{Normalization of classical logic: an overview}{Normalization of classical logic: an overview}}
\label{sect:overview}

Looking closer at the distinction between Prawitz's and Seldin's weak normalization strategies, we note that actually this distinction is not as sharp as it might appear at a first glance. In particular, each of these strategies can be exploited to eliminate classical detours by postponing the use of $\Raa$. Let us clarify this point.

\subsection{Prawitz's (weak) normalization strategy: its legacy (1965-2012)}

As already mentioned, Prawitz's original weak normalization strategy for classical natural deduction \cite{Prawitz65} was conceived only for the fragment $\{\neg, \land, \allowbreak \to, \bot, \forall\}$, which is adequate for the full first-order language $\mathcal{L}$ of classical logic.%
\footnote{In Prawitz \cite{Prawitz65}, as well as in the other approaches that will be analyzed in this section, negation $\neg$ is not treated as a primitive operator, but is defined by 
$\neg A \Defeq A \to \bot$, and its introduction and elimination rules are special cases of the introduction and elimination rules for implication. As we will see in \S\ref{sect:preliminairies}, our approach is different.} 
One of the challenges that arose from Prawitz's work was to find a weak normalization strategy for the full language $\mathcal{L}$. 
Specifically, the idea 
was to prove this result by somehow relaxing Prawitz's strategy: instead of proving the atomization of all the $\Raa$ occurrences present in a given derivation, it is sufficient to eliminate all classical detours, where a \emph{classical detour} is defined as an instance of the $\Raa$ rule that introduces a formula occurrence $A$ and is immediately followed by an instance of an elimination rule having $A$ as major premiss. 
In this paper, such detours will be called classical detours \textit{\`a la} Prawitz, since they 
correspond to the definition of maximum formula given by Prawitz 
in \cite[p. 34]{Prawitz65}.

In order to better appreciate this departure from Prawitz's original strategy, we give an example. Prawitz observes that the following derivation contains a complex instance of $\Raa$

\begin{equation}
\label{conjpraw}
	\begin{prooftree}
		\Hypo{\ulcorner \neg (A \wedge B)\urcorner^{1}}
		\Ellipsis{$\pi$}{\;\bot\;}
		\Infer1[\footnotesize$\Raa^{1}$]{A \wedge B}
		 \Ellipsis{$\pi_{0}$}{}
	\end{prooftree}
\end{equation}

\noindent and so needs to be reduced to a derivation with less complex instances of $\Raa$, \Ie

\begin{equation}
\label{conjprawred}
\begin{prooftree}
       \Hypo{\ulcorner \lnot A \urcorner^2}
       \Hypo{\ulcorner A \land B \urcorner^1}
       \Infer1[\footnotesize$\land_{\ElimOne}$]{A}
       \Infer2[\footnotesize$\lnot_\Elim$]{\;\bot\;}
       \Infer1[\footnotesize$\lnot_\Intro^1$]{\lnot (A \land B)}
       \Ellipsis{$\pi$}{\; \bot \;}
       \Infer1[\footnotesize$\Raa^2$]{\;A\;}
       \Hypo{\ulcorner \lnot B \urcorner^4}
       \Hypo{\ulcorner A \land B \urcorner^3}
       \Infer1[\footnotesize$\land_{\ElimTwo}$]{B}
       \Infer2[\footnotesize$\lnot_\Elim$]{\;\bot\;}
       \Infer1[\footnotesize$\lnot_\Intro^3$]{\lnot (A \land B)}
       \Ellipsis{$\pi$}{\; \bot \;}
       \Infer1[\footnotesize$\Raa^4$]{\;B\;}
       \Infer2[\footnotesize$\wedge_\Intro$]{A \wedge B}
       \Ellipsis{$\pi_{0}$}{}
 \end{prooftree}
\end{equation}

But is it necessary to reduce all complex instances of $\Raa$ (\Ie~the instances of $\Raa$ whose conclusions are non-atomic formulas)? Maybe it would be enough to focus only on a particular 
subset of them:
the ones forming classical detours \textit{\`a la} Prawitz.
The reason is that only in this situation we are creating a rules' configuration which is similar to the standard (intuitionistic) detours of the form $\circ$-introduction/$\circ$-elimination (for a certain connective $\circ$)
: the rule $\Raa$ 
plays the role of an introduction rule. Thus, instead of (\ref{conjpraw}), we could consider

\begin{equation}
\label{conjex}
	\begin{prooftree}
		\Hypo{\ulcorner \neg (A \wedge B)\urcorner^{1}}
		\Ellipsis{$\pi$}{\;\bot\;}
		\Infer1[\footnotesize$\Raa^{1}$]{A \wedge B}
		\Infer1[\footnotesize$\wedge_{\ElimOne}$]{A}
		 \Ellipsis{$\pi_{0}$}{}
	\end{prooftree}
\end{equation}

\noindent and reduce it to

\begin{equation}
\label{conjexred}
 \begin{prooftree}
       \Hypo{\ulcorner \lnot A \urcorner^2}
       \Hypo{\ulcorner A \land B \urcorner^1}
       \Infer1[\footnotesize$\land_{\ElimOne}$]{A}
       \Infer2[\footnotesize$\lnot_\Elim$]{\;\bot\;}
       \Infer1[\footnotesize$\lnot_\Intro^1$]{\lnot (A \land B)}
       \Ellipsis{$\pi$}{\; \bot \;}
       \Infer1[\footnotesize$\Raa^2$]{\;A\;}
       \Ellipsis{$\pi_{0}$}{}
 \end{prooftree}
\end{equation}

It is worth noticing that (\ref{conjexred}) is nothing but a subderivation of (\ref{conjprawred}). However, adopting this second kind of reduction does not mean that we are simply applying a special case of Prawitz's original strategy. There is indeed a crucial difference between the two reductions presented here.
In the reduced derivation (\ref{conjprawred}), the distance from the conclusion of the two instances of $\Raa$ is one unit greater than the distance from the conclusion of the instance of $\Raa$ which is present in the original derivation (\ref{conjpraw}). 
In contrast, in the reduced derivation (\ref{conjexred}), the distance from the conclusion of the instance of $\Raa$ is one unit smaller that the distance from the conclusion of the instance of $\Raa$ in the original derivation (\ref{conjex}). 
In other words, where Prawitz's original strategy (\ref{conjpraw})-(\ref{conjprawred}) brings the application of $\Raa$ forward, the second strategy (\ref{conjex})-(\ref{conjexred}) postpones the application of $\Raa$.

\subsubsection{Statman's approach (1974)}

As far as we know, the first to take into consideration this postponing strategy for solving the normalization problem of classical natural deduction for the \emph{full} first-order language was Statman \cite{Statman74}. He considered, in particular, the following general reduction scheme for classical detours \textit{\`a la} Prawitz:

{\small
\begin{equation}
\label{statmanred}
\begin{prooftree}
	\Hypo{\ulcorner \lnot A \urcorner^1}
	\Ellipsis{$\pi$}{\; \bot \;}
	\Infer1[\scriptsize$\Raa^1$]{A}
	 \Hypo{}
	 \Ellipsis{$\pi'$}{ }
	\Infer2[\scriptsize$\circ_\Elim$]{C}
	\Ellipsis{$\pi_0$}{}
\end{prooftree}
\qquad
\begin{textnormal}{\normalsize reduces to}\end{textnormal}
\qquad
\begin{prooftree}
	  \Hypo{\ulcorner \lnot C \urcorner^2}
	   \Hypo{\ulcorner A \urcorner^1}
	   \Hypo{}
	   \Ellipsis{$\pi'$}{}
	  \Infer2[\scriptsize$\circ_\Elim$]{C}
	  \Infer2[\scriptsize$\lnot_\Elim$]{\bot}
	  \Infer1[\scriptsize$\lnot_\Intro^1$]{\lnot A}
	  \Ellipsis{$\pi$}{\; \bot \;}
	  \Infer1[\scriptsize$\Raa^2$]{C}
	\Ellipsis{$\pi_0$}{}
\end{prooftree}
\end{equation}
}

  \noindent where $\circ_\Elim$ is an elimination rule for any connective $\circ$ of the full language of first-order logic (see \cite[pp.~78--79]{Statman74}) and $A$ is the major premiss of $\circ_\Elim$.%
  \footnote{The fact that, in the derivation $\Pi$ on the left-hand side of \eqref{statmanred}, $A$ is the major premiss of $\circ_\Elim$ ensures that $\circ_\Elim$ does not discharge any assumptions of the subderivation $\pi$, hence all assumptions in the derivation on the right-hand side of \eqref{statmanred} are already assumptions of $\Pi$. See also \S\ref{subsect:seldinstrategy} below.}
  This means that, unlike Prawitz, Statman does not have to drop disjunction and the existential quantifier in order to apply his reduction steps. 
  However, instead of reasoning in a pure combinatorial way on the application of the reduction steps -- as previously done by Prawitz -- Statman adopts a different approach, which basically consists of two phases. 
First, second-order predicate classical logic is considered and embedded into a system of second-order intuitionistic propositional logic by using a homomorphism which preserves the reduction relations. Secondly, a (strong) normalization theorem for this intuitionistic system is proved. To prove this theorem, impredicative methods are used. 

Besides the use of impredicative methods, there is 
another problem which could make Statman's approach not completely satisfactory: 
the $\lor$- and $\exists$-elimination rules can only be applied in the restricted case in which their conclusion is $\bot$ (see \cite[p.~90]{Statman74}), \Ie

\begin{center}
\begin{prooftree}
	\Hypo{A \lor B}
	\Hypo{\ulcorner A \urcorner^1}
	\Ellipsis{}{\bot}
	\Hypo{\ulcorner B \urcorner^1}
	\Ellipsis{}{\bot}
	\Infer3[\footnotesize${\lor'}_\Elim^{1}$]{\bot}
\end{prooftree}
\qquad{\normalsize and }\qquad
\begin{prooftree}
	\Hypo{\exists x A}
	\Hypo{\ulcorner A \urcorner^1}
	\Ellipsis{}{\bot}
	\Infer2[\footnotesize${\exists'}_\Elim^{1}$]{\bot}
\end{prooftree}
\end{center}

\noindent Statman's solution seems then to lack of generality and uniformity.

Some arguments can be invoked to justify Statman's choice to work with these restricted versions of the $\lor$- and $\exists$-elimination rules. 
For simplicity, let us restrict to the case of disjunction. 
First, it should be noted that, within the framework of classical logic, the usual $\lor$-elimination rule -- \Ie~$\lor_\Elim$ -- and the restricted version used by Statman -- \Ie~$\lor'_\Elim$ -- are equivalent from the point of view of derivability. 
Indeed, $\lor'_\Elim$ is just a special case of $\lor_\Elim$, 
and conversely, $\lor_\Elim$ is classically derivable from $\lor_\Elim'$:
\label{stator} 
\begin{center}
\begin{prooftree}
	\Hypo{A \lor B}
	\Hypo{\ulcorner \neg C \urcorner^2}
	\Hypo{\ulcorner A \urcorner^1}
	\Ellipsis{}{C}
	\Infer2[\footnotesize$\neg_\Elim$]{\bot}
	\Hypo{\ulcorner \neg C \urcorner^2}
	\Hypo{\ulcorner B \urcorner^1}
	\Ellipsis{}{C}
	\Infer2[\footnotesize$\neg_\Elim$]{\bot}
	\Infer3[\footnotesize${\lor'_\Elim}^{1}$]{\bot}
	\Infer1[\footnotesize$\Raa^2$]{C}
\end{prooftree}
\end{center}

The restriction imposed on the $\lor$-elimination rule can also be explained by a second argument. 
In his proof of (strong) normalization for classical natural deduction, Statman passes through an intermediate step: he shows that the set of classical derivations can be embedded into the subset of classical derivations not containing $\lor$ and $\exists$, and that this embedding preserves the reduction relations between derivations. 
This means that, unlike Prawitz, Statman does not start by replacing the full first-order language with the fragment $\{\neg, \wedge, \to, \bot, \forall\}$ 
(which is still adequate for classical logic).
In other words, Statman is not reasoning at the level of the derivability relation between sentences, but at the level of the reduction relations between derivations. Thus, strictly speaking, he maintains the full language of classical logic and uses an operation $m_{0}$ -- corresponding to Gentzen's negative translation -- just to narrow down the set of derivations which have to be analyzed with respect to the reduction relations.

In such a context, disjunction can be defined in terms of negation and conjunction, so that $m_{0}(A \lor B) \Defeq \neg (\neg m_{0}(A) \land \neg m_{0}(B))$.\footnote{Note that Statman composes then the operation $m_{0}$ with another operation $m_{1}$ -- corresponding to the so-called Russell--Prawitz translation \cite[p.~67]{Prawitz65} -- in order to further narrow down his analysis to the set of derivations and reductions which use only implication and the second-order universal quantifier (see \cite[p.~91--92]{Statman74}).} 
And if we have two derivations of the form
\begin{center}
\begin{prooftree}
	\Hypo{m_{0}(A)}
	\Ellipsis{}{\bot}
\end{prooftree}
\qquad
{\normalsize and}
\qquad
\begin{prooftree}
	\Hypo{m_{0}(B)}
	\Ellipsis{}{\bot}
\end{prooftree}
\end{center}
we could obtain $\neg m_{0}(A) \land \neg m_{0}(B)$ (namely, by discharging $m_{0}(A)$ and $m_{0}(B)$ through $\neg$-introduction rules, and applying then a $\land$-introduction), so that in presence of $\neg (\neg m_{0}(A) \land \neg m_{0}(B))$ -- which, as we said, is equal to $m_{0}(A \lor B)$ -- we can conclude $\bot$.\label{statrans} 
This shows that the rule 
\begin{center}
\begin{prooftree}
	\Hypo{m_{0}(A \lor B)}
	\Hypo{\ulcorner m_{0}(A)\urcorner^1}
	\Ellipsis{}{\bot}
	\Hypo{\ulcorner m_{0}(B)\urcorner^1}
	\Ellipsis{}{\bot}
	\Infer3[\footnotesize$S^{1}$]{\bot}
\end{prooftree}
\end{center}
is derivable. And this is nothing but a way to say that working with a restricted version of the $\lor$-elimination rule -- where $\bot$ is the conclusion -- does not represent a real limitation if one works modulo $m_{0}$: the rule $S$ and the rule $\lor'$ share the same inferential structure. In other words, the operation $m_{0}$ is an homomorphism which preserves the inferential structure.

\subsubsection{St{\aa}lmarck's approach (1991)}

Being conceived for second-order logic, Statman's approach rests on an impredicative analysis of derivations and of their reduction relations. An alternative approach 
to the (weak) normalization of classical natural deduction for the full first-order language using only pure combinatorial means was proposed by St{\aa}lmarck \cite{Stalmarck91}.

Like Statman, St{\aa}lmarck's fundamental idea consists of removing the maximum formulas which appear in classical detours \textit{\`a la} Prawitz by pushing $\Raa$ downwards with respect to the elimination rules for the connectives of the full first-order language
. But, unlike Statman, St{\aa}lmarck has no restrictions on $\lor$- and $\exists$-elimination rules.%
\footnote{Like Statman, Stålmarck's main interest is a strong normalization theorem for classical logic. 
Since we mainly focuses on the problem of the downward postponement of $\Raa$, it is 
sufficient to consider here only St{\aa}lmarck's weak normalization strategy \cite[pp.~133--135]{Stalmarck91}.} 
The problem for St{\aa}lmarck is that, by using the reduction scheme (\ref{statmanred}), when $\circ$ is a disjunction or an existential quantifier, the formula $C$ introduced by $\Raa$ in the reduced derivation could be more complex than the formula $A$ introduced by $\Raa$ in the original derivation. 
Thus, \textit{a priori} there is nothing which tells him when to stop the reduction of the classical detours. 
In particular, in order to prove the weak normalization of classical natural deduction for the full first-order language, 
St{\aa}lmarck cannot proceed like Prawitz in two phases, that is, first, by atomizing all the occurrences of $\Raa$, and secondly, by applying any normalization strategy for intuitionistic logic. On the contrary, St{\aa}lmarck has to treat classical and intuitionistic detours on a par, and he has to eliminate them together by (a main) induction on the complexity of the maximum formulas of the detours.

The problem is that this method could conflict with certain applications of the reduction scheme (\ref{statmanred}). 
Consider the case of conjunction. In (\ref{conjexred}), the formula $\neg (A \land B)$ introduced by a $\lnot_\Intro$ rule could be the major premiss of a $\lnot_\Elim$ rule; it could thus create a new detour with a maximum formula more complex than the original one (namely $A \land B$, associated with the classical detour). Indeed, in this case the derivation \eqref{conjexred} has the form

{\small
\begin{center}
 \begin{prooftree}
       \Hypo{\ulcorner \lnot A \urcorner^2}
       \Hypo{\ulcorner A \land B \urcorner^1}
       \Infer1[\footnotesize$\land_{\ElimOne}$]{A}
       \Infer2[\footnotesize$\lnot_\Elim$]{\;\bot\;}
       \Infer1[\footnotesize$\lnot_\Intro^1$]{\lnot (A \land B)}
       \Hypo{}
       \Ellipsis{$\pi_1$}{A \land B}
       \Infer2[\footnotesize$\lnot_\Elim$]{\;\bot\;}
       \Ellipsis{$\pi_2$}{\; \bot \;}
       \Infer1[\footnotesize$\Raa^2$]{\;A\;}
       \Ellipsis{$\pi_{0}$}{}
 \end{prooftree}
\end{center}
}

\noindent and, according to the reduction scheme \eqref{statmanred}, 
the derivation we have before the application of the reduction step 
is the following, where the maximum formula of the classical detour is $A \land B$:

{\small
\begin{equation}
\label{intdetconj}
\begin{prooftree}
	\Hypo{\ulcorner \lnot (A \land B) \urcorner^1}
	\Hypo{}
	\Ellipsis{$\pi_1$}{A \land B}
	\Infer2[\footnotesize$\lnot_\Elim$]{\; \bot \;}
	\Ellipsis{$\pi_2$}{\;\bot\;}
	 \Infer1[\footnotesize$\Raa^1$]{A \land B}
	 \Infer1[\footnotesize$\land_{\ElimOne}$]{A}
	 \Ellipsis{$\pi_0$}{}
\end{prooftree} 
\end{equation}
}

St{\aa}lmarck's solution consists 
in adding to \eqref{statmanred} a new reduction scheme which transforms (\ref{intdetconj}) into

{\small
\begin{center}
\begin{prooftree}
	\Hypo{\ulcorner \lnot A \urcorner^1}
	\Hypo{\ulcorner \lnot A \urcorner^1}
	\Hypo{}
	\Ellipsis{$\pi_1$}{A \land B}
	\Infer1[\footnotesize$\land_{\ElimOne}$]{A}
	\Infer2[\footnotesize$\lnot_\Elim$]{\; \bot \;}
	\Ellipsis{$\pi_2$}{\;\bot\;}
	\Infer1[\footnotesize$\Efq$]{A \land B}
	\Infer1[\footnotesize$\land_{\ElimOne}$]{A}
	\Infer2[\footnotesize$\lnot_\Elim$]{\; \bot \;}
	\Infer1[\footnotesize$\Raa^1$]{\;A\;}
	\Ellipsis{$\pi_0$}{}
\end{prooftree} 
\end{center}
}

In this way, the occurrence of $\Raa$ is not only pushed downward, but also applied to a formula of lower complexity than the original one, respecting in this sense Prawitz's original idea. 
%
However, unlike Prawitz's atomization procedure, in order to define this reduction, a step of $\Efq$ has to be added. This represents an important point on which we will return later (see \S\ref{subsect:seldinstrategy}).

The other difference to Prawitz's procedure is that, by considering the full first-order language, St{\aa}lmarck also has to define reduction steps for classical detours \textit{\`a la} Prawitz created by a $\Raa$ immediately followed by a $\vee$- or $\exists$-elimination rule. Let us focus on the $\lor$ case.

As we have seen for conjunction, 
if we follow the reduction scheme \eqref{statmanred} to define a postponement of $\Raa$ with respect to $\lor_\Elim$, we would have that
 
{\footnotesize
\begin{equation}
\label{stalredisj1}
\begin{prooftree}[label separation = 0.3em]
	\Hypo{\ulcorner \neg(A \vee B) \urcorner^1}
	\Ellipsis{$\pi$}{\bot}
	\Infer1[\scriptsize$\Raa^1$]{A \vee B}
	\Hypo{\ulcorner A\urcorner^2}
	\Ellipsis{$\pi'$}{C}
	\Hypo{\ulcorner B\urcorner^2}
	\Ellipsis{$\pi''$}{C}
	\Infer[separation = 0.9em]3[\scriptsize$\lor_\Elim^{2}$]{C}
	\Ellipsis{$\pi_{0}$}{}
\end{prooftree}
\qquad
\begin{textnormal}{\small reduces to}\end{textnormal}
\qquad
\begin{prooftree}[label separation = 0.3em]
	\Hypo{\ulcorner \neg C\urcorner^3}
	\Hypo{\ulcorner A \vee B\urcorner^2}
	\Hypo{\ulcorner A\urcorner^1}
	\Ellipsis{$\pi'$}{C}
	\Hypo{\ulcorner B\urcorner^1}
	\Ellipsis{$\pi''$}{C}
	\Infer[separation = 0.9em]3[\scriptsize$\vee_\Elim^{1}$]{C}
	\Infer[separation = 0.5em]2[\scriptsize$\neg_\Elim$]{\bot}
	\Infer1[\scriptsize$\neg_\Intro^2$]{\neg (A \vee B)}
	\Ellipsis{$\pi$}{\,\bot\,}
	\Infer1[\scriptsize$\Raa^3$]{C}
	\Ellipsis{$\pi_{0}$}{}
\end{prooftree}
\end{equation}
}\label{stalred}

\noindent and thus, when in the subderivation $\pi$ the assumption $\lnot(A \lor B)$ is the major premiss of a $\lnot_\Elim$, we could create a new detour 
$\lnot_\Intro$/$\lnot_\Elim$, having $\neg (A \lor B)$ as maximum formula, which is more complex than the maximum formula eliminated by the reduction step itself, namely $A \lor B$. 
However, when we are dealing with disjunction, the solution proposed by St{\aa}lmarck for conjunction is no longer sufficient to circumvent this problem. More precisely, the derivation

{\footnotesize
\begin{equation}
\label{stalredisj2}
\begin{prooftree}[label separation = 0.2em, separation = 0.6em]
	\Hypo{\ulcorner \lnot (A \lor B) \urcorner^1}
	\Hypo{}
	\Ellipsis{$\pi_1$}{A \lor B}
	\Infer2[\scriptsize$\lnot_\Elim$]{\; \bot \;}
	\Ellipsis{$\pi_2$}{\;\bot\;}
	 \Infer1[\scriptsize$\Raa^1$]{A \lor B}
	 \Hypo{\!\ulcorner A \urcorner^2}
	 \Ellipsis{$\pi'$}{C}
	 \Hypo{\,\ulcorner B \urcorner^2}
	 \Ellipsis{$\pi''$}{C}
	 \Infer3[\scriptsize$\lor_{\Elim}^{2}$]{C}
	 \Ellipsis{$\pi_0$}{}
\end{prooftree} 
\ \ 
\begin{textnormal}{\small would reduce to}\end{textnormal}
\ 
\begin{prooftree}[label separation = 0.2em, separation = 0.6em]
	\Hypo{\ulcorner \lnot C \urcorner^2}
	\Hypo{}
	\Ellipsis{$\pi_1$}{A \lor B}
	 \Hypo{\ulcorner A \urcorner^1}
	 \Ellipsis{$\pi'$}{C}
	 \Hypo{\,\ulcorner B \urcorner^1}
	 \Ellipsis{$\pi''$}{C}
	  \Infer3[\scriptsize$\lor_{\Elim}^{1}$]{C}	  
	  \Infer2[\scriptsize$\lnot_\Elim$]{\;\bot\;}
	  \Ellipsis{$\pi_2$}{\;\bot\;}
	  \Infer1[\scriptsize$\Raa^2$]{C}
	   \Ellipsis{$\pi_0$}{}
\end{prooftree}
\end{equation}
}

\noindent but since $C$ could be any compound formula, it could create a new classical detour (when $C$ in $\pi_0$ is the major premiss of an elimination rule) more complex than the one just eliminated.\footnote{Note also that in this reduction for $\lor_\Elim$, differently from the $\land$ case, the appeal to the $\Efq$ rule would be pointless. If $\Efq$ was used after the $\bot$ of the $\pi$ derivation in order to obtain $A \lor B$, then another step of $\lor_\Elim$ would be needed. This would force one to conclude $C$, and thus to use $\neg_\Elim$ again\,---\,with $\neg C$ as major premiss\,---\,followed by the $\pi$ derivation. Hence, the result would be nothing but a repeated duplication of the original reduction pattern.}



In order to solve this problem, St{\aa}lmarck decomposes the reduction for the $\lor$ case into two steps. 
First, he starts with a sort of expansion of the $\lor_\Elim$ rules present in a given derivation, where $\bot$ plays the role of a minimum formula in the two threads starting from the assumptions $A$ and $B$, respectively, and the end-formula $C$ (when $C \neq \bot$) is introduced by $\Raa$,\footnote{We use here the term `thread' in the sense of Prawitz \cite[p.~25]{Prawitz65} (see also Definition~\ref{def:distance}, \textit{infra}). For the definition of expansion and minimum formula see Prawitz (\cite[part II.3]{Prawitz71}).}\label{threadef} \Ie~the derivation

{\footnotesize
\begin{equation*}
\begin{prooftree}[label separation = 0.2em, separation = 0.3em]
	\Hypo{\ulcorner \!\lnot (A \lor B) \urcorner^1}
	\Hypo{}
	\Ellipsis{$\pi_1$}{A \lor B}
	\Infer2[\footnotesize$\lnot_\Elim$]{\; \bot \;}
	\Ellipsis{$\pi_2$}{\;\bot\;}
	 \Infer1[\footnotesize$\Raa^1$]{A \lor B}
	 \Hypo{\ulcorner A \urcorner^2}
	 \Ellipsis{$\pi'$}{C}
	 \Hypo{\,\ulcorner B \urcorner^2}
	 \Ellipsis{$\pi''$}{C}
	 \Infer3[\footnotesize$\lor_{\Elim}^{2}$]{C}
	 \Ellipsis{$\pi_0$}{}
\end{prooftree} 
\quad\!\!
\textnormal{\small{reduces to}}
\!\!\!\!\!\!\!\!
%
\begin{prooftree}[label separation =0.2em, separation = 0.3em]
	\Hypo{\ulcorner \!\lnot (A \lor B) \urcorner^1}
	\Hypo{}
	\Ellipsis{$\pi_1$}{A \lor B}
	\Infer2[\scriptsize$\lnot_\Elim$]{\; \bot \;}
	\Ellipsis{$\pi_2$}{\;\bot\;}
	 \Infer1[\scriptsize$\Raa^1$]{A \lor B}
	  \Hypo{\ulcorner \lnot C \urcorner^3}
	    \Hypo{\ulcorner A \urcorner^2}
	  \Ellipsis{$\pi'$}{C}
	\Infer[separation = -0.1em]2[\scriptsize$\lnot_\Elim$\qquad ]{\,\bot\,}
	  \Hypo{\ulcorner \lnot C \urcorner^3}
	   \Hypo{\ulcorner B \urcorner^2}
	  \Ellipsis{$\pi''$}{C}
	\Infer[separation = -0.1em]2[\scriptsize$\lnot_\Elim$]{\;\bot\;}
	\Infer[separation = -1.3em]3[\scriptsize$\lor_{\Elim}^{2}$]{\;\bot\;}
	\Infer1[\scriptsize$\Raa^3$]{C}
	\Ellipsis{$\pi_0$}{}
\end{prooftree}
\end{equation*}
}


Second, the original reduction procedure is applied, following the complexity of the maximum formulas of the detours (namely, starting with $C$). 
This is possible since the reduction of a $\lor$ classical detour is now operated only when the conclusion of the $\lor_\Elim$ is $\bot$, so no new 
more complex detours are created.\footnote{The reason is that the burden of the creation of new (possibly) more complex detours rests now on the preliminary step corresponding to the ``expansion'' of the $\lor_\Elim$ rule described above.} When it comes the turn of the $\vee$ classical detour, we have that

{\footnotesize
\begin{equation}
\label{stalsol}
\begin{prooftree}[label separation = 0.2em, separation = 0.7em]
	\Hypo{\ulcorner \lnot (A \lor B) \urcorner^1}
	\Hypo{}
	\Ellipsis{$\pi_1$}{A \lor B}
	\Infer2[\scriptsize$\lnot_\Elim$]{\; \bot \;}
	\Ellipsis{$\pi_2$}{\;\bot\;}
	 \Infer1[\scriptsize$\Raa^1$]{A \lor B}
	 \Hypo{\ulcorner A \urcorner^2}
	 \Ellipsis{$\pi'$}{\bot}
	 \Hypo{\,\ulcorner B \urcorner^2}
	 \Ellipsis{$\pi''$}{\bot}
	 \Infer[separation = 0.6em]3[\scriptsize$\lor_{\ElimOne}^{2}$]{\bot}
	 \Ellipsis{$\pi_0$}{}
\end{prooftree} 
\ 
\begin{textnormal}{\small reduces to\!\!\!\!\!\!}\end{textnormal}
\begin{prooftree}[label separation = 0.2em, separation = 0.7em]
	\Hypo{}
	\Ellipsis{$\pi_1$}{A \lor B}
	 \Hypo{\ulcorner A \urcorner^1}
	 \Ellipsis{$\pi'$}{\;\bot\;}
	 \Hypo{\ulcorner B \urcorner^1}
	 \Ellipsis{$\pi''$}{\;\bot\;}
	 \Infer3[\scriptsize$\lor_{\ElimOne}^{1}$]{\;\bot\;}
	 \Ellipsis{$\pi_2$}{\;\bot\;}
	 \Infer1[\scriptsize$\Efq$]{A \lor B}
	 \Hypo{\ulcorner  A \urcorner^2}
	 \Ellipsis{$\pi'$}{\bot}
	 \Hypo{\,\ulcorner B \urcorner^2}
	 \Ellipsis{$\pi''$}{\bot}
	 \Infer3[\scriptsize$\land_{\ElimOne}^{2}$]{\;\bot\;}
	 \Ellipsis{$\pi_0$}{}
\end{prooftree}
\end{equation}
}

Note that St{\aa}lmarck's first step for the reduction of the $\vee$ case can also be described as the application of a \textit{permutative conversion} between the $\lor_\Elim$ and the $\neg_\Elim$ rules in the reduced derivation of (\ref{stalredisj2}), as remarked by de Groote (\cite[p.~184]{deGroote01}).\footnote{See also \cite[p.\,136]{Stalmarck91}. The 
definition of permutative conversions (or reductions) is given in \cite[p.\,51]{Prawitz65} and \cite[p.\,253]{Prawitz71}.} It is for this reason that we can say that St{\aa}lmarck works with the standard set of inference rules for classical logic. More precisely, unlike Statman\,---\,who works with a system of classical logic where the rule $\lor_\Elim$ is replaced by the rule $\lor'_\Elim$\,---\,St{\aa}lmarck keeps defining the system of classical rules using the standard $\lor_\Elim$. In particular, since the permutative conversions are taken into account when a $\lor_\Elim$ is involved in a classical detours, this rule is used in a restricted way only with respect to the reduction procedure for classical detours, but not with respect to the simple construction of a derivation. What is lost, then, is the uniformity of the reduction procedure for classical detours, even if there is no loss from the point of view of derivability.

\subsubsection{\label{vonPlatoapproach}von Plato and Siders' approach (2012)}

A uniform solution to the weak normalization problem of classical natural deduction for the full first-order language\,---\,\Ie~a solution imposing no sort of restriction to the application of the inference rules of the full first-order classical natural deduction\,---\,was finally given by von Plato and Siders \cite{vonPlatoSiders12}.%
\footnote{In \cite{vonPlatoSiders12}, von Plato and Siders 
propose only a weak normalization theory for classical logic. A strong normalization strategy is suggested in \cite{vonPlato16}.}

They consider the same notion of classical detour \textit{\`a la} Prawitz as St{\aa}lmarck, but they use a system of natural deduction for classical logic where all the elimination rules are presented in a general form as done in \cite{vonPlato01}. 
For example, the elimination rules for conjunction and implication become respectively
%
%
\begin{center}
\begin{prooftree}
	\Hypo{A \wedge B}
	\Hypo{\ulcorner A,B \urcorner^{1}}
	\Ellipsis{$\pi$}{C}
	\Infer2[\footnotesize$\wedge_{\Elim}^{1}$]{C}
\end{prooftree}
\qquad\textup{ and }\qquad
\begin{prooftree}
	\Hypo{A \to B}
	\Hypo{A}
	\Hypo{\ulcorner B \urcorner^{1}}
	\Ellipsis{$\pi$}{C}
	\Infer3[\footnotesize$\to_{\Elim}^{1}$]{C}
\end{prooftree}
\end{center}
\label{genelim}

Using these general elimination rules, von Plato and Siders are able to define the postponement of $\Raa$ (for classical detours \textit{\`a la} Prawitz) in an uniform way, 
 following the general reduction scheme (\ref{statmanred}).

%

In this way, unlike St{\aa}lmarck's reductions, no use of $\Efq$ is needed. Moreover, Stålmarck's treatment of $\vee_\Elim$ through a permutative conversion with a $\neg_\Elim$ can be extended to all the other elimination rules, thanks to their general form. 
Besides, unlike St{\aa}lmarck, the appeal to permutative conversions does not represent a necessary step, since there is no need for reasoning over the complexity of the maximum formulas of classical detours. 
The reason is that, by using general elimination rules, a derivation in classical logic can be considered to be normal when all major premisses of the elimination rules are assumptions. In order to obtain this normal form, it is sufficient to apply, first, the reductions associated to the classical detours -- as described in (\ref{statmanred}) -- and then the reductions associated to the intuitionistic detours. In this way Prawitz's original strategy is respected, even if the reductions of classical detours are not necessarily associated with the atomization of the applications of $\Raa$ rules. 
Thus, as in Prawitz (see \cite[p.~41]{Prawitz65} and \cite[part. II.3.2]{Prawitz71}), a classical proof in normal form becomes a derivation essentially composed of two parts:%
\footnote{Note that, by restricting to the fragment $\{\neg, \wedge, \to, \bot, \forall\}$, Prawitz's notion of classical normal derivation can be defined with respect to the notion of branch (see \cite[p.~52]{Prawitz65}). 
On the other hand, working with the full language and using general elimination rules obliges Siders and von Plato to define the notion of classical normal derivation with respect to another notion, that of thread (see \cite[p.~208]{vonPlatoSiders12}). 
The reader has to pay attention that von Plato and Siders' definition of thread is different from Prawitz's one (already mentioned here at p.~\pageref{threadef}). More precisely, von Plato and Siders' notion of thread is conceived as a sort of generalization of Prawitz's notion of path, allowing one to go through a derivation by jumping from the major premiss of a general elimination rule to the assumptions of the minor premisses of the same rule. For more detail see Negri and von Plato (\cite[p.~196 et sqq.]{NegrivonPlato01} and \cite[p.~26--27]{NegrivonPlato11}).}

\begin{quote}
\begin{small}
[\ldots] from the endformula upward, there will be a sequence of $I$-rules and their nested premisses [\ldots] until a conclusion of a rule \textit{DN} [\textit{i.e.}~$\Raa$] is reached. Its premiss is $\bot$. Looking from the other direction, from top formulas downward, we find a nested sequence of major premisses of $E$-rules. [\ldots] the presence of rule \textit{DN} can force conclusions of $E$-rules in normal derivations to be equal to the premiss $\bot$, without any \textit{a priori} requirement that this should be so. (\cite[p.~208]{vonPlatoSiders12})
\end{small}
\end{quote}

Therefore, in a normal derivation the application of the $\Raa$ rule is what separates one of the parts containing the elimination rules from the part containing the introduction rules. Schematically, this takes the form:

\begin{center}
  \begin{prooftree}
    \Hypo{\ulcorner \lnot A_1 \urcorner^1}
    \Hypo{}
    \Ellipsis{$\pi_1$}{A_1}
    \Infer2[\footnotesize$\lnot_\Elim$]{\bot}
    \Ellipsis{\footnotesize $E$-rules}{\,\bot\,}
    \Infer1[\footnotesize$\Raa^1$]{A_1}
    \Infer[no rule,rule margin=0pt]1{\qquad\ddots}
    \Hypo{\ulcorner \lnot A_n \urcorner^n}
    \Hypo{}
    \Ellipsis{$\pi_n$}{A_n}
    \Infer2[\footnotesize$\lnot_\Elim$]{\bot}
    \Ellipsis{\footnotesize $E$-rules}{\,\bot\,}
    \Infer1[\footnotesize$\Raa^n$]{A_n}
    \Infer[no rule,rule margin=0pt]1{\iddots\qquad}
    \Infer[no rule,rule margin=-5pt]2{\text{\footnotesize $I$-rules}}
    \Ellipsis{}{C}
  \end{prooftree}
\end{center}

This makes clear that in order to obtain the normal form for classical natural deduction (with respect to classical detours \textit{\`a la} Prawitz), it is sufficient to push $\Raa$ downwards with respect to the elimination rules whose major premiss is the conclusion of $\Raa$, while nothing is said about the possibility of pushing $\Raa$ downwards with respect to introduction rules, or to elimination rules whose major premiss is not a conclusion of $\Raa$. 
Nevertheless, in \cite[p.~210]{vonPlatoSiders12}, von Plato and Siders also mention the possibility of defining the reduction steps for pushing $\Raa$ downwards with respect to the introduction rules of the propositional fragment. These reduction steps are then explicitly given by von Plato in \cite[pp.~86--87]{vonPlato13}.

\subsection{Seldin's normalization strategy (1986)}
\label{subsect:seldinstrategy}

Turning now to Seldin's approach \cite{Seldin86} for the weak normalization of first-order classical natural deduction, we can see that the principal difference to Prawitz rests on the notion of classical detour which is adopted. Seldin considers that a classical detour consists in a $\Raa$ rule introducing a formula occurrence $A$, which is immediately followed by another rule having $A$ as one of its premisses (see \cite[p.~638]{Seldin86}).

This characterization is more general than Prawitz's: in order to eliminate a classical detour, $\Raa$ must be pushed downwards not only with respect to the major premiss of the elimination rules, but also with respect to
any (introduction or elimination) inference rule immediately below $\Raa$.

It would be tempting, for such a purpose, to appeal to the reduction scheme (\ref{statmanred}) and to generalize it in order to define a reduction procedure for Seldin's notion of classical detour. The idea would be to replace the elimination rule $\circ_\Elim$ in \eqref{statmanred} by some rule $\mathsf{s}$. 
More precisely, given a derivation with an instance $\mathsf{r}$ of $\Raa$ which is not its last rule, a reduction step $\leadsto$ either pushes $\mathsf{r}$ downwards or erases $\mathsf{r}$, \Ie~if $\mathsf{s}$ is an instance of a ($1$-, $2$- or $3$-ary) rule immediately below $\mathsf{r}$, one has (for $C \neq \bot$):
{\small
\begin{align*}
  \begin{prooftree}[label separation = 0.3em]
    \Hypo{\ulcorner \lnot A \urcorner^1}
    \Ellipsis{$\pi$}{\; \bot \;}
    \Infer1[\scriptsize$\Raa^1$]{A}
      \Hypo{}
      \Ellipsis{$\pi'$}{ }
    \Infer2[\scriptsize$\mathsf{s}$]{C}
    \Ellipsis{$\pi_0$}{}
  \end{prooftree}
  \quad\
  &\leadsto
  \begin{prooftree}[label separation = 0.3em]
      \Hypo{\ulcorner \lnot C \urcorner^2}
	\Hypo{\ulcorner A \urcorner^1}
	\Hypo{}
	\Ellipsis{$\pi'$}{}
      \Infer2[\scriptsize$\mathsf{s}$]{C}
      \Infer2[\scriptsize$\lnot_\Elim$]{\bot}
      \Infer1[\scriptsize$\lnot_\Intro^1$]{\lnot A}
      \Ellipsis{$\pi$}{\; \bot \;}
      \Infer1[\scriptsize$\Raa^2$]{C}
    \Ellipsis{$\pi_0$}{}
  \end{prooftree}
  &\textup{ and }&&
  \begin{prooftree}
    \Hypo{\ulcorner \lnot A \urcorner^1}
    \Ellipsis{$\pi$}{\; \bot \;}
    \Infer1[\scriptsize$\Raa^1$]{A}
      \Hypo{}
      \Ellipsis{$\pi'$}{ }
    \Infer2[\scriptsize$\mathsf{s}$]{\bot}
    \Ellipsis{$\pi_0$}{}
  \end{prooftree}
  \quad\
  &\leadsto
  \begin{prooftree}
      \Hypo{\ulcorner A \urcorner^1}
	\Hypo{}
	\Ellipsis{$\pi'$}{}
      \Infer2[\scriptsize$\mathsf{s}$]{\bot}
      \Infer1[\scriptsize$\lnot_\Intro^1$]{\lnot A}
      \Ellipsis{$\pi$}{\bot}
    \Ellipsis{$\pi_0$}{}
  \end{prooftree}
\end{align*}
}

However, as already noticed by Seldin \cite[pp.~642, 645]{Seldin86}, these schemata work only if no assumption of $\pi$ is discharged at $\mathsf{s}$ in the derivations on the left-hand side of $\leadsto$; otherwise the transformation $\leadsto$ would change the set of non-discharged assumptions, adding new formulas in it (think for example of the case in which $\mathsf{s}$ is a $\to_\Intro$ rule, or a $\exists_\Elim$ rule with $A$ as its minor premiss).

Seldin's solution consists in adopting two alternative schemata ($C \neq \bot$):

\begin{small}
\begin{align}\label{eq:seldin}
	\begin{prooftree}
	  \Hypo{\ulcorner \lnot A \urcorner^1}
	  \Ellipsis{$\pi$}{\; \bot \;}
	  \Infer1[\footnotesize$\Raa^1$]{A}
	    \Hypo{}
	    \Ellipsis{$\pi'$}{ }
	  \Infer2[\footnotesize$\mathsf{s}$]{C}
	  \Ellipsis{$\pi_0$}{}
	\end{prooftree}
	\qquad
	&\leadsto
	\qquad
	\begin{prooftree}
	  \Hypo{\ulcorner \lnot C \urcorner^2}
	    \Hypo{\ulcorner \lnot C \urcorner^2}
	      \Hypo{\ulcorner A \urcorner^1}
	      \Hypo{}
	      \Ellipsis{$\pi'$}{}
	    \Infer2[\footnotesize$\mathsf{s}$]{C}
	    \Infer2[\footnotesize$\lnot_\Elim$]{\bot}
	    \Infer1[\footnotesize$\lnot_\Intro^1$]{\lnot A}
	    \Ellipsis{$\pi$}{\; \bot \;}
	    \Infer1[\footnotesize$\Efq$]{A}
	      \Hypo{}
	      \Ellipsis{$\pi'$}{}
	    \Infer2[\footnotesize$\mathsf{s}$]{C}
	    \Infer2[\footnotesize$\lnot_\Elim$]{\; \bot \;}
	    \Infer1[\footnotesize$\Raa^2$]{C}
	  \Ellipsis{$\pi_0$}{}
	\end{prooftree}
 \end{align}

\begin{align}\label{eq:seldinbis}
	\begin{prooftree}
	  \Hypo{\ulcorner \lnot A \urcorner^1}
	  \Ellipsis{$\pi$}{\; \bot \;}
	  \Infer1[\footnotesize$\Raa^1$]{A}
	    \Hypo{}
	    \Ellipsis{$\pi'$}{ }
	  \Infer2[\footnotesize$\mathsf{s}$]{\bot}
	  \Ellipsis{$\pi_0$}{}
	\end{prooftree}
	\qquad
	&\leadsto
	\quad
	\begin{prooftree}
	    \Hypo{\ulcorner A \urcorner^1}
	      \Hypo{}
	      \Ellipsis{$\pi'$}{}
	    \Infer2[\footnotesize$\mathsf{s}$]{\bot}
	    \Infer1[\footnotesize$\lnot_\Intro^1$]{\lnot A}
	    \Ellipsis{$\pi$}{\; \bot \;}
	    \Infer1[\footnotesize$\Efq$]{A}
	      \Hypo{}
	      \Ellipsis{$\pi'$}{}
	    \Infer2[\footnotesize$\mathsf{s}$]{\bot}
	  \Ellipsis{$\pi_0$}{}
	\end{prooftree}
\end{align}
\end{small}

Note that these reductions make use of a $\Efq$ rule. We have already seen the same occurring in St{\aa}lmarck's reductions for the $\{\neg, \wedge, \to, \bot, \forall\}$ fragment. In fact, St{\aa}lmarck's reductions are just a particular case of Seldin's ones, as they only work when $\mathsf{s}$ is an elimination rule with $A$ as major premiss. 

It should also be noted that, in Seldin's case, no restrictions have to be made on the applications of $\lor_\Elim$ and $\exists_\Elim$ rules during the process of normalization. The reason is that his normalization strategy is not like St{\aa}lmarck's one, but follows the same pattern as Prawitz. 
First, all classical detours are eliminated, and secondly, all intuitionistic detours are eliminated. The difference to Prawitz is the way in which classical detours are eliminated. In Seldin, the elimination of classical detour consists in pushing down all instances of $\Raa$ with respect to all the other rules, and then contract these instances of $\Raa$ into one. In this way, what tells him when the elimination of classical detours has to stop is the position of $\Raa$ in the derivation tree, and not, like in Prawitz, the complexity of the formula to which $\Raa$ is applied. 
By borrowing a terminology from Girard's jargon, we could say that for Seldin the termination of classical detours elimination can be characterized in a \textit{geometrical} way rather than in a \textit{syntactical} one.\footnote{Note that also von Plato and Siders' normalization strategy can be characterized in geometric rather than in syntactic terms: in order to establish when a proof is in normal form it is sufficient to look at the position of the major premisses of the elimination rules -- namely, the fact of being in the position of assumptions -- and not at their syntactical form.}

However, even Seldin's strategy is not immune from some restrictions in order to work. 
The $\forall$ has to be dropped (\Ie~Seldin's approach works for the $\{\neg, \land, \lor, \to, \bot, \exists\}$ fragment, which is as expressive as the full language of first-order classical logic), since the reduction schemata \eqref{eq:seldin}-\eqref{eq:seldinbis} cannot be applied when $\mathsf{s}$ is a $\forall_\Intro$ rule.
  Indeed, the natural way to treat the $\forall_\Intro$ case would be the following reduction step:
  
  {\small
  \begin{align*}
    \Pi = 
    \begin{prooftree}
      \Hypo{\ulcorner \lnot A \urcorner^1}
      \Ellipsis{$\pi'$}{\; \bot \;}
      \Infer1[\small$\Raa^1$]{A}
      \Infer1[\small$\forall_\Intro$]{\forall x A}
      \Ellipsis{$\pi$}{}
    \end{prooftree}
    \qquad
    &\leadsto
    \qquad
    \begin{prooftree}
      \Hypo{\ulcorner \lnot \forall x A \urcorner^2}
	\Hypo{\ulcorner A \urcorner^1}
      \Infer1[\small$\forall_\Intro$]{\forall x A}
      \Infer2[\small$\lnot_\Elim$]{\bot}
      \Infer1[\small$\lnot_i^1$]{\lnot A}
      \Ellipsis{$\pi'$}{\; \bot \;}
      \Infer1[\small$\Raa^2$]{\forall x A}
      \Ellipsis{$\pi$}{}
    \end{prooftree}
    = \Pi'
  \end{align*}
  }%
  but $\Pi'$ is not a derivation in classical natural deduction (nor in its subsystems) because in $\Pi'$ the rule $\forall_\Intro$ is not correctly instantiated: indeed the variable $x$ may occur free in $A$ and $A$ is a non-discharged assumption when the rule $\forall_\Intro$ is applied in $\Pi'$.
  There is no (reasonable) way to treat the $\forall_\Intro$ case without adding any rule of some intermediate logic such as $\mathsf{MH}$, see \cite[in particular pp.~639-640]{Seldin86}.
 
Nevertheless, as already anticipated in the introduction, the problem concerning $\forall$ is a small limitation with respect to the great advantage of Seldin's strategy, consisting in obtaining Glivenko's theorem for intuitionistic logic as an immediate corollary (see \cite[pp.~637-638]{Seldin86})%
\footnote{Note that the same result can be obtained using the reduction rules proposed by von Plato, that we mentioned at the end of \S\ref{vonPlatoapproach} (see \cite[pp. 87--88; pp. 142--143]{vonPlato13}).}%
: in a derivation where $\Raa$ is postponed (possibly it contains several instances of $\Efq$ that are not its last rule), it is sufficient to replace its last rule -- a $\Raa$ rule -- by a $\lnot$-introduction rule discharging the same assumptions.

\subsection{Towards a unified approach}
In this article we will focus on the postponement of $\Raa$: we aim to show that Seldin's result about the postponement of $\Raa$ can be generalized in such a way that a Glivenko's theorem can be obtained not only for intuitionistic but also for minimal logic. In order to do this, \mbox{we will proceed in two steps.}

First, we will show that, with respect to Seldin's definition of classical detour, the use of the $\Efq$ rule in the reduction steps can be limited to just the case where $\mathsf{s}$ corresponds to the $\to_\Intro$ rule. However, the reduction steps that we define for classical detours will not come out from a uniform reduction scheme -- as in Seldin's original formulation -- but they will come out from a mixing of techniques. 
More precisely, we can divide the definition of our reduction steps according to two main cases (see \S\ref{sect:rewrite}):

\begin{enumerate}
\item when the maximum formula $A$ is obtained from a $\Raa$ followed by an elimination rule $\mathsf{s}$,

\begin{enumerate}
\item if $A$ is the major premiss of $\mathsf{s}$, we will follow the general scheme (\ref{statmanred}); 
\item if $A$ is one of the minor premisses of $\mathsf{s}$, we will introduce new specific reduction steps.
\end{enumerate}

\item when the maximum formula $A$ is obtained from a $\Raa$ followed by an introduction rule, we will follow (with some emendations for the case of the implication) the reduction steps proposed by von Plato (\cite[p.~85--86]{vonPlato13}; cf. also the end of \S\ref{vonPlatoapproach}, \textit{supra}). 
\end{enumerate}

As it will become clear in \S\ref{sect:postponement} (see also Definition~\ref{def:distance}), our main concern is the ``geometrical'' character of the postponement of $\Raa$, while normalization of classical logic is only an indirect target. In this sense, we will see that it is not necessary to take into consideration the complexity of the formulas introduced by the $\Raa$ rule. This will lead us to consider also other kinds of reduction steps involving instances of rules in which more than one of their premisses is obtained from $\Raa$. Multiple occurrences of $\Raa$ then need to be considered -- and reduced -- at the same time, even if they create maximum formulas of greater complexity. Nevertheless, even by working with this new kind of reduction step, we can eventually show that a (weak) normalization theorem for classical logic can be recovered.

Secondly, we will show how the reductions that we propose can shed light on a problem raised by Pereira \cite{Pereira00} and determine the exact relations between normalization strategies for classical logic and negative translations. It is not difficult to see that Prawitz's original normalization strategy \cite[pp.39-40]{Prawitz65} for the fragment $\{\neg, \wedge, \to, \bot, \forall\}$ of classical natural deduction induces a negative translation: it is sufficient to replace every atomic instance of $\Raa$ present in a normal classical proof with a $\neg_\Intro$ in order to obtain a variant $(\cdot)^\mathtt{g}$ of Gentzen's translation, such that 
\begin{align*}
  (P(t_1, \ldots, t_n))^{\texttt{g}} &= \neg \neg P(t_1, \ldots, t_n) &
  (\bot)^{\texttt{g}} &= \bot \\
  (A \wedge B)^{\texttt{g}} &= A^{\texttt{g}} \wedge B^{\texttt{g}} &
  (A \to B)^{\texttt{g}} &= A \to B^{\texttt{g}} \\
  (\lnot A)^{\texttt{g}} &= \lnot A^{\texttt{g}} &
  (\forall{x}A)^{\texttt{g}} &= \forall{x}A^{\texttt{g}}
\end{align*}

\noindent and where $\vee$ and $\exists$ are already translated, since they are defined via a combinations of the primitive connectives, \Ie~$A \vee B \Defeq \neg (\neg A \wedge \neg B)$ and $\exists{x}A \Defeq \neg \forall{x} \neg A$. 
Note also that since $A^{\texttt{g}}$ is minimally equivalent to $\neg \neg A^{\texttt{g}}$, then $A \to B^{\texttt{g}}$ is minimally equivalent to $A \to \neg \neg B^{\texttt{g}}$; but $A \to \neg \neg B^{\texttt{g}}$ is also minimally equivalent to $\neg (A \wedge B^{\texttt{g}})$. 
Hence, just by using minimal logic steps this variant of Gentzen's negative translation can be transformed into a variant of G\"odel's negative translation (cf. \cite[p.~22]{Pereira00}).\footnote{Even when the full language fragment is considered, like in Stålmarck's \cite{Stalmarck91} or in von Plato and Siders' \cite{vonPlatoSiders12} approaches, it is possible to detect the use of some kind of negative translations. See Appendix \ref{sect:hidnegativetrans} for more details.}

In a similar way, we will show that Seldin's normalization strategy induces another kind of negative translation, namely (a variant of) Kuroda's one. 
However, a crucial difference exists between the Gödel--Gentzen translation and Kuroda's original translation: the former 
can embed classical logic into minimal logic, while the latter cannot (see \cite{FerreiraOliva12}).
The parallel between Prawitz's and Seldin's strategies would then work only partially. 
Actually, we will show that this is too harsh a conclusion. 
Indeed, by a slight modification of Kuroda's translation induced by our reduction steps, we obtain an embedding of full first-order classical logic 
into the fragment $\{\neg, \wedge, \vee, \bot, \exists\}$ of minimal logic.
In particular, the (adequate) fragment $\{\neg, \wedge, \vee, \bot, \exists\}$ of first-order classical logic (where $\to$ and $\forall$ are defined using the other connectives) is embedded into minimal logic via this variant of Kuroda's translation simply by adding a double negation in front of formulas: we get in this way a Glivenko theorem for minimal logic. 

In fact, in \cite[pp.~203,216]{Seldin89}, Seldin already proved a form of Glivenko's theorem consisting of embedding the system $\mathit{TD}^*$ into minimal logic. But $\mathit{TD}^*$ is a weaker system than first-order classical logic, since it corresponds to first-order minimal logic plus the rule of \textit{consequentia mirabilis}

\begin{center}
\begin{prooftree}
	\Hypo{\ulcorner \neg A \urcorner^1}
	 \Ellipsis{$\pi$}{A}
	 \Infer1[\footnotesize$\mathsf{cm}^1$]{\;A\;}
\end{prooftree}
\end{center}

\noindent and in this system neither $\Raa$ nor $\Efq$ are derivable (see \cite{AriolaHerbelinSabry05} for details). Our result is thus more general. 

It is worth noting that, unlike the algebraic demonstration given in \cite{ErtolaSagastume08}, our demonstration makes use of purely proof-theoretic tools and is not restricted to the propositional fragment. But this does not mean that our result is the only proof-theoretic demonstration of Glivenko's theorem for first-order minimal logic. A proof-theoretic demonstration of this theorem is also given by Tennant \cite{Tennant87}. However, unlike our approach, he does not appeal to the postponement of $\Raa$ for classical logic, but he translates each classical inference rule into a corresponding derivable rule in the fragment $\{\neg, \wedge, \vee, \bot, \exists\}$ of minimal logic. In this way, by induction on the length of a derivation, he can then transform a classical derivation into a derivation in minimal logic (for more details see Appendix \ref{sect:tennant}). 


%
\section{The syntax of first-order natural deduction}
\label{sect:preliminairies}

Let us first recall (quite informally) the language of first-order logic that we will use for our presentation. 

Formulas 
are generated by the propositional connectives $\top$ (\emph{truth}
), $\bot$ (\emph{falsehood}
), $\lnot$ (\emph{negation}), 
$\land$ (\emph{conjunction}), $\lor$ (\emph{disjunction}), $\to$ (\emph{implication}), and the quantifiers $\forall$ (\emph{universal}) and $\exists$ (\emph{existential}), starting from an infinite set of individual variables (denoted by $x, y, z$, etc.) and, for any $n \in \Nat$, a set of $n$-ary function symbols (which for our purposes here we do not need to denote apart from the terms in which they occur) and a set of $n$-ary predicate symbols (denoted by $\Var, \Vartwo, \Varthree$, etc.).%
\footnote{In particular, for $n = 0$, we get a set of individual constants and a set of proposition symbols: therefore, we will consider propositional natural deduction as a subsystem of first-order natural deduction.}
Terms are denoted by $s,t$, etc.; formulas are denoted by $A, B, C$, etc., in particular atomic formulas different from $\bot, \top$ are denoted by $\Var(t_1, \dots, t_n)$ where $\Var$ is a $n$-ary predicate symbol ($n \in \Nat$). 
Sets of formulas are denoted by $\Gamma, \Delta$, etc.

Formulas are identified up to renaming of bound variables. 
The capture-avoiding substitution of a term $t$ for all the free occurrences of an individual variable $x$ in a formula $A$ is denoted by $A\Sub{t}{x}$: 
it is implicitly assumed that none of the individual variables occurring in $t$ are bound in $A$ 
(this condition can always be fulfilled by renaming the bound variables of $A$).


\ebproofset{label separation = 0.2em}
\begin{figure}[t]
  \begin{center}
  {\footnotesize
  \begin{prooftree}
    \Hypo{}
    \Ellipsis{}{\; \bot \;}
    \Infer1[\scriptsize$\Efq$]{A}
  \end{prooftree}
  \quad\
  \begin{prooftree}
    \Hypo{\ulcorner \lnot A \urcorner^*}
    \Ellipsis{}{\; \bot \;}
    \Infer1[\scriptsize$\Raa^*$]{A}
  \end{prooftree}
  \quad\
  \begin{prooftree}
    \Infer0[\scriptsize$\!\top_{\Intro}$]{\; \top \;}
  \end{prooftree}
  \quad\
  \begin{prooftree}
      \Hypo{\ulcorner A \urcorner^*}
      \Ellipsis{}{\; \bot \;}
      \Infer1[\scriptsize$\lnot_\Intro^*$]{\lnot A}
  \end{prooftree}
  \quad\
  \begin{prooftree}[separation=1.3em]
      \Hypo{}
      \Ellipsis{}{\lnot A}
      \Hypo{}
      \Ellipsis{}{A}
      \Infer2[\scriptsize$\lnot_\Elim$]{\bot}
  \end{prooftree}
  \quad\
  \begin{prooftree}
    \Hypo{\ulcorner A \urcorner^*}
    \Ellipsis{}{\; B \;}
    \Infer1[\scriptsize$\to_\Intro^*$]{A \to B}
  \end{prooftree}
  \quad\
  \begin{prooftree}[separation=1.3em]
    \Hypo{}
    \Ellipsis{}{A \to B}
    \Hypo{}
    \Ellipsis{}{A}
    \Infer2[\scriptsize$\to_\Elim$]{B}
  \end{prooftree}

\medskip
 \begin{prooftree}[separation=1.3em]
    \Hypo{}
    \Ellipsis{}{A}
    \Hypo{}
    \Ellipsis{}{B}
    \Infer2[\scriptsize$\land_\Intro$]{A \land B}
  \end{prooftree}
  \quad\ 
  \begin{prooftree}
    \Hypo{}
    \Ellipsis{}{A \land B}
    \Infer1[\scriptsize$\land_{\ElimOne}$]{A}
  \end{prooftree}
\quad\ 
\begin{prooftree}[separation=1.3em]
    \Hypo{}
    \Ellipsis{}{A \land B}
    \Infer1[\scriptsize$\land_{\ElimTwo}$]{B}
  \end{prooftree}
  \quad\ 
  \begin{prooftree}
    \Hypo{}
    \Ellipsis{}{A}
    \Infer1[\scriptsize$\lor_{\IntroOne}$]{A \lor B}
  \end{prooftree}
  \quad\ 
  \begin{prooftree}
    \Hypo{}
    \Ellipsis{}{B}
    \Infer1[\scriptsize$\lor_{\IntroTwo}$]{A \lor B}
  \end{prooftree}
  \quad\ 
  \begin{prooftree}[separation=0.8em]
    \Hypo{}
    \Ellipsis{}{A \lor B}
    \Hypo{\ulcorner A \urcorner^*}
    \Ellipsis{}{C}
    \Hypo{\ulcorner B \urcorner^*}
    \Ellipsis{}{C}
    \Infer3[\scriptsize$\lor_\Elim^*$]{C}
  \end{prooftree}
  
  \medskip
  \begin{prooftree}
    \Hypo{}
    \Ellipsis{}{A}
    \Infer1[\scriptsize$\forall_{\Intro}$]{\forall x A}
  \end{prooftree}
  \qquad
  \begin{prooftree}
    \Hypo{}
    \Ellipsis{}{\forall x A}
    \Infer1[\scriptsize$\forall_{\Elim}$]{A \Sub{t}{x}}
  \end{prooftree}
  \qquad
  \begin{prooftree}
    \Hypo{}
    \Ellipsis{}{A \Sub{t}{x}}
    \Infer1[\scriptsize$\exists_{\Intro}$]{\exists x A}
  \end{prooftree}
  \qquad
  \begin{prooftree}[separation=0.8em]
    \Hypo{}
    \Ellipsis{}{\exists x A}
    \Hypo{\ulcorner A \urcorner^*}
    \Ellipsis{$\pi$}{C}
    \Infer2[\scriptsize$\exists_\Elim^*$]{C}
  \end{prooftree}
  }
  \end{center}
  
\medskip
{\footnotesize 
Every formula occurrence in a derivation $\pi$ that is not the conclusion of some rule instance in $\pi$ is an \emph{assumption of $\pi$}: it may be \emph{discharged} or \emph{undischarged} (in $\pi$).
A discharged assumption $A$ of a derivation $\pi$ is denoted by $\ulcorner A \urcorner$ with a marker -- here noted with a placeholder $*$ for a numeral -- for indicating the instance of rule in $\pi$ that has discharged it.
The rules that can discharge assumptions are $\Raa$, 
 $\to_\Intro$, $\lor_\Elim$ and $\exists_\Elim$: any instance of these rules may discharge an arbitrary number of assumptions; possibly none. 
In the rule $\forall_\Intro$, the variable $x$ must not be free in the undischarged assumptions. 
In the rule $\exists_\Elim$, the variable $x$ must not be free in $C$ or the non-discharged assumptions of $\pi$ different from $A$
.}
  
  \caption{\label{fig:rules}Inference rules for first-order natural deduction ($\Intro$ stands for intro, $\Elim$ for elim).}
\end{figure}
\ebproofset{label separation = 0.5em}

A \emph{derivation system} in first-order natural deduction is the set of derivations that can be obtained from a given set of inference rules. In other words, a derivation system is identified with the set of its inference rules.
The complete list of inference rules that we will consider for any derivation system in first-order natural deduction is in Figure~\ref{fig:rules}.
Given two derivation systems $\mathsf{D}$ and $\mathsf{D}'$, $\mathsf{D}$ is a \emph{subsystem} of $\mathsf{D}'$ if $\mathsf{D} \subseteq \mathsf{D}'$; hence, any derivation in $\mathsf{D}$ is also a derivation in $\mathsf{D}'$.
The notions of derivation, conclusion and (major, minor) premisses of an instance of rule are taken for granted (see for example \cite{Prawitz65}).


Looking at the inference rules in Figure~\ref{fig:rules}, observe that $\Efq$ is nothing but the special case of $\Raa$ where no assumption is discharged. This means that, in a derivation $\pi$, every instance of the rule $\Efq$ is just an instance of the rule $\Raa$ discharging no assumption. We can thus say that, in a derivation $\pi$, an instance of $\Raa$ is \emph{discharging} if it is not an instance of $\Efq$ (\Ie~it discharges at least one assumption).


Note that 
negation $\lnot$ is here considered as 
primitive
:\footnote{This is also what is done by Andou \cite{Andou95} in his proof of (weak) normalization for classical logic (cf. \cite[p.~152]{Andou95}). Taking $\neg$ as primitive is essential for him in order to define what he calls a \textit{regular proof}, that is, a proof in which all the discharged assumptions of $\Raa$ are major premisses of a $\neg_\Elim$ \cite[p.~154]{Andou95}. We have not analyzed Andou's proposal in \S\ref{sect:overview} since we should have considered not only the detour reductions, but also other kinds of transformations necessary for putting non-normal proofs in the regular form.}
$\lnot A$ will not be treated as 
a shorthand for $A \to \bot$, 
and the inferences rules $\lnot_\Intro$ and $\lnot_\Elim$ (see Figure~\ref{fig:rules}) will be not special cases of $\to_\Intro$ and $\to_\Elim$, respectively.
The reason is that, for our purposes (see in particular \S\ref{sect:Glivenko}), it 
turns out that the rules $\to_\Intro$ and $\lnot_\Intro$ have different behavior.%
\footnote{
Taking $\lnot$ as primitive is just a matter of convenience: all our results can be proved in a setting where $\lnot A \Defeq A \to \bot$ and the rules $\lnot_\Intro$ and $\lnot_\Elim$ are special cases of $\to_\Intro$ and $\to_\Elim$. However, this requires us to distinguish, for all 
formulas $A \to B$, whether $B$ equals $\bot$.}

We say that the \emph{first-order minimal natural deduction} is the derivation system $\Nm = \{
\top_\Intro, \lnot_\Intro, \lnot_\Elim, \land_\Intro, \allowbreak\land_{\ElimOne}, \allowbreak\land_{\ElimTwo}, \allowbreak \lor_{\IntroOne}, \allowbreak\lor_{\IntroTwo}, \lor_\Elim, \to_\Intro, \to_\Elim, \forall_\Intro, \forall_\Elim, \exists_\Intro, \exists_\Elim\}$ (\Ie~in $\Nm$ there are all the inference rules in Figure~\ref{fig:rules} except $\Raa$ and $\Efq$);
the \emph{first-order intuitionistic natural deduction} is the derivation system $\Nj = \Nm \cup \{\Efq\}$; and the \emph{first-order classical natural deduction} is the derivation system $\Nk = \Nm \cup \{\Raa\}$.
This means that $\Nm$ is a subsystem of $\Nj$, and that $\Nm$ and $\Nj$ are both subsystems of $\Nk$. 
All the derivation systems we will consider are subsystems of $\Nk$.

%
%



\begin{note}
  Let $\mathsf{D} \subseteq \Nk$ be a derivation system.
  \begin{romanenum}
    \item Derivations in $\mathsf{D}$ are denoted by $\Pi, \Pi', \dots$, or also $\pi, \pi', \dots$
    .
    
    \item Given a derivation $\pi$ in $\mathsf{D}$, let $\RAA_{\pi}$ (resp.~$\RAAp_{\pi}$) denote the set of instances (resp.~discharging instances) of the rule $\Raa$ in $\pi$.

    \item\label{note:assumption} Given a formula $B$, a set of formulas $\Gamma$, and a derivation $\pi$ in $\mathsf{D}$, we write $\pi \colon \Gamma \vdash_\mathsf{D} B$ (or simply $\pi \colon \Gamma \vdash B$ when no ambiguity arises) to indicate that the conclusion of $\pi$ is $B$ and the undischarged assumptions of $\pi$ are occurrences of some formulas in $\Gamma$; possibly, not all formulas in $\Gamma$ occur as undischarged assumptions in $\pi$.
    
    \item If there is a derivation $\pi \colon \Gamma \vdash_\mathsf{D} B$
    , we write $\Gamma \vdash_\mathsf{D} B$ and say that \emph{$B$ is derivable from $\Gamma$} (or \emph{$\Gamma \vdash B$ is derivable}) \emph{in $\mathsf{D}$}; \mbox{otherwise we write $\Gamma \not\vdash_\mathsf{D} B$.}
  \end{romanenum}
\end{note}

Clearly, $\RAAp_{\pi} \subseteq \RAA_{\pi}$ for every derivation $\pi$ in $\mathsf{D} \subseteq \Nk$.

Given an instance $\mathsf{r}$ of a rule in a derivation $\pi$, it is natural to define the notion of distance of  $\mathsf{r}$ in $\pi$ 
as the number of instances of rules in $\pi$ between $\mathsf{r}$ and the last rule of $\pi$. More formally, we have that:

\begin{defin}[Thread; distance of a rule; $\RAA$-size of a proof; standard derivation]\label{def:distance}
  Let $\pi$ be a derivation in $\mathsf{D} \subseteq \Nk$
  .
  
  Given two formula occurrences $A$ and $B$ in $\pi$, a \emph{thread from $A$ to $B$ in $\pi$} is a sequence $\mathsf{t} = (A_i)_{0 \leq i \leq n}$ (with $n \in \Nat$) of formula occurrences in $\pi$ such that $A_0 = A$, $A_n = B$ and, for any $0 \leq i < n$ there is an instance of rule in $\pi$ having $A_i$ as a premise and $A_{i+1}$ as its conclusion;
  the \emph{length of $\mathsf{t}$} is $n$.%
  \footnote{Note that if $A = B$ (as formula occurrences in $\pi$) then the length of $\mathsf{t}$ is $0$.}
  
  For every instance $\mathsf{r}$ of a rule in $\pi$, the \emph{distance of $\mathsf{r}$} (\emph{from the conclusion of $\pi$}), denoted by $\Dist{\pi}(\mathsf{r})$, is the length of the thread from the conclusion of $\mathsf{r}$ to the conclusion of $\pi$.%
  \footnote{This notion is well-defined since a derivation is a tree (\Ie~a rooted acyclic connected graph) whose nodes are formula occurrences. Hence, for every formula occurrence $A$ in $\pi$, there exists exactly one thread from $A$ to the conclusion of $\pi$.}

  An $\mathsf{r} \in \RAA_{\pi}$ is \emph{$\RAA_{\pi}$-maximal} if $\Dist{\pi}(\mathsf{r}) \geq \Dist{\pi}(\mathsf{r}')$ for any $\mathsf{r}' \in \RAA_{\pi}$. 

  An $\mathsf{r} \in \RAAp_{\pi}$ is \emph{$\RAAp_{\pi}$-maximal} if $\Dist{\pi}(\mathsf{r}) \geq \Dist{\pi}(\mathsf{r}')$ for any $\mathsf{r}' \in \RAAp_{\pi}$. 

  The \emph{$\RAA$-size of $\pi$} is $\Size_\RAA(\pi) = \sum_{\mathsf{r} \in \RAA_{\pi}} \Dist{\pi}(\mathsf{r})$.
  
  The \emph{$\RAAp$-size of $\pi$} is $\Size_{\RAAp}(\pi) = \sum_{\mathsf{r} \in \RAAp_{\pi}} \Dist{\pi}(\mathsf{r})$.
  
  We say that $\pi$ is \emph{\Mstandard} if in $\pi$ there is at most one instance of the rule $\Raa$, and this instance, if any, is the last rule of $\pi$, the rest of $\pi$ being a derivation in $\Nm$.
  
  We say that $\pi$ is \emph{\Jstandard} if in $\pi$ there is at most one discharging instance of the rule $\Raa$, and this instance, if any, is the last rule of $\pi$, the rest of $\pi$ being a derivation in $\Nj$.
\end{defin}

The notion of \Mstandard\ (resp.~\Jstandard) derivations characterizes exactly the derivations in $\Nk$ where the $\Raa$ (resp.~discharging $\Raa$) is postponed.
These notions will be used in Theorem~\ref{thm:postponement} and Corollary~\ref{cor:postponement}.
A \Jstandard\ derivation might contains several instances of $\Efq$, which are not its last rule.

Note that, for any derivation $\pi$ and any instance $\mathsf{r}$ of a rule in $\pi$, one has that $\Dist{\pi}(\mathsf{r}) \in \Nat$ and $\Size_{\RAA}(\pi) \geq \Size_{\RAAp}(\pi) \in \Nat$.
Moreover, $\Dist{\pi}(\mathsf{r}) = 0$ if and only if $\mathsf{r}$ is the last rule in $\pi$.

Since we are mainly interested in the postponement of $\Raa$ instead of normalization, differently from the approaches discussed in \S\ref{sect:overview}, our definitions of $\RAA_\pi$-maximal and $\RAAp_\pi$-maximal in a derivation $\pi$ depend only on the distances of the instances of $\Raa$ from the conclusion of $\pi$, without taking into account the complexity of the formulas occurring in the conclusions of these instances. 
In a way, as we will see in \S\ref{sect:postponement}, our approach to prove the postponement of $\Raa$ is purely ``geometrical''.

\begin{rmk}\label{rmk:size0}
  Let $\pi$ be a derivation in $\mathsf{D} \subseteq \Nk$:
  \begin{enumerate}
    \item\label{rmk:size0.intuitionistic} $\Size_{\RAAp}(\pi) = 0$ if and only if $\pi$ is \Jstandard;
    \item\label{rmk:size0.minimal} $\Size_{\RAA}(\pi) = 0$ if and only if 
    $\pi$ is \Mstandard.

  \end{enumerate}
\end{rmk}

Intuitively, in a derivation $\pi$ in $\mathsf{D} \subseteq \Nk$, an instance (resp.~a discharging instance) $\mathsf{r}$ of the rule $\Raa$ is $\RAA_{\pi}$-maximal (resp.~$\RAAp_{\pi}$-maximal) when there are no other instances 
(resp.~
discharging instances) of $\Raa$ above $\mathsf{r}$.

Since a derivation $\pi$ in $\mathsf{D} \subseteq \Nk$ can be seen as a finite tree, if $\RAA_{\pi} \neq \emptyset$ (resp.~$\RAAp_{\pi} \neq \emptyset$), \Ie~if there is at least one instance (resp.~discharging instance) of the rule $\Raa$ in $\pi$, then there is a $\RAA_{\pi}$-maximal (resp.~$\RAAp_{\pi}$-maximal discharging) instance of the rule $\Raa$ in $\pi$.


%
\section{\texorpdfstring{Reduction steps for the postponement of $\Raa$}{Reduction steps for the postponement of raa}}
\label{sect:rewrite}

We define reduction steps case by case, 
depending on the inference rule instantiated immediately below the 
instance of $\Raa$ under focus (thus there is no case with a $0$-ary inference rule).

\begin{description}
  \item[$\lnot$ introduction:]
  {\small
  \begin{align}\label{eq:notintro}
    \Pi = 
    \begin{prooftree}
      \Hypo{\ulcorner A \urcorner^2, \ulcorner \lnot \bot \urcorner^1}
      \Ellipsis{$\pi'$}{\; \bot \;}
      \Infer1[\footnotesize$\Raa^1$]{\bot}
      \Infer1[\footnotesize$\lnot_\Intro^2$]{\lnot A}
      \Ellipsis{$\pi$}{}
    \end{prooftree}
    \qquad
    &\leadsto
    \qquad
    \begin{prooftree}
	\Hypo{\ulcorner A \urcorner^2,}
	  \Hypo{\ulcorner \bot \urcorner^1}
	\Infer1[\footnotesize$\lnot_\Intro^1$]{\lnot \bot}
	\Infer[no rule,separation=0.2em]2{}
	\Ellipsis{$\pi'$}{\bot}
	\Infer1[\footnotesize$\lnot_\Intro^2$]{\lnot A}
      \Ellipsis{$\pi$}{}
    \end{prooftree}
    = \Pi'
  \end{align}
  }
  
  \item[$\lnot$ elimination:]
  \begin{subequations}\label{eq:notelimgeneral}
  {\small
  \begin{align}
    \Pi = 
    \begin{prooftree}
      \Hypo{\ulcorner \lnot \lnot A \urcorner^1}
      \Ellipsis{$\pi'$}{\; \bot \;}
      \Infer1[\footnotesize$\Raa^1$]{\lnot A}
      \Hypo{}
      \Ellipsis{$\pi''$}{A}
      \Infer2[\footnotesize$\lnot_\Elim$]{\bot}
      \Ellipsis{$\pi$}{}
    \end{prooftree}
    \qquad
    &\leadsto
    \qquad
    \begin{prooftree}
      \Hypo{\ulcorner \lnot A \urcorner^1}
      \Hypo{}
      \Ellipsis{$\pi''$}{A}
      \Infer2[\footnotesize$\lnot _\Elim$]{\bot}      
      \Infer1[\footnotesize$\lnot_\Intro^1$]{\lnot \lnot A}
      \Ellipsis{$\pi'$}{\bot}
      \Ellipsis{$\pi$}{}
    \end{prooftree}
    = \Pi'
  \end{align}
  }\nopagebreak where the last rule of the derivation $\pi''$ is not an instance of the rule $\Raa$;
    
  {\small
  \begin{align}
    \Pi = 
    \begin{prooftree}
      \Hypo{}
      \Ellipsis{$\pi'$}{\lnot A}
      \Hypo{\ulcorner \lnot A \urcorner^1}
      \Ellipsis{$\pi''$}{\; \bot \;}
      \Infer1[\footnotesize$\Raa^1$]{A}
      \Infer2[\footnotesize$\lnot_\Elim$]{\bot}
      \Ellipsis{$\pi$}{}
    \end{prooftree}
    \qquad
    &\leadsto
    \qquad
    \begin{prooftree}
      \Hypo{}
      \Ellipsis{$\pi'$}{\lnot A}
      \Ellipsis{$\pi''$}{\bot}
      \Ellipsis{$\pi$}{}
    \end{prooftree}
    \quad = \Pi'
  \end{align}
  }where the last rule of the derivation $\pi'$ is not an instance of the rule $\Raa$;
    
  {\small
  \begin{align}\label{eq:notelim}
    \Pi = 
    \begin{prooftree}
      \Hypo{\ulcorner \lnot \lnot A \urcorner^1}
      \Ellipsis{$\pi'$}{\; \bot \;}
      \Infer1[\footnotesize$\Raa^1$]{\lnot A}
      \Hypo{\ulcorner \lnot A \urcorner^2}
      \Ellipsis{$\pi''$}{\; \bot \;}
      \Infer1[\footnotesize$\Raa^2$]{A}
      \Infer2[\footnotesize$\lnot_\Elim$]{\bot}
      \Ellipsis{$\pi$}{}
    \end{prooftree}
    \qquad
    &\leadsto
    \qquad
    \begin{prooftree}
      \Hypo{\ulcorner \lnot A \urcorner^2}
      \Hypo{\ulcorner A \urcorner^1}
      \Infer2[\footnotesize$\lnot_\Elim$]{\bot}      
      \Infer1[\footnotesize$\lnot_\Intro^1$]{\lnot A}
      \Ellipsis{$\pi''$}{\; \bot \;}
      \Infer1[\footnotesize$\lnot_\Intro^2$]{\lnot \lnot A}
      \Ellipsis{$\pi'$}{\bot}
      \Ellipsis{$\pi$}{}
    \end{prooftree}
    = \Pi'
  \end{align}
  }
  \end{subequations}

  \item[$\land$ introduction:] 
  \begin{subequations}
  {\small
  \begin{align}
    \Pi = 
    \begin{prooftree}
      \Hypo{\ulcorner \lnot A \urcorner^1}
      \Ellipsis{$\pi'$}{\; \bot \;}
      \Infer1[\footnotesize$\Raa^1$]{A}
      \Hypo{}
      \Ellipsis{$\pi''$}{B}
      \Infer2[\footnotesize$\land_\Intro$]{A \land B}
      \Ellipsis{$\pi$}{}
    \end{prooftree}
    \qquad
    &\leadsto
    \qquad
    \begin{prooftree}
      \Hypo{\ulcorner \lnot (A \land B) \urcorner^2}
      \Hypo{\ulcorner A \urcorner^1}
      \Hypo{}
      \Ellipsis{$\pi''$}{B}
      \Infer2[\footnotesize$\land_\Intro$]{A \land B}      
      \Infer2[\footnotesize$\lnot_\Elim$]{\bot}
      \Infer1[\footnotesize$\lnot_\Intro^1$]{\lnot A}
      \Ellipsis{$\pi'$}{\; \bot \;}
      \Infer1[\footnotesize$\Raa^2$]{A \land B}
      \Ellipsis{$\pi$}{}
    \end{prooftree}
    = \Pi'
  \end{align}
  }where the last rule of the derivation $\pi''$ is not an instance of the rule $\Raa$;
    
  {\small
  \begin{align}
    \Pi = 
    \begin{prooftree}
      \Hypo{}
      \Ellipsis{$\pi'$}{A}
      \Hypo{\ulcorner \lnot B \urcorner^1}
      \Ellipsis{$\pi''$}{\; \bot \;}
      \Infer1[\footnotesize$\Raa^1$]{B}
      \Infer2[\footnotesize$\land_\Intro$]{A \land B}
      \Ellipsis{$\pi$}{}
    \end{prooftree}
    \qquad
    &\leadsto
    \qquad
    \begin{prooftree}
      \Hypo{\ulcorner \lnot (A \land B) \urcorner^2}
      \Hypo{}
      \Ellipsis{$\pi'$}{A}
      \Hypo{\ulcorner B \urcorner^1}
      \Infer2[\footnotesize$\land_\Intro$]{A \land B}      
      \Infer2[\footnotesize$\lnot_\Elim$]{\bot}
      \Infer1[\footnotesize$\lnot_\Intro^1$]{\lnot B}
      \Ellipsis{$\pi''$}{\; \bot \;}
      \Infer1[\footnotesize$\Raa^2$]{A \land B}
      \Ellipsis{$\pi$}{}
    \end{prooftree}
    = \Pi'
  \end{align}
  }where the last rule of the derivation $\pi'$ is not an instance of the rule $\Raa$;
  
  {\small
  \begin{align}\label{eq:andelim}
    \Pi = 
    \begin{prooftree}[label separation = 0.2em]
      \Hypo{\ulcorner \lnot A \urcorner^1}
      \Ellipsis{$\pi'$}{\; \bot \;}
      \Infer1[\footnotesize$\Raa^1$]{A}
      \Hypo{\ulcorner \lnot B \urcorner^2}
      \Ellipsis{$\pi''$}{\; \bot \;}
      \Infer1[\footnotesize$\Raa^2$]{B}
      \Infer2[\footnotesize$\land_\Intro$]{A \land B}
      \Ellipsis{$\pi$}{}
    \end{prooftree}
    \quad
    &\leadsto
    \
    \begin{prooftree}[label separation = 0.2em]
      \Hypo{\ulcorner \lnot (A \land B) \urcorner^3}
      \Hypo{\ulcorner A \urcorner^2}
      \Hypo{\ulcorner B \urcorner^1}
      \Infer2[\footnotesize$\land_\Intro$]{A \land B}      
      \Infer2[\footnotesize$\lnot_\Elim$]{\bot}
      \Infer1[\footnotesize$\lnot_\Intro^1$]{\lnot B}
      \Ellipsis{$\pi''$}{\; \bot \;}
      \Infer1[\footnotesize$\lnot_\Intro^2$]{\lnot A}
      \Ellipsis{$\pi'$}{\bot}
      \Infer1[\footnotesize$\Raa^3$]{A \land B}
      \Ellipsis{$\pi$}{}
    \end{prooftree}
    \!\!\!\!= \Pi'
  \end{align}
  }
  \end{subequations}
  
  \item[$\land$ elimination:]
  \begin{subequations}
  {\small
  \begin{align}
    \Pi = 
    \begin{prooftree}
      \Hypo{\ulcorner \lnot (A \land B) \urcorner^1}
      \Ellipsis{$\pi'$}{\; \bot \;}
      \Infer1[\footnotesize$\Raa^1$]{A \land B}
      \Infer1[\footnotesize$\land_{\ElimOne}$]{A}
      \Ellipsis{$\pi$}{}
    \end{prooftree}
    \qquad
    &\leadsto
    \qquad
    \begin{prooftree}
      \Hypo{\ulcorner \lnot A \urcorner^2}
      \Hypo{\ulcorner A \land B \urcorner^1}
      \Infer1[\footnotesize$\land_{\ElimOne}$]{A}
      \Infer2[\footnotesize$\lnot_\Elim$]{\;\bot\;}
      \Infer1[\footnotesize$\lnot_\Intro^1$]{\lnot (A \land B)}
      \Ellipsis{$\pi'$}{\; \bot \;}
      \Infer1[\footnotesize$\Raa^2$]{\;A\;}
      \Ellipsis{$\pi$}{}
    \end{prooftree}
    = \Pi'
    \\ \notag\\
    \Pi = 
    \begin{prooftree}
      \Hypo{\ulcorner \lnot (A \land B) \urcorner^1}
      \Ellipsis{$\pi'$}{\; \bot \;}
      \Infer1[\footnotesize$\Raa^1$]{A \land B}
      \Infer1[\footnotesize$\land_{\ElimTwo}$]{B}
      \Ellipsis{$\pi$}{}
    \end{prooftree}
    \qquad
    &\leadsto
    \qquad
    \begin{prooftree}
      \Hypo{\ulcorner \lnot B \urcorner^2}
      \Hypo{\ulcorner A \land B \urcorner^1}
      \Infer1[\footnotesize$\land_{\ElimTwo}$]{B}
      \Infer2[\footnotesize$\lnot_\Elim$]{\;\bot\;}
      \Infer1[\footnotesize$\lnot_\Intro^1$]{\lnot (A \land B)}
      \Ellipsis{$\pi'$}{\; \bot \;}
      \Infer1[\footnotesize$\Raa^2$]{\;B\;}
      \Ellipsis{$\pi$}{}
    \end{prooftree}
    = \Pi'
  \end{align}
  }
  \end{subequations}
  
  \item[$\lor$ introduction:]
  \begin{subequations}
  {\small
  \begin{align}
    \Pi = 
    \begin{prooftree}
      \Hypo{\ulcorner \lnot A \urcorner^1}
      \Ellipsis{$\pi'$}{\; \bot \;}
      \Infer1[\footnotesize$\Raa^1$]{A}
      \Infer1[\footnotesize$\lor_{\IntroOne}$]{A \lor B}
      \Ellipsis{$\pi$}{}
    \end{prooftree}
    \qquad
    &\leadsto
    \qquad
    \begin{prooftree}
      \Hypo{\ulcorner \lnot (A \lor B) \urcorner^2}
      \Hypo{\ulcorner A \urcorner^1}
      \Infer1[\footnotesize$\lor_{\IntroOne}$]{A \lor B}
      \Infer2[\footnotesize$\lnot_\Elim$]{\;\bot\;}
      \Infer1[\footnotesize$\lnot_\Intro^1$]{\lnot A}
      \Ellipsis{$\pi'$}{\; \bot \;}
      \Infer1[\footnotesize$\Raa^2$]{A \lor B}
      \Ellipsis{$\pi$}{}
    \end{prooftree}
    = \Pi'
    \\ \notag\\
    \Pi = 
    \begin{prooftree}
      \Hypo{\ulcorner \lnot B \urcorner^1}
      \Ellipsis{$\pi'$}{\; \bot \;}
      \Infer1[\footnotesize$\Raa^1$]{B}
      \Infer1[\footnotesize$\lor_{\IntroTwo}$]{A \lor B}
      \Ellipsis{$\pi$}{}
    \end{prooftree}
    \qquad
    &\leadsto
    \qquad
    \begin{prooftree}
      \Hypo{\ulcorner \lnot (A \lor B) \urcorner^2}
      \Hypo{\ulcorner B \urcorner^1}
      \Infer1[\footnotesize$\lor_{\IntroTwo}$]{A \lor B}
      \Infer2[\footnotesize$\lnot_\Elim$]{\;\bot\;}
      \Infer1[\footnotesize$\lnot_\Intro^1$]{\lnot B}
      \Ellipsis{$\pi'$}{\; \bot \;}
      \Infer1[\footnotesize$\Raa^2$]{A \lor B}
      \Ellipsis{$\pi$}{}
    \end{prooftree}
    = \Pi'
  \end{align}
  }
  \end{subequations}
  
  \item[$\lor$ elimination:]
  \begin{subequations}
  {\small
  \begin{align}
    \Pi = \!\!\!\!
    \begin{prooftree}[label separation = 0.2em, separation = 0.7em]
      \Hypo{\ulcorner \lnot(A \lor B) \urcorner^1}
      \Ellipsis{$\pi'$}{\bot}
      \Infer1[\scriptsize$\Raa^1$]{A \lor B}
      \Hypo{\ulcorner A \urcorner^2}
      \Ellipsis{$\pi''$}{C}
      \Hypo{\ulcorner B \urcorner^2}
      \Ellipsis{$\pi'''$}{C}
      \Infer3[\scriptsize$\lor_\Elim^2$]{C}
      \Ellipsis{$\pi$}{}
    \end{prooftree}
    \ 
    &\leadsto
    \ 
    \begin{prooftree}[label separation = 0.2em, separation = 0.7em]
      \Hypo{\ulcorner A \lor B \urcorner^2}
      \Hypo{\ulcorner \lnot C \urcorner^3}
      \Hypo{\ulcorner A \urcorner^1}
	\Ellipsis{$\pi''$}{C}
	\Infer[separation = 0em]2[\scriptsize$\lnot_\Elim$]{\; \bot \;}
      \Hypo{\ulcorner \lnot C \urcorner^3}
      \Hypo{\ulcorner B \urcorner^1}
	\Ellipsis{$\pi'''$}{C}
	\Infer[separation = 0em]2[\scriptsize$\lnot_\Elim$]{\; \bot \;}
      \Infer3[\scriptsize$\lor_\Elim^1$]{\bot}
      \Infer1[\scriptsize$\lnot_\Intro^2$]{\lnot(A \lor B)}
      \Ellipsis{$\pi'$}{\; \bot \;}
      \Infer1[\scriptsize$\Raa^3$]{C}
      \Ellipsis{$\pi$}{}
    \end{prooftree}
    \!\!\!\! = \Pi'
  \end{align}
  }where the last rule of the derivations $\pi''$ and $\pi'''$ is not an instance of 
  $\Raa$;
  
  {\small
  \begin{align}
    \Pi = 
    \begin{prooftree}[label separation = 0.2em, separation = 0.2em]
      \Hypo{}
      \Ellipsis{$\pi'$}{A \lor B}
      \Hypo{\ulcorner A \urcorner^2\!, \ulcorner \lnot C \urcorner^1}
      \Ellipsis{$\pi''$}{\; \bot \;}
      \Infer1[\scriptsize$\Raa^1$]{C}
      \Hypo{\ulcorner B \urcorner^2}
      \Ellipsis{$\pi'''$}{C}
      \Infer3[\scriptsize$\lor_\Elim^2$]{C}
      \Ellipsis{$\pi$}{}
    \end{prooftree}
    \ 
    &\leadsto
    \ 
    \begin{prooftree}[label separation = 0.2em, separation = 0em]
      \Hypo{}
      \Ellipsis{$\pi'$}{A \lor B}
      \Hypo{\ulcorner A \urcorner^1\!, \ulcorner \lnot C \urcorner^2}
	\Ellipsis{$\pi''$}{\bot}
      \Hypo{\ulcorner \lnot C \urcorner^2}
      \Hypo{\ulcorner B \urcorner^1}
	\Ellipsis{$\pi'''$}{C}
	\Infer2[\scriptsize$\lnot_\Elim$]{\bot}
      \Infer3[\scriptsize$\lor_\Elim^1$]{\; \bot \;}
      \Infer1[\scriptsize$\Raa^2$]{C}
      \Ellipsis{$\pi$}{}
    \end{prooftree}
    = \Pi'
  \end{align}\nopagebreak
  }\nopagebreak where the last rule of the derivations $\pi'$ and $\pi'''$ is not an instance of 
  $\Raa$;

  {\small
  \begin{align}
    \Pi = 
    \begin{prooftree}[label separation = 0.2em, separation = 0.5em]
      \Hypo{}
      \Ellipsis{$\pi'$}{A \lor B}
      \Hypo{\ulcorner A \urcorner^2}
      \Ellipsis{$\pi''$}{C}
      \Hypo{\ulcorner B \urcorner^2\!, \ulcorner \lnot C \urcorner^1}
      \Ellipsis{$\pi'''$}{\; \bot \;}
      \Infer1[\scriptsize$\Raa^1$]{C}
      \Infer3[\scriptsize$\lor_\Elim^2$]{C}
      \Ellipsis{$\pi$}{}
    \end{prooftree}
    \ 
    &\leadsto
    \ 
    \begin{prooftree}[label separation = 0.2em, separation = 0.7em]
      \Hypo{}
      \Ellipsis{$\pi'$}{A \lor B}
      \Hypo{\ulcorner \lnot C \urcorner^2}
      \Hypo{\ulcorner A \urcorner^1}
	\Ellipsis{$\pi''$}{C}
	\Infer[separation = 0em]2[\scriptsize$\lnot_\Elim$]{\bot}
      \Hypo{\ulcorner B \urcorner^1\!, \ulcorner \lnot C \urcorner^2}
	\Ellipsis{$\pi'''$}{\bot}
      \Infer3[\scriptsize$\lor_\Elim^1$]{\; \bot \;}
      \Infer1[\scriptsize$\Raa^2$]{C}
      \Ellipsis{$\pi$}{}
    \end{prooftree}
    \!\!\!\! = \Pi'
  \end{align}
  }\nopagebreak where the last rule of the derivations $\pi'$ and $\pi''$ is not an instance of 
  $\Raa$;

  {\small
  \begin{align}
    \Pi = 
    \begin{prooftree}[label separation = 0.2em, separation = 0.6em]
      \Hypo{}
	\Ellipsis{$\pi'$}{A \lor B}
      \Hypo{\ulcorner A \urcorner^3\!, \ulcorner \lnot C \urcorner^1}
	\Ellipsis{$\pi''$}{\; \bot \;}
	\Infer1[\scriptsize$\Raa^1$]{C}
      \Hypo{\ulcorner B \urcorner^3\!, \ulcorner \lnot C \urcorner^2}
	\Ellipsis{$\pi'''$}{\; \bot \;}
	\Infer1[\scriptsize$\Raa^2$]{C}
      \Infer3[\scriptsize$\lor_\Elim^3$]{C}
      \Ellipsis{$\pi$}{}
    \end{prooftree}
    \ 
    &\leadsto
    \ 
    \begin{prooftree}[label separation = 0.2em, separation = 0.6em]
      \Hypo{}
	\Ellipsis{$\pi'$}{A \lor B}
      \Hypo{\ulcorner A \urcorner^1\!, \ulcorner \lnot C \urcorner^2}
	\Ellipsis{$\pi''$}{\bot}
      \Hypo{\ulcorner B \urcorner^1\!, \ulcorner \lnot C \urcorner^2}
	\Ellipsis{$\pi'''$}{\bot}
      \Infer3[\scriptsize$\lor_\Elim^1$]{\; \bot \;}
      \Infer1[\scriptsize$\Raa^2$]{C}
      \Ellipsis{$\pi$}{}
    \end{prooftree}
    \!\!\!\! = \Pi'
  \end{align}
  }\nopagebreak where the last rule of the derivation $\pi'$ is not an instance of the rule $\Raa$;

  {\small
  \begin{align}
    \Pi = \!\!\!\!
    \begin{prooftree}[label separation = 0.2em, separation=0.5em]
      \Hypo{\ulcorner \lnot(A \lor B) \urcorner^1}
	\Ellipsis{$\pi'$}{\bot}
	\Infer1[\scriptsize$\Raa^1$]{A \lor B}
      \Hypo{\ulcorner A \urcorner^3\!, \ulcorner \lnot C \urcorner^2}
      \Ellipsis{$\pi''$}{\; \bot \;}
      \Infer1[\scriptsize$\Raa^2$]{C}
      \Hypo{\ulcorner B \urcorner^3}
	\Ellipsis{$\pi'''$}{C}
      \Infer3[\scriptsize$\lor_\Elim^3$]{C}
      \Ellipsis{$\pi$}{}
    \end{prooftree}
    \ 
    &\leadsto
    \ 
    \begin{prooftree}[label separation = 0.2em, separation=0em]
      \Hypo{\ulcorner A \lor B \urcorner^2}
      \Hypo{\ulcorner A \urcorner^1\!, \ulcorner \lnot C \urcorner^3}
      \Ellipsis{$\pi''$}{\bot}
	\Hypo{\ulcorner \lnot C \urcorner^3}
	\Hypo{\ulcorner B \urcorner^1}
	\Ellipsis{$\pi'''$}{C}
      \Infer2[\scriptsize$\lnot_\Elim$]{\bot}
      \Infer3[\scriptsize$\lor_\Elim^1$]{\bot}
      \Infer1[\scriptsize$\lnot_\Intro^2$]{\lnot(A \lor B)}
      \Ellipsis{$\pi'$}{\; \bot \;}
      \Infer1[\scriptsize$\Raa^3$]{C}
      \Ellipsis{$\pi$}{}
    \end{prooftree}
    \!\!\!\!\! = \Pi'
  \end{align}
  }where the last rule of the derivation $\pi'''$ is not an instance of the rule $\Raa$;

  {\small
  \begin{align}
    \Pi = \!\!\!\!
    \begin{prooftree}[label separation = 0.2em, separation = 0.2em]
      \Hypo{\ulcorner \lnot(A \lor B) \urcorner^1}
	\Ellipsis{$\pi'$}{\bot}
	\Infer1[\scriptsize$\Raa^1$]{A \lor B}
      \Hypo{\ulcorner A \urcorner^3}
	\Ellipsis{$\pi''$}{C}
      \Hypo{\ulcorner B \urcorner^3\!, \ulcorner \lnot C \urcorner^2}
	\Ellipsis{$\pi'''$}{\; \bot \;}
	\Infer1[\scriptsize$\Raa^2$]{C}
      \Infer3[\scriptsize$\lor_\Elim^3$]{C}
      \Ellipsis{$\pi$}{}
    \end{prooftree}
    \ 
    &\leadsto
    \ 
    \begin{prooftree}[label separation = 0.2em, separation = 0.8em]
      \Hypo{\ulcorner A \lor B \urcorner^2}
	\Hypo{\ulcorner \lnot C \urcorner^3}
	\Hypo{\ulcorner A \urcorner^1}
	\Ellipsis{$\pi''$}{C}
      \Infer[separation = 0em]2[\scriptsize$\lnot_\Elim$]{\bot}
	\Hypo{\ulcorner B \urcorner^1\!, \ulcorner \lnot C \urcorner^3}
      \Ellipsis{$\pi'''$}{\bot}
      \Infer3[\scriptsize$\lor_\Elim^1$]{\bot}
      \Infer1[\scriptsize$\lnot_\Intro^2$]{\lnot(A \lor B)}
      \Ellipsis{$\pi'$}{\; \bot \;}
      \Infer1[\scriptsize$\Raa^3$]{C}
      \Ellipsis{$\pi$}{}
    \end{prooftree}
    \!\!\!\!\!\!\!\! = \Pi'
  \end{align}
  }where the last rule of the derivation $\pi''$ is not an instance of the rule $\Raa$;

  {\small
  \begin{align}
    \Pi = \!\!\!\!
    \begin{prooftree}[label separation = 0.2em, separation=0.5em]
      \Hypo{\ulcorner \lnot(A \lor B) \urcorner^1}
	\Ellipsis{$\pi'$}{\bot}
	\Infer1[\scriptsize$\Raa^1$]{A \lor B}
      \Hypo{\!\ulcorner A \urcorner^4\!, \ulcorner \lnot C \urcorner^2}
	\Ellipsis{$\pi''$}{\; \bot \;}
	\Infer1[\scriptsize$\Raa^2$]{C}
      \Hypo{\ulcorner B \urcorner^4\!, \ulcorner \lnot C \urcorner^3}
	\Ellipsis{$\pi'''$}{\; \bot \;}
	\Infer1[\scriptsize$\Raa^3$]{C}
      \Infer3[\scriptsize$\lor_\Elim^4$]{C}
      \Ellipsis{$\pi$}{}
    \end{prooftree}
    \ 
    &\leadsto
    \ 
    \begin{prooftree}[label separation = 0.2em, separation=0.7em]
      \Hypo{\ulcorner A \lor B \urcorner^2}
      \Hypo{\!\!\!\!\!\!\!\!\!\!\!\!\!\ulcorner A \urcorner^1\!, \ulcorner \lnot C \urcorner^3}
	\Ellipsis{$\pi''$}{\bot}
      \Hypo{\ulcorner B \urcorner^1\!, \ulcorner \lnot C \urcorner^3}
	\Ellipsis{$\pi'''$}{\bot}
      \Infer3[\scriptsize$\lor_\Elim^1$]{\bot}
      \Infer1[\scriptsize$\lnot_\Intro^2$]{\lnot(A \lor B)}
      \Ellipsis{$\pi'$}{\; \bot \;}
      \Infer1[\scriptsize$\Raa^3$]{C}
      \Ellipsis{$\pi$}{}
    \end{prooftree}
    \!\!\!\!\!\!\!\!\!\!\!\!\! = \Pi'
  \end{align}
  }
  \end{subequations}

  \item[$\to$ introduction:]
  {\small
  \begin{align}\label{eq:tointro}
    \Pi = 
    \begin{prooftree}[label separation = 0.2em]
      \Hypo{\ulcorner A \urcorner^2, \ulcorner \lnot B \urcorner^1}
      \Ellipsis{$\pi'$}{\; \bot \;}
      \Infer1[\footnotesize$\Raa^1$]{B}
      \Infer1[\footnotesize$\to_\Intro^2$]{A \to B}
      \Ellipsis{$\pi$}{}
    \end{prooftree}
      \quad
      &\leadsto
      \quad
    \begin{prooftree}[label separation = 0.2em]
      \Hypo{\ulcorner \lnot (A \to B) \urcorner^3}
	\Hypo{\ulcorner A \urcorner^2,}
	  \Hypo{\ulcorner \lnot (A \to B) \urcorner^3}
	    \Hypo{\ulcorner B \urcorner^1}
	    \Infer1[\footnotesize$\to_\Intro^0$]{A \to B}
	  \Infer[separation = 0.7em]2[\footnotesize$\lnot_\Elim$]{\bot}
	\Infer1[\footnotesize$\lnot_\Intro^1$]{\lnot B}
	\Infer[no rule,separation=-1.2em]2{}
	\Ellipsis{$\pi'$}{\; \bot \;}
	\Infer1[\footnotesize$\Efq$]{B}
	\Infer1[\footnotesize$\to_\Intro^2$]{A \to B}
      \Infer[separation = -0.2em]2[\footnotesize$\lnot_\Elim$]{\bot}
      \Infer1[\footnotesize$\Raa^3$]{A \to B}
      \Ellipsis{$\pi$}{}
    \end{prooftree}
    \!\!\!\!\!\!\! = \Pi'
  \end{align}
  }
\label{tointroefq}

  \item[$\to$ elimination:]
  \begin{subequations}\label{eq:toelimgeneral}
  {\small
  \begin{align}
    \Pi = \!
    \begin{prooftree}[label separation = 0.2em]
      \Hypo{\ulcorner \lnot (A \to B) \urcorner^1}
      \Ellipsis{$\pi'$}{\; \bot \;}
      \Infer1[\footnotesize$\Raa^1$]{A \to B}
      \Hypo{}
      \Ellipsis{$\pi''$}{A}
      \Infer[separation = 1em]2[\footnotesize$\to_\Elim$]{ B}
      \Ellipsis{$\pi$}{}
    \end{prooftree}
    \quad
    &\leadsto
    \quad
    \begin{prooftree}[label separation = 0.2em]
      \Hypo{\ulcorner \lnot B \urcorner^2}
      \Hypo{\ulcorner A \to B \urcorner^1}
      \Hypo{}
      \Ellipsis{$\pi''$}{A}
      \Infer[separation = 1em]2[\footnotesize$\to_\Elim$]{B}      
      \Infer[separation = 1em]2[\footnotesize$\lnot_\Elim$]{\bot}
      \Infer1[\footnotesize$\lnot_\Intro^1$]{\lnot (A \to B)}
      \Ellipsis{$\pi'$}{\; \bot \;}
      \Infer1[\footnotesize$\Raa^2$]{B}
      \Ellipsis{$\pi$}{}
    \end{prooftree}
    \!\!\! = \Pi'
  \end{align}
  }where the last rule of the derivation $\pi''$ is not an instance of the rule $\Raa$;
    
  {\small
  \begin{align}
    \Pi = 
    \begin{prooftree}[label separation = 0.3em]
      \Hypo{}
      \Ellipsis{$\pi'$}{A \to B}
      \Hypo{\ulcorner \lnot A \urcorner^1}
      \Ellipsis{$\pi''$}{\; \bot \;}
      \Infer1[\footnotesize$\Raa^1$]{A}
      \Infer[separation = 1em]2[\footnotesize$\to_\Elim$]{B}
      \Ellipsis{$\pi$}{}
    \end{prooftree}
    \qquad
    &\leadsto
    \qquad
    \begin{prooftree}[label separation = 0.3em]
      \Hypo{\ulcorner \lnot B \urcorner^2}
      \Hypo{}
      \Ellipsis{$\pi'$}{A \to B}
      \Hypo{\ulcorner A \urcorner^1}
      \Infer2[\footnotesize$\to_\Elim$]{B}      
      \Infer[separation = 1em]2[\footnotesize$\lnot_\Elim$]{\bot}
      \Infer1[\footnotesize$\lnot_\Intro^1$]{\lnot A}
      \Ellipsis{$\pi''$}{\; \bot \;}
      \Infer1[\footnotesize$\Raa^2$]{B}
      \Ellipsis{$\pi$}{}
    \end{prooftree}
    = \Pi'
    \end{align}
    }
    where the last rule of the derivation $\pi'$ is not an instance of the rule $\Raa$;
    
    {\small
    \begin{align}\label{eq:toelim}
    \Pi = 
    \begin{prooftree}[label separation = 0.2em]
      \Hypo{\ulcorner \lnot (A \to B) \urcorner^1}
      \Ellipsis{$\pi'$}{\; \bot \;}
      \Infer1[\footnotesize$\Raa^1$]{A \to B}
      \Hypo{\ulcorner \lnot A \urcorner^2}
      \Ellipsis{$\pi''$}{\; \bot \;}
      \Infer1[\footnotesize$\Raa^2$]{A}
      \Infer[separation = 0.7em]2[\footnotesize$\to_\Elim$]{B}
      \Ellipsis{$\pi$}{}
    \end{prooftree}
    \quad\ 
    &\leadsto
    \quad
    \begin{prooftree}[label separation = 0.2em]
      \Hypo{\ulcorner \lnot B \urcorner^3}
      \Hypo{\ulcorner A \to B \urcorner^2}
      \Hypo{\ulcorner A \urcorner^1}
      \Infer[separation = 0.8em]2[\footnotesize$\to_\Elim$]{B}      
      \Infer[separation = 0.5em]2[\footnotesize$\lnot_\Elim$]{\bot}
      \Infer1[\footnotesize$\lnot_\Intro^1$]{\lnot A}
      \Ellipsis{$\pi''$}{\; \bot \;}
      \Infer1[\footnotesize$\lnot_\Intro^2$]{\lnot (A \to B)}
      \Ellipsis{$\pi'$}{\; \bot \;}
      \Infer1[\footnotesize$\Raa^3$]{B}
      \Ellipsis{$\pi$}{}
    \end{prooftree}
    \!\!\!\!\!\! = \Pi'
  \end{align}
  }
  \end{subequations}
  
  \item[$\exists$ introduction:]
  {\small
  \begin{align}
    \Pi = 
    \begin{prooftree}
      \Hypo{\ulcorner \lnot A\Sub{t}{x} \urcorner^1}
      \Ellipsis{$\pi'$}{\bot}
      \Infer1[\footnotesize$\Raa^1$]{A\Sub{t}{x}}
      \Infer1[\footnotesize$\exists_\Intro$]{\exists x A}
      \Ellipsis{$\pi$}{}
    \end{prooftree}
    \qquad
    &\leadsto
    \qquad
    \begin{prooftree}
      \Hypo{\ulcorner \lnot \exists x A \urcorner^2}
       \Hypo{\ulcorner A\Sub{t}{x} \urcorner^1}
     \Infer1[\footnotesize$\exists_\Intro$]{\exists x A}
     \Infer2[\footnotesize$\lnot_\Elim$]{\bot}
     \Infer1[\footnotesize$\lnot_\Intro^1$]{\lnot A\Sub{t}{x}}
     \Ellipsis{$\pi'$}{\bot}
      \Infer1[\footnotesize$\Raa^2$]{\exists  x A}
      \Ellipsis{$\pi$}{}      
    \end{prooftree}
    = \Pi'
  \end{align}
  }
 
  \item[$\exists$ elimination:]
  \begin{subequations}
  {\small
  \begin{align}
    \Pi = 
    \begin{prooftree}[label separation = .3em]
	\Hypo{\ulcorner \lnot \exists x A \urcorner^1}
	\Ellipsis{$\pi'$}{\bot}
      \Infer1[\footnotesize$\Raa^1$]{\exists x A}
	\Hypo{\ulcorner A \urcorner^2}
	\Ellipsis{$\pi''$}{C}
      \Infer[separation = 1em]2[\footnotesize$\exists_\Elim^2$]{C}
      \Ellipsis{$\pi$}{}
    \end{prooftree}
    \qquad
    &\leadsto
    \qquad
    \begin{prooftree}[label separation = .3em]
	\Hypo{\ulcorner  \exists x A \urcorner^2}
	  \Hypo{\ulcorner \lnot C \urcorner^3}
	  \Hypo{\ulcorner A \urcorner^1}
	\Ellipsis{$\pi''$}{C}
      \Infer[separation = 1em]2[\footnotesize$\lnot_\Elim$]{\bot}
     \Infer[separation = 1em]2[\footnotesize$\exists_\Elim^1$]{\bot}
     \Infer1[\footnotesize$\lnot_\Intro^2$]{\lnot \exists x A}
     \Ellipsis{$\pi'$}{\; \bot \;}
      \Infer1[\footnotesize$\Raa^3$]{C}
      \Ellipsis{$\pi$}{}      
    \end{prooftree}
    = \Pi'
  \end{align}
  }where the last rule of the derivation $\pi''$ is not an instance of the rule $\Raa$ (notice that the variable $x$ does not occur free in $C$ and hence even in $\lnot C$, thus the rule $\exists_\Elim$ is correctly instantiated in $\Pi'$);

  {\small
  \begin{align}
    \Pi = 
    \begin{prooftree}
	\Hypo{}
	\Ellipsis{$\pi'$}{\exists x A}
	\Hypo{\ulcorner \lnot C \urcorner^1, \ulcorner A \urcorner^2}
	\Ellipsis{$\pi''$}{\; \bot \;}
      \Infer1[\footnotesize$\Raa^1$]{C}
      \Infer2[\footnotesize$\exists_\Elim^2$]{C}
      \Ellipsis{$\pi$}{}
    \end{prooftree}
    \qquad
    &\leadsto
    \qquad
    \begin{prooftree}
	  \Hypo{}
	\Ellipsis{$\pi'$}{\exists x A}
	  \Hypo{\ulcorner \lnot C \urcorner^2, \ulcorner A \urcorner^1}
	\Ellipsis{$\pi''$}{\bot}
     \Infer2[\footnotesize$\exists_\Elim^1$]{\; \bot \;}
      \Infer1[\footnotesize$\Raa^2$]{C}
      \Ellipsis{$\pi$}{}      
    \end{prooftree}
    = \Pi'
  \end{align}
  }where the last rule of the derivation $\pi'$ is not an instance of the rule $\Raa$ (notice that the variable $x$ does not occur free in $C$ and hence even in $\lnot C$, thus the rule $\exists_\Elim$ is correctly instantiated in $\Pi'$);

  {\small
  \begin{align}
    \Pi = 
    \begin{prooftree}
	\Hypo{\ulcorner \lnot \exists x A \urcorner^1}
	\Ellipsis{$\pi'$}{\bot}
      \Infer1[\footnotesize$\Raa^1$]{\exists x A}
	\Hypo{\ulcorner \lnot C \urcorner^2, \ulcorner A \urcorner^3}
	\Ellipsis{$\pi''$}{\; \bot \;}
      \Infer1[\footnotesize$\Raa^2$]{C}
      \Infer[separation = .9em]2[\footnotesize$\exists_\Elim^3$]{C}
      \Ellipsis{$\pi$}{}
    \end{prooftree}
    \qquad
    &\leadsto
    \qquad
    \begin{prooftree}
	\Hypo{\ulcorner  \exists x A \urcorner^2}
	  \Hypo{\ulcorner \lnot C \urcorner^3, \ulcorner A \urcorner^1}
	\Ellipsis{$\pi''$}{\bot}
     \Infer[separation = .4em]2[\footnotesize$\exists_\Elim^1$]{\bot}
     \Infer1[\footnotesize$\lnot_\Intro^2$]{\lnot \exists x A}
     \Ellipsis{$\pi'$}{\; \bot \;}
      \Infer1[\footnotesize$\Raa^3$]{C}
      \Ellipsis{$\pi$}{}      
    \end{prooftree}
    = \Pi'
  \end{align}
  }(notice that the variable $x$ does not occur free in $C$ and hence even in $\lnot C$, thus the rule $\exists_\Elim$ is correctly instantiated in $\Pi'$);

  \end{subequations}
  
  \item[$\forall$ elimination:]
  {\small
  \begin{align}
    \Pi = 
    \begin{prooftree}
      \Hypo{\ulcorner \lnot \forall x A \urcorner^1}
      \Ellipsis{$\pi'$}{\bot}
      \Infer1[\footnotesize$\Raa^1$]{\forall x A}
      \Infer1[\footnotesize$\forall_\Elim$]{A\Sub{t}{x}}
      \Ellipsis{$\pi$}{}
    \end{prooftree}
    \qquad
    &\leadsto
    \qquad
    \begin{prooftree}
      \Hypo{\ulcorner \lnot A\Sub{t}{x} \urcorner^2}
        \Hypo{\ulcorner \forall x A \urcorner^1}
      \Infer1[\footnotesize$\forall_\Elim$]{A\Sub{t}{x}}
      \Infer2[\footnotesize$\lnot_\Elim$]{\bot}
      \Infer1[\footnotesize$\lnot_\Intro^1$]{\lnot \forall x A}
      \Ellipsis{$\pi'$}{\bot}
      \Infer1[\footnotesize$\Raa^2$]{A\Sub{t}{x}}
      \Ellipsis{$\pi$}{}
    \end{prooftree}
    = \Pi'
  \end{align}
  }
  
  \item[$\Efq$ and $\Raa$:]
  \begin{subequations}
  {\small
  \begin{align}
    \Pi = 
    \begin{prooftree}\label{eq:eqf}
      \Hypo{\ulcorner \lnot \bot \urcorner^1}
      \Ellipsis{$\pi'$}{\bot}
      \Infer1[\footnotesize$\Raa^1$]{\; \bot \;}
      \Infer1[\footnotesize$\Efq$]{B}
      \Ellipsis{$\pi$}{}
    \end{prooftree}
    \qquad
    &\leadsto
    \qquad
    \begin{prooftree}
      \Hypo{\ulcorner \bot \urcorner^1}
      \Infer1[\footnotesize$\lnot_\Intro^1$]{\lnot \bot}
      \Ellipsis{$\pi'$}{\; \bot \;}
      \Infer1[\footnotesize$\Efq$]{B}
      \Ellipsis{$\pi$}{}
    \end{prooftree}
    = \Pi'
    \\ \notag\\
    \Pi = 
    \begin{prooftree}\label{eq:raa}
      \Hypo{\ulcorner \lnot B \urcorner^2, \ulcorner \lnot \bot \urcorner^1}
      \Ellipsis{$\pi'$}{\bot}
      \Infer1[\footnotesize$\Raa^1$]{\; \bot \;}
      \Infer1[\footnotesize$\Raa^2$]{B}
      \Ellipsis{$\pi$}{}
    \end{prooftree}
    \qquad
    &\leadsto
    \qquad
    \begin{prooftree}
      \Hypo{\ulcorner \lnot B \urcorner^2,}
	\Hypo{\ulcorner \bot \urcorner^1}
      \Infer1[\footnotesize$\lnot_\Intro^1$]{\lnot \bot}
      \Infer[no rule,separation=0.2em]2{}
      \Ellipsis{$\pi'$}{\; \bot \;}
      \Infer1[\footnotesize$\Raa^2$]{B}
      \Ellipsis{$\pi$}{}
    \end{prooftree}
    = \Pi'
  \end{align}
  }
  \end{subequations}

\end{description}

\begin{note}
  For all derivations $\Pi$ and $\Pi'$ in $\Nk$, we write $\Pi \leadsto \Pi'$ if $\Pi'$ is obtained from $\Pi$ by applying one of the reduction steps listed above.
  The reflexive-transitive closure of 
  $\leadsto$ is denoted by $\leadsto^*$.
  
  Given $\Pi \leadsto \Pi'$, we say that each instance of $\Raa$ in $\Pi$ that is explicitly represented in the left-hand side of any reduction step listed above\,---\,with the exception of the reduction step \eqref{eq:raa}\,---\,is \textit{active}.
  Concerning the reduction step \eqref{eq:raa} (\Ie~the $\Raa$ case), if $\mathsf{r}_1$ and $\mathsf{r}_2$ are the two instances of $\Raa$ in $\Pi$ that are explicitly represented in the left-hand side of the reduction step, and if $\mathsf{r}_1$ (resp.~$\mathsf{r}_2$) is the instance whose conclusion is $\bot$ (resp.~$B$), then~only~$\mathsf{r}_1$~is~\emph{active}.
  
  If $\Pi \leadsto \Pi'$ and $\mathsf{r}_1, \dots, \mathsf{r}_n$ are the active instances of $\Raa$ in $\Pi$, we will write $\Pi \overset{\mathsf{r}_1, \dots, \mathsf{r}_n}{\leadsto} \Pi'$. According to the reduction rules listed above, $n \in \{1, 2, 3\}$.
\end{note}

It is easy to check that, for each reduction step listed above, there is at least one active instance of the rule $\Raa$ in $\Pi$.
These reduction steps might involve some non-local modifications over derivations. For example, when $\Pi \leadsto \Pi'$, a subderivation of $\Pi$ might be erased or duplicated in $\Pi'$, depending on the number of assumptions that are discharged by the active instances of $\Raa$ in $\Pi \leadsto \Pi'$. Moreover, some subderivations of $\Pi$ can be moved in $\Pi'$ above some other subderivations of $\Pi$ (this corresponds to an operation of proof composition).
 
When 
$\Pi \leadsto \Pi'$\!, there is no reduction step which introduces in $\Pi'$ a new discharging instance of $\Raa$: any discharging instance of $\Raa$ in $\Pi'$\! can thus be seen as a ``residual'' of an instance of $\Raa$ in $\Pi$ (possibly non-discharging or applied to another formula).
However, it is not true that any instance of $\Raa$ in $\Pi$ has a 
residual in $\Pi'$; for example in the reduction steps 
\eqref{eq:notelim}, \eqref{eq:eqf} or \eqref{eq:raa} the active instance of $\Raa$ in $\Pi$ vanishes in $\Pi'$.

\begin{rmk}\label{rmk:tointro}
  The case where $\Pi \leadsto \Pi'$ by applying the reduction step \eqref{eq:tointro} (\Ie~the $\to_\Intro$ case) is the only one introducing in $\Pi'$ a new instance of $\Efq$: 
  the instance of $\Efq$ in $\Pi'$ explicitly represented in the right-hand side of the reduction step 
  \eqref{eq:tointro} is not a 
  residual 
  of any instance of $\Efq$ in $\Pi$.
  In a way, it is impossible to avoid adding an instance of $\Efq$ in the $\to_\Intro$ case: this is deeply related to the fact that $(A \to \lnot \lnot B) \to \lnot \lnot(A \to B)$ is provable in $\Nj$ but not in $\Nm$.
  Indeed, if it were possible to define the following reduction step (where in the subderivations $\pi$ and $\pi'$ there is no instance of $\Raa$, and the formulas occurring in the non-discharged assumptions of $\Pi'$ are a subset the formulas occurring the non-discharged assumptions of $\Pi$)
  
  {\small
  \begin{align*}
    \Pi = 
    \begin{prooftree}
      \Hypo{\ulcorner A \urcorner^2, \ulcorner \lnot B \urcorner^1}
      \Ellipsis{$\pi$}{\; \bot \;}
      \Infer1[\footnotesize$\Raa^1$]{B}
      \Infer1[\footnotesize$\to_\Intro^2$]{A \to B}
    \end{prooftree}
      \qquad
      &\leadsto
      \qquad
    \begin{prooftree}
      \Hypo{\ulcorner \lnot (A \to B) \urcorner^2}
	\Hypo{\ulcorner A \urcorner^1}
	\Ellipsis{$\pi'$}
	{B}
	\Infer1[\footnotesize$\to_\Intro^1$]{A \to B}
      \Infer2[\footnotesize$\lnot_\Elim$]{\bot}
      \Infer1[\footnotesize$\Raa^2$]{A \to B}
    \end{prooftree}
    = \Pi'
  \end{align*}
  }
  
  \noindent then, by replacing the instances of $\Raa$ in $\Pi$ and $\Pi'$ by instances of $\lnot_\Intro$, the conclusions of $\Pi$ and $\Pi'$ would be $A \to \lnot \lnot B$ and $\lnot \lnot (A \to B)$, respectively: this would mean that the derivability of $A \to \lnot \lnot B$ in $\Nm$ would imply the derivability of $\lnot \lnot (A \to B)$ in $\Nm$, which is impossible.
\end{rmk}

\begin{rmk}\label{rmk:forall}
  It is easy to check that if $\Pi$ is a derivation in $\Nk$ and $\Pi \leadsto \Pi'$, then $\Pi'$ is a derivation in $\Nk$ and $\Pi$ is not a derivation in $\Nm$ ($\Pi$ contains at least one instance of $\Raa$, discharging or not discharging).
  In the reduction steps listed above, the $\forall_\Intro$ case (where in $\Pi$ an instance of the rule $\forall_\Intro$ is immediately below the 
  instance of $\Raa$ under focus) is absent, otherwise $\Pi'$ could not be a derivation in $\Nk$, as we have pointed out in \S\ref{subsect:seldinstrategy}. 
  In other words, if $\Pi \overset{\mathsf{r}}{\leadsto} \Pi'$ then $\mathsf{r}$ is not an instance of $\Raa$ in $\Pi$ whose conclusion is the premise of an instance of $\forall_\Intro$. 
\end{rmk}

\begin{rmk}\label{rmk:non-conf}
The reduction steps \eqref{eq:notelim}, \eqref{eq:andelim} and \eqref{eq:toelim} possess a certain degree of arbitrariness since they could also be defined 
so that, in $\Pi'$\!, the subderivation $\pi'$ would be put above 
$\pi''$\!, and not vice-versa.
\end{rmk}

Observe the similarities between the reduction steps in the $\land_\Intro$ and $\to_\Elim$ cases, or in the $\land_\Elim$, $\lor_\Intro$ and $\forall_\Elim$ cases, or in the $\lor_\Elim$ and $\exists_\Elim$ cases.
In contrast, the reduction steps for the $\lnot_\Intro$ and $\to_\Intro$ cases (resp.\ $\lnot_\Elim$ and $\to_\Elim$ cases) are rather different: 
the former -- \eqref{eq:notintro} (resp.~\eqref{eq:notelimgeneral}) -- erases an instance of $\Raa$, whereas the latter -- \eqref{eq:tointro} (resp.~\eqref{eq:toelimgeneral}) -- postpones an instance of $\Raa$ after an instance of $\to_\Intro$ 
(resp.~$\to_\Elim$), moreover in the $\to_\Intro$ case the reduction step introduces a new instance of $\Efq$, unlike the $\lnot_\Intro$ case. 
This differing behavior justifies our choice to consider $\lnot$ as primitive and not to treat $\lnot A$ as a shorthand for~$A \to \bot$.


%
\section{\texorpdfstring{Postponement of $\Raa$}{Postponement of raa}}
\label{sect:postponement}

In this section we prove the first main result of this paper: the postponement of $\Raa$ (Theorem~\ref{thm:postponement}, Corollary~\ref{cor:postponement}), \Ie~the fact that it is possible to transform a derivation in $\Nk$ in such a way that any possible instance of $\Raa$ is pushed downward until it vanishes or it occurs only in the last rule, preserving the same conclusion and without adding any new non-discharged assumptions. 
More precisely, we show that, by repeated applications of the reduction steps of \S\ref{sect:rewrite} following a suitable strategy
:
\begin{itemize}
  \item a derivation $\pi$ in $\Nk$ without the rule $\forall_\Intro$ reduces all its instances of $\Raa$ to at most one instance of $\Raa$ occurring as the last rule, the rest of the derivation being in $\Nj$ (Theorem~\ref{thm:postponement}.\ref{thm:postponement.intuitionistic});

  \item a derivation $\pi$ in $\Nk$ without the rules $\forall_\Intro$ and $\to_\Intro$ reduces all its instances of $\Raa$ to at most one instance of $\Raa$ occurring as the last rule, the rest of the derivation being in $\Nm$ (Theorem~\ref{thm:postponement}.\ref{thm:postponement.minimal}).
\end{itemize}

\noindent This result will be then reformulated by focusing on the form of the conclusion and of the non-discharged assumptions of $\pi$, rather than on the kind of inference rules used in $\pi$ (Corollary~\ref{cor:postponement}).


Two lemmas are used in 
the proof of Theorem~\ref{thm:postponement}. 
The first one (Lemma~\ref{lemma:preservation}) says that when the reduction steps of \S\ref{sect:rewrite} are applied, no new non-discharged assumptions are added in the reduced derivation, and the conclusion of the original derivation is preserved. Moreover, the reduction steps neither introduce nor erase any instance of the rules $\to_\Intro$ and $\forall_\Intro$.
The second one (Lemma~\ref{lemma:size}) says that 
the $\Size_\RAA$ (resp.~$\Size_\RAAp$) of a derivation strictly decreases when one applies a reduction rule whose active instance of $\Raa$ is a $\RAA$-maximal (resp.~$\RAAp$-maximal) instance.

\begin{lemma}[Preservations]\label{lemma:preservation}
  Let $\Pi$ and $\Pi'$ be derivations in $\Nk$ with $\Pi \leadsto \Pi'$. 
  \begin{enumerate}
    \item\label{lemma:preservation.subject} If $\Pi \colon \Gamma \vdash A$, then $\Pi' \colon \Gamma \vdash A$;
    
    \item\label{lemma:preservation.connective} If $\Pi$ has no 
    instance of the rule $\to_\Intro$ (resp.~
    $\forall_\Intro$), 
    then $\Pi'$ has no 
    instance of the rule $\to_\Intro$ (resp.~
    $\forall_\Intro$).
  \end{enumerate}
\end{lemma}

\begin{proof}
  By straightforward inspection of the reduction steps listed in \S\ref{sect:rewrite}. 
\end{proof}

The converse of Lemma~\ref{lemma:preservation}.\ref{lemma:preservation.subject} does not hold, as shown by the following counterexample (taking $\Gamma = \{\neg P, P\}$ and $A = Q \wedge R$, where $P$, $Q$ and $R$ are distinct proposition symbols):
\begin{align}\label{eq:counterexample}
  \Pi = 
  \begin{prooftree}
    \Hypo{\lnot \Var}
    \Hypo{\Var}
    \Infer2[\footnotesize$\lnot_\Elim$]{\; \bot \;}
    \Infer1[\footnotesize$\Efq$]{\Vartwo}
    \Hypo{\Varthree}
    \Infer2[\footnotesize$\land_\Intro$]{\Vartwo \land \Varthree}
  \end{prooftree}
  \qquad
  &\leadsto
  \qquad
  \begin{prooftree}
    \Hypo{ \lnot \Var }
    \Hypo{\Var}
    \Infer2[\footnotesize$\lnot_\Elim$]{\bot}
    \Infer1[\footnotesize$\Efq$]{\Vartwo \land \Varthree}
  \end{prooftree}
  = \Pi'
\end{align}

\noindent Hence, in \eqref{eq:counterexample} we have $\Pi' : \Gamma \vdash A$ but we do not have $\Pi : \Gamma \vdash A$, since the set of non-discharged assumptions of $\Pi$ also contains (an occurrence of) the formula $R$ (see point \ref{note:assumption} about notations at p.~\pageref{note:assumption}). 
In other words, Lemma~\ref{lemma:preservation}.\ref{lemma:preservation.subject} says that if $\Pi \leadsto \Pi'$ then the formulas occurring among the non-discharged assumptions of $\Pi'$ are a \emph{subset} of the formulas occurring among the non-discharged assumptions of $\Pi$.
  
Derivability in $\Nj$ is not preserved by the reduction steps of \S\ref{sect:rewrite}: the fact that $\Pi$ is a derivation in $\Nj$ and $\Pi \leadsto \Pi'$ does not always imply that $\Pi'$ is a derivation in $\Nj$, because an instance of the rule $\Efq$ in $\Pi$ can be transformed into a discharging instance of the rule $\Raa$ in $\Pi'$. Consider, for example, the following situation:

{\small
\begin{equation*}
  \Pi = 
  \begin{prooftree}[label separation = 0.2em]
    \Hypo{\lnot \Var \lor \Vartwo}
      \Hypo{\ulcorner \lnot \Var \urcorner^1}
      \Hypo{\Var}
      \Infer[separation = 0.8em]2[\footnotesize$\lnot_{\Elim}$]{\; \bot \;}
    \Infer1[\footnotesize$\Efq$]{\Vartwo}
    \Hypo{\ulcorner \Vartwo \urcorner^1}
    \Infer[separation = -0.5em]3[\footnotesize$\lor_\Elim^1$]{\Vartwo}
  \end{prooftree}
  \ \leadsto \quad\!
  \begin{prooftree}[label separation = 0.2em]
    \Hypo{\lnot \Var \lor \Vartwo}
      \Hypo{\ulcorner \lnot \Var \urcorner^1}
      \Hypo{\Var}
    \Infer[separation = 0.8em]2[\footnotesize$\lnot_{\Elim}$]{\bot}
      \Hypo{\ulcorner \lnot \Vartwo \urcorner^2}
      \Hypo{\ulcorner \Vartwo \urcorner^1}
    \Infer[separation = 0.8em]2[\footnotesize$\lnot_{\Elim}$]{\bot}
    \Infer[separation = 0.5em]3[\footnotesize$\lor_\Elim^1$]{\; \bot \;}
    \Infer1[\footnotesize$\Raa^2$]{\Vartwo}
  \end{prooftree}
  \!= \Pi'.
\end{equation*}
}

\begin{lemma}[Size decreasing]\label{lemma:size}
  Let $\Pi$ and $\Pi'$ be derivations in $\Nk$ such that $\Pi \overset{\mathsf{r}_1, \dots, \mathsf{r}_n}{\leadsto} \Pi'$ where $n \in \Natp$. 
  \begin{enumerate}
    \item\label{lemma:size.maximal} If $\mathsf{r}_i$ is $\RAA_{\Pi}$-maximal (resp.~$\RAAp_{\Pi}$-maximal) for some $1 \leq i \leq n$, then $\mathsf{r}_j$ is $\RAA_{\Pi}$-maximal (resp.~$\RAAp_{\Pi}$-maximal) for all $1 \leq j \leq n$.
  
    \item\label{lemma:size.intuitionistic} If $\Pi$ contains no instance of the rule $\forall_\Intro$, and if $\mathsf{r}_j$ is $\RAAp_{\Pi}$-maximal for some $1 \leq j \leq n$, then $\Size_{\RAAp}(\Pi') < \Size_{\RAAp}(\Pi)$.

    \item\label{lemma:size.minimal} If $\Pi$ contains no 
    instance of the rules $\to_\Intro $ and 
    $ \forall_\Intro$, and if $\mathsf{r}_j$ is $\RAA_{\Pi}$-maximal for some $1 \leq j \leq n$, then $\Size_{\RAA}(\Pi') < \Size_{\RAA}(\Pi)$.
  \end{enumerate}

\end{lemma}

\begin{proof}
  By straightforward inspection of all the reduction steps listed in \S\ref{sect:rewrite}.
  The hypothesis of maximality for the active instances of $\Raa$ is crucial: it ensures that when a subderivation is moved above another subderivation according to the reduction step, no instance of $\Raa$ is moved away from the conclusion of the derivation.
\end{proof}

Note that, in the proof of Lemma~\ref{lemma:size}, the complexity of the formulas occurring in the conclusions of the instances of $\Raa$ play no role (see Definition~\ref{def:distance}): the fact that the sizes $\Size_{\RAA}$ and $\Size_{\RAAp}$ decrease applying the reduction steps listed in \S\ref{sect:rewrite} is purely ``geometrical'', due to the decrease of the distance of a maximal $\Raa$ from the conclusion of the derivation.

Lemma~\ref{lemma:size}.\ref{lemma:size.minimal} becomes false if $\Pi$ contains an instance of the rule $\to_\Intro$, as shown by the following counterexample (see also Remark~\ref{rmk:tointro}):
{\small
\begin{align*}
  \Pi = 
  \begin{prooftree}
    \Hypo{\ulcorner \Var \urcorner^2}
    \Hypo{\ulcorner \lnot \Var \urcorner^1}
    \Infer2[\footnotesize$\lnot_\Elim$]{\; \bot \;}
    \Infer1[\footnotesize$\Raa^1$]{\Var}
    \Infer1[\footnotesize$\to_\Intro^2$]{\Var \to \Var}
  \end{prooftree}
    \quad
    &\leadsto
    \quad
  \begin{prooftree}
    \Hypo{\ulcorner \lnot (\Var \to \Var) \urcorner^3}
      \Hypo{\ulcorner \Var \urcorner^2}
	\Hypo{\ulcorner \lnot (\Var \to \Var) \urcorner^3}
	  \Hypo{\ulcorner \Var \urcorner^1}
	  \Infer1[\footnotesize$\to_\Intro^0$]{\Var \to \Var}
	\Infer2[\footnotesize$\lnot_\Elim$]{\bot}
      \Infer1[\footnotesize$\lnot_\Intro^1$]{\lnot \Var}
      \Infer[separation = -0.5em]2[\footnotesize$\lnot_\Elim$]{\; \bot \;}
      \Infer1[\footnotesize$\Efq$]{\Var}
      \Infer1[\footnotesize$\to_\Intro^2$]{\Var \to \Var}
    \Infer[separation = -0.5em]2[\footnotesize$\lnot_\Elim$]{\bot}
    \Infer1[\footnotesize$\Raa^3$]{\Var \to \Var}
  \end{prooftree}
  \!= \Pi'
\end{align*}
}
where $\Size_\RAA(\Pi) = 1 < 3 = \Size_\RAA(\Pi')$, since the instance of $\Efq$ in $\Pi'$ belongs to $\RAA_{\Pi'}$.

\begin{thm}[Postponement of $\Raa$, version 1]\label{thm:postponement}
  Let $\Pi \colon \Gamma \vdash A$ be a derivation in $\Nk \smallsetminus \{\forall_\Intro\}$. 
  
  \begin{enumerate}
    \item\label{thm:postponement.intuitionistic} 
    One has $\Pi \leadsto^* \Pi'$ for some \Jstandard\ derivation $\Pi' \colon \Gamma \vdash A$ in $\Nk \smallsetminus \{\forall_\Intro\}$. 
    \item\label{thm:postponement.minimal} If $\Pi$ contains no 
    instance of the rule $\to_\Intro$, then there exists 
    some \Mstandard\ derivation $\Pi' \colon \Gamma \vdash A$ in $\Nk \smallsetminus \{\to_\Intro, \forall_\Intro\}$ such that $\Pi \leadsto^* \Pi'$.
  \end{enumerate}
\end{thm}

\begin{proof}\hfill
  \begin{enumerate}
    \item By induction on $\Size_{\RAAp}(\Pi) \in \Nat$. 
    
    If $\Size_{\RAAp}(\Pi) = 0$, then just take $\Pi' = \Pi$, according to Remark~\ref{rmk:size0}.\ref{rmk:size0.intuitionistic}.
    
    Otherwise, $\Size_{\RAAp}(\Pi) > 0$ and there exists $\mathsf{r} \in \RAAp_\Pi$ which is $\RAAp_{\Pi}$-maximal and which is not the last rule of $\Pi$. 
    Since there is no instance of the rule $\forall_\Intro$ in $\Pi$, there necessarily exists a $\Pi'$ such that $\Pi \overset{\mathsf{r}_1, \dots, \mathsf{r}_n}{\leadsto} \Pi'$ where $n \in \Natp$ and $\mathsf{r} = \mathsf{r}_j$ for some $1 \leq j \leq n$.
    According to Lemma~\ref{lemma:size}.\ref{lemma:size.maximal}, all $\mathsf{r}_1, \dots, \mathsf{r_n}$ are $\RAAp_{\Pi}$-maximal.
    By Lemma~\ref{lemma:preservation}.\ref{lemma:preservation.subject}, $\Pi' \colon \Gamma \vdash A$, and, by Lemma~\ref{lemma:preservation}.\ref{lemma:preservation.connective}, $\Pi'$ has no instance of the rule $\forall_\Intro$.
    According to Lemma~\ref{lemma:size}.\ref{lemma:size.intuitionistic}, $\Size_{\RAAp}(\Pi') < \Size_{\RAAp}(\Pi)$. 
    Hence, by the induction hypothesis, there is a 
    \Jstandard\ derivation $\Pi'' \colon \Gamma \vdash A$ in $\Nk \smallsetminus \{\forall_\Intro\}$ such that $\Pi \leadsto \Pi' \leadsto^* \Pi''$.
    
    \item By induction on $\Size_{\RAA}(\Pi) \in \Nat$. 
    
    If $\Size_{\RAA}(\Pi) = 0$, then just take $\Pi' = \Pi$, according to Remark~\ref{rmk:size0}.\ref{rmk:size0.minimal}.
    
    Otherwise, $\Size_{\RAA}(\Pi) > 0$ and there exists $\mathsf{r} \in \RAA_\Pi$ which is $\RAA_{\Pi}$-maximal and which is not the last rule of $\Pi$. 
    Since there is no instance of the rule $\forall_\Intro$ in $\Pi$, there necessarily exists a $\Pi'$ such that $\Pi \overset{\mathsf{r}_1, \dots, \mathsf{r}_n}{\leadsto} \Pi'$, where $n \in \Natp$ and $\mathsf{r} = \mathsf{r}_j$ for some $1 \leq j \leq n$.
    According to Lemma~\ref{lemma:size}.\ref{lemma:size.maximal}, all $\mathsf{r}_1, \dots, \mathsf{r_n}$ are $\RAA_{\Pi}$-maximal.
    By Lemma~\ref{lemma:preservation}.\ref{lemma:preservation.subject}, $\Pi' \colon \Gamma \vdash A$, and, by Lemma~\ref{lemma:preservation}.\ref{lemma:preservation.connective}, $\Pi'$ has no instance of the rules $\forall_\Intro $ and 
    $ \to_\Intro$.
    According to  Lemma~\ref{lemma:size}.\ref{lemma:size.minimal}, $\Size_\RAA(\Pi') < \Size_\RAA(\Pi)$. 
    Hence, by the induction hypothesis, there is a 
    \Mstandard\ derivation $\Pi'' \colon \Gamma \vdash A$ in $\Nk \smallsetminus \{\to_\Intro, \forall_\Intro\}$ such that $\Pi \leadsto \Pi' \leadsto^* \Pi''$.
    \qedhere
  \end{enumerate}

\end{proof}


Theorem~\ref{thm:postponement} can be see as a \emph{weak} standardization theorem:\label{weaknorm} for every derivation $\Pi$ in $\Nk$ (fulfilling suitable conditions), we have shown that there exists a particular strategy for the application of the reduction steps of \S\ref{sect:rewrite} (we fire only maximal instances of $\Raa$) transforming $\Pi$ into a ``standard'' derivation in $\Nk$, where ``standard'' is here understood in the sense of Definition~\ref{def:distance}.
We conjecture that Theorem~\ref{thm:postponement} can be strengthened to a \emph{strong} standardization theorem: whatever strategy in the application of the reduction steps of \S\ref{sect:rewrite} terminates in a ``standard'' derivation in $\Nk$.
To prove that, one should refine the notion of size of a derivation and proceed by a more complex induction.

Thanks to the normalization theorem and the suitable subformula property for $\Nk$ proved by Stålmarck in \cite[pp.~130, 135]{Stalmarck91},\footnote{See also \cite[p.~208]{vonPlatoSiders12}.} we can reformulate Theorem~\ref{thm:postponement}.\ref{cor:postponement.intuitionistic} (resp.~Theorem~\ref{thm:postponement}.\ref{thm:postponement.minimal}) with a more satisfactory hypothesis: instead of supposing that the derivation $\Pi \colon \Gamma \vdash A$ in $\Nk$ is without any instance of the rule $\forall_\Intro$ (resp.~rules $\forall_\Intro$ and $\to_\Intro$), it is sufficient to suppose that $A$ and the formulas in $\Gamma$ do not contain any occurrence of $\forall$ (resp.~$\forall$ and $\to$). 
Note that, since in Stålmark’s normalization strategy for classical logic $\Raa$ is pushed downward only with respect to elimination rules, a proof in classical normal form in his sense can still contain instances of $\Raa$ that can reduced via the reduction steps proposed in \S\ref{sect:rewrite}. 
In this sense the following corollary of Theorem~\ref{thm:postponement} is not completely trivial.


\begin{cor}[Postponement of $\Raa$, version 2]\label{cor:postponement}
  Suppose $\Gamma \vdash_\Nk A$.
  \begin{enumerate}
    \item\label{cor:postponement.intuitionistic} If $A$ and the formulas in $\Gamma$ do not contain any occurrence of $\forall$, then there exists a 
    \Jstandard\ derivation $\Pi' \colon \Gamma \vdash A$ in $\Nk$.

    \item\label{cor:postponement.minimal} If $A$ and the formulas in $\Gamma$ do not contain any occurrence of $\forall$ nor 
    $\to$, then there exists a 
    \Mstandard\ derivation $\Pi' \colon \Gamma \vdash A$ in $\Nk$.
  \end{enumerate}
\end{cor}

\begin{proof}
First, we recall some facts that will be used to prove Corollaries~\ref{cor:postponement}.\ref{cor:postponement.intuitionistic}-\ref{cor:postponement.minimal}.
As $\Gamma \vdash A$ is derivable in $\Nk$, there exists a normal (in the sense of \cite[p.~130]{Stalmarck91}) derivation $\Pi \colon \Gamma \vdash A$ in $\Nk$ and hence, in conformity to the aforementioned \emph{subformula principle} for $\Nk$ (see \cite[p.~130]{Stalmarck91}), each formula occurrence $B$ in $\Pi$ satisfies one of the clauses \ref{point:subformula1}-\ref{point:subformula3} below:
    \begin{romanenum}
      \item\label{point:subformula1} $B$ is an occurrence of a subformula of $A$ or of some formula in $\Gamma$;
      \item\label{point:subformula2} $B$ is an assumption discharged by some instance of the rule $\Raa$, $B$ has the form $\lnot C$, and $C$ is a subformula of $A$ or of some formula in $\Gamma$;
      \item\label{point:subformula3} $B$ has the form $\bot$ and stands immediately below an assumption which satisfies \ref{point:subformula2} above. 
    \end{romanenum}

  We can now prove Corollaries~\ref{cor:postponement}.\ref{cor:postponement.intuitionistic}-\ref{cor:postponement.minimal}.
  \begin{enumerate}
    \item According to the subformula principle, there are no instances of the rule $\forall_\Intro$ in $\Pi$. 
    By Theorem~\ref{thm:postponement}.\ref{thm:postponement.intuitionistic}, 
    there is a \Jstandard\ \mbox{derivation $\Pi' \colon \Gamma \vdash A$ in $\Nk$.}
    
    \item According to the subformula principle, there are no 
    instances of the rule $\to_\Intro$ or the rule 
    $\forall_\Intro$ in $\Pi$. 
    By Theorem~\ref{thm:postponement}.\ref{thm:postponement.minimal}, 
    there is a \Mstandard\ derivation $\Pi' \colon \Gamma \vdash A$ in $\Nk$.
    \qedhere
  \end{enumerate}

\end{proof}

Thus, according to Corollary~\ref{cor:postponement}.\ref{cor:postponement.intuitionistic} (resp.~Corollary~\ref{cor:postponement}.\ref{cor:postponement.minimal}), when we look for a derivation of $A$ from $\Gamma$, where $A$ and all formulas in $\Gamma$ do not contain any occurrence of $\forall$ (resp.~$\forall$ nor 
$\to$), we can consider the use of a discharging (resp.~either discharging or undischarging) instance of $\Raa$ only at the end of the derivation, if this use is required.

In what follows we give two examples of the postponement of $\Raa$, according Theorems~\ref{thm:postponement}.\ref{thm:postponement.intuitionistic} and \ref{thm:postponement}.\ref{thm:postponement.minimal}.

\begin{ex}
  Given two proposition symbols $\Var$ and $\Vartwo$, let $\Pi \colon \vdash (\lnot \lnot \Var \to \Var) \land (\lnot \lnot \Vartwo \to \Vartwo)$ be the following derivation in $\Nk$ (note that there are no instances of $\forall_\Intro$ in $\Pi$):
  \begin{align*}
    \Pi \ &= \ 
    \begin{prooftree}
	  \Hypo{\ulcorner \lnot \lnot \Var \urcorner^2}
	  \Hypo{\ulcorner \lnot \Var \urcorner^1}
	\Infer2[\footnotesize$\lnot_\Elim$]{\; \bot \;}
	\Infer1[\footnotesize$\Raa^1$]{\Var}
      \Infer1[\footnotesize$\to_\Intro^2$]{\lnot \lnot \Var \to \Var}
	  \Hypo{\ulcorner \lnot \lnot \Vartwo \urcorner^4}
	  \Hypo{\ulcorner \lnot \Vartwo \urcorner^3}
	\Infer2[\footnotesize$\lnot_\Elim$]{\; \bot \;}
	\Infer1[\footnotesize$\Raa^3$]{\Vartwo}
      \Infer1[\footnotesize$\to_\Intro^4$]{\lnot \lnot \Vartwo \to \Vartwo}
      \Infer2[\footnotesize$\land_\Intro$]{(\lnot \lnot \Var \to \Var) \land (\lnot \lnot \Vartwo \to \Vartwo)}
    \end{prooftree}
  \end{align*}
  For every proposition symbol $Z$, let $\pi(Z) \colon \lnot (\lnot \lnot Z \to Z) \vdash \bot$ be the following derivation in $\Nj$:
  {\small
  \begin{align*}
    \pi(Z) \ &= \ 
    \begin{prooftree}
	\Hypo{\lnot (\lnot \lnot Z \to Z) }
	  \Hypo{\ulcorner \lnot \lnot Z \urcorner^2}
	    \Hypo{\lnot (\lnot \lnot Z \to Z)}
	      \Hypo{\ulcorner Z \urcorner^1}
	      \Infer1[\footnotesize$\to_\Intro^0$]{\lnot \lnot Z \to Z}
	    \Infer2[\footnotesize$\lnot_\Elim$]{\bot}
	  \Infer1[\footnotesize$\lnot_\Intro^1$]{\lnot Z}
	  \Infer2[\footnotesize$\lnot_\Elim$]{\; \bot \;}
	  \Infer1[\footnotesize$\Efq$]{Z}
	  \Infer1[\footnotesize$\to_\Intro^2$]{\lnot \lnot Z \to Z}
	\Infer2[\footnotesize$\lnot_\Elim$]{\bot}    
    \end{prooftree}
  \end{align*}
  }
  By applying the procedure defined in the proof of Theorem~\ref{thm:postponement}.\ref{thm:postponement.intuitionistic}, we get:
  {\small 
  \begin{align*}
    \Pi \ 
    &\leadsto\!
    \begin{prooftree}[label separation = 0.2em]
	\Hypo{\ulcorner \lnot (\lnot \lnot \Var \to \Var) \urcorner^1}
	\Ellipsis{$\pi(\Var)$}{\bot}
      \Infer1[\scriptsize$\Raa^1$]{\lnot \lnot \Var \to \Var}
	  \Hypo{\ulcorner \lnot \lnot \Vartwo \urcorner^3}
	  \Hypo{\ulcorner \lnot \Vartwo \urcorner^2}
	\Infer[separation = 0.5em]2[\scriptsize$\lnot_\Elim$]{\; \bot \;}
	\Infer1[\scriptsize$\Raa^2$]{\Vartwo}
      \Infer1[\scriptsize$\to_\Intro^3$]{\lnot \lnot \Vartwo \to \Vartwo}
      \Infer[separation = 0.2em]2[\scriptsize$\land_\Intro$]{(\lnot \lnot \Var \to \Var) \land (\lnot \lnot \Vartwo \to \Vartwo)}
    \end{prooftree}
    \ \leadsto 
    \begin{prooftree}[label separation = 0.2em]
	\Hypo{\ulcorner \lnot (\lnot \lnot \Var \to \Var) \urcorner^1}
	\Ellipsis{$\pi(\Var)$}{\bot}
      \Infer1[\scriptsize$\Raa^1$]{\lnot \lnot \Var \to \Var}
	\Hypo{\ulcorner \lnot (\lnot \lnot \Vartwo \to \Vartwo) \urcorner^2}
	\Ellipsis{$\pi(\Vartwo)$}{\bot}
      \Infer1[\scriptsize$\Raa^2$]{\lnot \lnot \Vartwo \to \Vartwo}
      \Infer[separation = 0.2em]2[\scriptsize$\land_\Intro$]{(\lnot \lnot \Var \to \Var) \land (\lnot \lnot \Vartwo \to \Vartwo)}
    \end{prooftree}
    \\ \\
    &\leadsto \ 
    \begin{prooftree}
      \Hypo{\ulcorner \lnot ((\lnot \lnot \Var \to \Var) \land (\lnot \lnot \Vartwo \to \Vartwo)) \urcorner^3}
	\Hypo{\ulcorner \lnot \lnot \Var \to \Var \urcorner^2}
	\Hypo{\ulcorner \lnot \lnot \Vartwo \to \Vartwo \urcorner^1}
      \Infer2[\footnotesize$\land_\Intro$]{(\lnot \lnot \Var \to \Var) \land (\lnot \lnot \Vartwo \to \Vartwo)}
      \Infer2[\footnotesize$\lnot_\Elim$]{\bot}
	\Infer1[\footnotesize$\lnot_\Intro^1$]{\lnot (\lnot \lnot \Vartwo \to \Vartwo)}
	\Ellipsis{$\pi(\Vartwo)$}{\bot}
	\Infer1[\footnotesize$\lnot_\Intro^2$]{\lnot (\lnot \lnot \Var \to \Var)}
	\Ellipsis{$\pi(\Var)$}{\bot}
      \Infer1[\footnotesize$\Raa^3$]{(\lnot \lnot \Var \to \Var) \land (\lnot \lnot \Vartwo \to \Vartwo)}
    \end{prooftree} \ .
  \end{align*}
  }
\end{ex}

\begin{ex}
  Given two proposition symbols $\Var$ and $\Vartwo$, let $\Pi \colon \Var \vdash \Var \lor \Vartwo$ be the following derivation in $\Nk$ (note that there are no 
    instances of $\to_\Intro$ or $\forall_\Intro$
    ):
  {\small
  \begin{align*}
    \Pi \ &= \ 
    \begin{prooftree}
      \Hypo{\ulcorner \lnot \Var \urcorner^1}
      \Hypo{\ulcorner \lnot \Var \urcorner^1}
      \Hypo{\Var}
      \Infer2[\footnotesize$\lnot_\Elim$]{\; \bot \;}
      \Infer1[\footnotesize$\Efq$]{\Var \land \Vartwo}
      \Infer1[\footnotesize$\land_{\ElimOne}$]{\Var}
      \Infer2[\footnotesize$\lnot_\Elim$]{\; \bot \;}
      \Infer1[\footnotesize$\Raa^1$]{\Var}
      \Infer1[\footnotesize$\lor_{\IntroOne}$]{\Var \lor \Vartwo}
    \end{prooftree}
  \end{align*}
  }
  By applying the procedure defined in the proof of Theorem~\ref{thm:postponement}.\ref{thm:postponement.minimal}, we get:
  {\small 
  \begin{equation*}
    \Pi \ 
    \leadsto \ 
    \begin{prooftree}[label separation = 0.2em]
      \Hypo{\ulcorner \lnot \Var \urcorner^1}
      \Hypo{\ulcorner \lnot \Var \urcorner^1}
      \Hypo{\Var}
      \Infer[separation = 0.9em]2[\footnotesize$\lnot_\Elim$]{\; \bot \;}
      \Infer1[\footnotesize$\Efq$]{\Var}
      \Infer[separation = 0.5em]2[\footnotesize$\lnot_\Elim$]{\; \bot \;}
      \Infer1[\footnotesize$\Raa^1$]{\Var}
      \Infer1[\footnotesize$\lor_{\IntroOne}$]{\Var \lor \Vartwo}
    \end{prooftree} \ 
    \leadsto \ 
    \begin{prooftree}[label separation = 0.2em]
      \Hypo{\ulcorner \lnot \Var \urcorner^1}
      \Hypo{\Var}
      \Infer[separation = 0.9em]2[\footnotesize$\lnot_\Elim$]{\; \bot \;}
      \Infer1[\footnotesize$\Raa^1$]{\Var}
      \Infer1[\footnotesize$\lor_{\IntroOne}$]{\Var \lor \Vartwo}
    \end{prooftree} \ 
    \leadsto \ 
    \begin{prooftree}[label separation = 0.2em]
      \Hypo{\ulcorner \lnot (\Var \lor \Vartwo) \urcorner^2}
      \Hypo{\ulcorner \Var \urcorner^1}
      \Infer1[\footnotesize$\lor_{\IntroOne}$]{\Var \lor \Vartwo}
      \Infer[separation = 0.9em]2[\footnotesize$\lnot_\Elim$]{\bot}
      \Infer1[\footnotesize$\lnot_\Intro^1$]{\lnot \Var}
      \Hypo{\Var}
      \Infer[separation = 0.5em]2[\footnotesize$\lnot_\Elim$]{\; \bot \;}
      \Infer1[\footnotesize$\Raa^2$]{\Var \lor \Vartwo}
    \end{prooftree} \,.
  \end{equation*}
  }
\end{ex}


%
\section{\texorpdfstring{Generalized Glivenko's theorem}{Generalized Glivenko's theorem}}
\label{sect:Glivenko}

As already mentioned in \S\ref{sect:intro}-\ref{sect:overview}, an immediate consequence of the postponement of $\Raa$ is the weak normalization of $\Nk \smallsetminus \{\forall_\Intro\}$.
Indeed, a normalization strategy is the following: by Theorem~\ref{thm:postponement}.\ref{thm:postponement.intuitionistic}, any derivation $\pi \colon \Gamma \vdash A$ in $\Nk \smallsetminus \{\forall_\Intro\}$ reduces to a \Jstandard\ derivation $\pi' \colon \Gamma \vdash A$ in $\Nk \smallsetminus \{\forall_\Intro\}$ (which is a derivation in $\Nj 
$, possibly except for its last rule which could be a discharging instance of $\Raa$), then one can apply Prawitz's original weak normalization theorem for $\Nj$ \cite[p.~50]{Prawitz65} to $\pi'$ (or to $\pi'$ without its last rule), so as to obtain a normal derivation $\pi'' \colon \Gamma \vdash A$ in $\Nk \smallsetminus \{\forall_\Intro\}$.

Another consequence of the postponement of $\Raa$ (Theorem~\ref{thm:postponement}, Corollary~\ref{cor:postponement}) is a 
strengthened form of Glivenko's theorem embedding full first-order classical logic not only into the fragment $\{\bot, \top, \lnot, \land, \lor, \to, \exists\}$ of intuitionistic logic (Theorem~\ref{thm:glivenko-intuitionist}), but also into the fragment $\{\bot, \top, \lnot, \land, \lor, \exists\}$ of minimal logic (Theorem~\ref{thm:glivenko}).
The idea is that, given a derivation $\Pi$ in $\Nk$ whose last rule is an instance of $\Raa$, the rest of $\Pi$ is a subderivation in $\Nj$ or $\Nm$; the instance of $\Raa$ can thus be replaced by an instance of $\lnot_\Intro$.

We define a translation $(\cdot)^\Transl$ (resp.~$(\cdot)^\Translj$) on formulas that just redefines the implication and the universal quantifier (resp.~only the universal quantifier) in a  classical way, using the negation, the disjunction and the existential quantifier (resp.~the negation and the existential quantifier).
All other connectives and the existential quantifier are left alone.

\begin{defin}[Minimal and intuitionistic translations]\label{def:translation}
  The \emph{minimal translation} is a function $(\cdot)^\Transl$ associating with every formula $A$ a formula $A^\Transl$ defined by induction on $A$ as follows:
  \begin{align*}
   ( \Var(t_1, \dots, t_n))^\Transl &= \Var(t_1, \dots, t_n)  	& \top^\Transl &= \top					& \bot^\Transl &= \bot \\
    (A \land B)^\Transl &= A^\Transl \land B^\Transl 	& (A \lor B)^\Transl &= A^\Transl \lor B^\Transl 	& (\lnot A)^\Transl &= \lnot A^\Transl \\
    (A \to B)^\Transl &= \lnot A^\Transl \lor B^\Transl & (\forall x  A)^\Transl &= \lnot \exists x \, \lnot A^\Transl & (\exists x  A)^\Transl &= \exists x  A^\Transl
  \end{align*}

  The \emph{intuitionistic translation} is a function $(\cdot)^\Translj$ associating with every formula $A$ a formula $A^\Translj$ defined by induction on $A$ as follows:
  \begin{align*}
    (\Var(t_1, \dots, t_n))^\Translj &= \Var(t_1, \dots, t_n)  	& \top^\Translj &= \top					& \bot^\Translj &= \bot \\
    (A \land B)^\Translj &= A^\Translj \land B^\Translj 	& (A \lor B)^\Translj &= A^\Translj \lor B^\Translj 	& (\lnot A)^\Translj &= \lnot A^\Translj \\
    (A \to B)^\Translj &= A^\Translj \to B^\Translj & (\forall x  A)^\Translj &= \lnot \exists x \, \lnot A^\Translj & (\exists x  A)^\Translj &= \exists x  A^\Translj.
  \end{align*}
  Given a set of formulas $\Gamma$, we set $\Gamma^\Transl = \{A^\Transl \mid A \in \Gamma\}$ and $\Gamma^\Translj = \{A^\Translj \mid A \in \Gamma\}$.
\end{defin}

The difference between $(\cdot)^\Transl$ and $(\cdot)^\Translj$ is only in the translation of $A \to B$.
Our minimal and intuitionistic translations are deeply related to Kuroda's negative translation.
More precisely, if $(\cdot)^{\Transl'}$ and $(\cdot)^{\Translj'}$ are the translations defined as in Definition~\ref{def:translation}, except for 
\begin{align*}
   (\forall x  A)^{\Transl'} &= \forall x \, \lnot\lnot A^{\Transl'} & (\forall x  A)^{\Translj'} &= \forall x \, \lnot\lnot A^{\Translj'},
\end{align*}
then the negative translation $A \mapsto \lnot\lnot A^{\Translj'}$ is the one defined by Kuroda in \cite{Kuroda51}, while the negative translation $A \mapsto \lnot\lnot A^{\Transl'}$ is a variant of Kuroda's one introduced in \cite[p.~231]{FerreiraOliva12}.\label{kurodaoriginal}

Using the terminology of \cite{FerreiraOliva12}, $(\cdot)^\Transl$ and $(\cdot)^\Translj$ are modular translations in the sense that the translation of a formula is based on the translation of its immediate subformulas.
The names ``minimal'' and ``intuitionistic'' associated with $(\cdot)^\Transl$ and $(\cdot)^\Translj$, respectively, are due to the derivability relation they preserve: this will be clarified in Theorems \ref{thm:glivenko}.\ref{thm:glivenko.nottranslated}-\ref{thm:glivenko-intuitionist}.\ref{thm:glivenko-intuitionist.nottranslated} and Propositions~\ref{prop:inverse}.\ref{prop:inverse-minimal}-\ref{prop:inverse-intuitionist} below, which imply that
\begin{equation*}
\begin{aligned}
  \vdash_\Nk A &\quad \mathrm{iff}& \vdash_\Nj \lnot \lnot A^\Translj &\quad \mathrm{iff}& \vdash_\Nm \lnot\lnot A^\Transl.
\end{aligned}
\end{equation*}
A consequence of this fact, together with Remark~\ref{rmk:translation}.\ref{rmk:translation.interderivable} below, is that the translation $A \mapsto \lnot\lnot A^\Transl$ (resp.~$A \mapsto \lnot\lnot A^\Translj$) is modular \emph{negative} according to \cite{FerreiraOliva12}.
Besides, the modular negative translations $A \mapsto \lnot\lnot A^\Transl$ and $A \mapsto \lnot\lnot A^{\Transl'}$ (resp.~$A \mapsto \lnot\lnot A^\Translj$ and $A \mapsto \lnot\lnot A^{\Translj'}$) are the same according to \cite[Definition~2]{FerreiraOliva12}, in the sense that they are interderivable with respect to minimal logic.%
\footnote{Note, in particular, that $\neg\exists{x}\neg A$ is provably equivalent to $\forall{x}\neg\neg A$ in minimal logic.}
However, quite interestingly, they have a very different behavior with respect to the postponement of $\Raa$: only the negative translation $A \mapsto \lnot\lnot A^{\Transl}$ (resp.~$A \mapsto \lnot\lnot A^\Translj$ ) allows one to use Theorem~\ref{thm:postponement}.\ref{thm:postponement.minimal} (resp.~Theorem~\ref{thm:postponement}.\ref{thm:postponement.intuitionistic}), since in $A^{\Transl'}$ (resp.~$A^{\Translj'}$) the universal quantifier might occur, with the disturbing effect pointed out in \S\ref{subsect:seldinstrategy} and Remark~\ref{rmk:forall}.
A variant of our minimal translation is discussed in Appendix \ref{sect:alternative}.

\begin{rmk}\label{rmk:translation}
  For every formula $A$, by induction on $A$ we can prove that:
  
  \begin{enumerate}
    \item\label{rmk:translation.interderivable} $\vdash_\Nk A$ if and only if $\vdash_\Nk A^\Transl$ if and only if $\vdash_\Nk A^\Translj$;

    \item\label{rmk:translation.nobadconnectives} $A^\Transl$ contains no occurrences of $\to$ and $\forall$; $A^\Translj$ contains no occurrences of $\forall$;
    
    \item\label{rmk:translation.free} the free variables in $A^\Transl$ and $A^\Translj$ are the same as in $A$, and $(A \Sub{t}{x})^\Transl = A^\Transl \Sub{t}{x}$ and $(A \Sub{t}{x})^\Translj = A^\Translj \Sub{t}{x}$ for any term $t$;
    
    \item\label{rmk:translation.identity} $A^\Transl = A$ if $A$ contains no occurrences of $\to$ and $\forall$; $A^\Translj = A$ if $A$ contains no occurrences of $\forall$.
  \end{enumerate}
\end{rmk}

Actually, one direction of the equivalences in Remark~\ref{rmk:translation}.\ref{rmk:translation.interderivable} can be reformulated in a more informative and constructive way from a proof-theoretic viewpoint, thanks to (the proof of) the following lemma.

%
\begin{lemma}[Preservation of derivability in $\Nk$ w.r.t. translations]\label{lemma:translation}
  For every derivation $\Pi \colon \Gamma \vdash A$ in $\Nk$ there exist a derivation $\Pi' \colon \Gamma^\Transl \vdash A^\Transl$ in $\Nk \smallsetminus \{\to_\Intro, \forall_\Intro\}$ and a derivation $\Pi'' \colon \Gamma^\Translj \vdash A^\Translj$ in $\Nk \smallsetminus \{\forall_\Intro\}$.
\end{lemma}

\begin{proof}
  By induction on the derivation $\Pi$ in $\Nk$. 
  Let us consider its last rule $\mathsf{r}$.
  Due to Definition~\ref{def:translation}, to prove the existence of a derivation $\Pi' \colon \Gamma^\Transl \vdash A^\Transl$ in $\Nk \smallsetminus \{\to_\Intro, \forall_\Intro\}$ the only interesting cases are when $\mathsf{r}$ is an instance of $\to_\Intro$ or $\to_\Elim$ or $\forall_\Intro$ or $\forall_\Elim$.
  
  \smallskip
  If $\mathsf{r}$ is an instance of $\to_\Intro$, then $\Pi \colon \Gamma \vdash A \to B$ and there is a derivation $\pi \colon \Gamma, A \vdash B$ in $\Nk$ such that
  {\small
  \begin{equation*}
    \Pi = 
    \begin{prooftree}
      \Hypo{\ulcorner A \urcorner^1}
      \Ellipsis{$\pi$}{B}
      \Infer1[\footnotesize$\to_\Intro^1$]{A \to B}
    \end{prooftree}.
  \end{equation*}
  }
  By the induction hypothesis, there is a derivation $\pi' \colon \Gamma^\Transl, A^\Transl \vdash B^\Transl$ in $\Nk \smallsetminus \{\to_\Intro, \forall_\Intro\}$.
  Let 
  {\small
  \begin{equation*}
    \Pi' = 
    \begin{prooftree}[label separation=0.3em]
	\Hypo{\ulcorner \lnot(\lnot A^\Transl \lor A^\Transl) \urcorner^2}
	  \Hypo{\ulcorner \lnot(\lnot A^\Transl \lor A^\Transl) \urcorner^2}
	    \Hypo{\ulcorner A^\Transl \urcorner^1}
	  \Infer1[\footnotesize$\lor_{\IntroOne}$]{\lnot A^\Transl \lor A^\Transl}
	  \Infer2[\footnotesize$\lnot_\Elim$]{\bot}
	  \Infer1[\footnotesize$\lnot_\Intro^1$]{\lnot A^\Transl}
	\Infer1[\footnotesize$\lor_{\IntroTwo}$]{\lnot A^\Transl \lor A^\Transl}
	\Infer[separation=-3em]2[\footnotesize$\lnot_\Elim$]{\bot}
      \Infer1[\footnotesize$\Raa^2$]{\lnot A^\Transl \lor A^\Transl}
	\Hypo{\ulcorner \lnot A^\Transl \urcorner^3}
      \Infer[separation=-1em]1[\footnotesize$\lor_{\IntroOne}$]{\!\!\lnot A^\Transl \lor B^\Transl}
	\Hypo{\ulcorner A^\Transl \urcorner^3}
	\Ellipsis{$\pi'$}{B^\Transl}
      \Infer1[\footnotesize$\lor_{\IntroTwo}$]{\lnot A^\Transl \lor B^\Transl}
      \Infer[separation=1em]3[\footnotesize$\lor_\Elim^3$]{\lnot A^\Transl \lor B^\Transl}
    \end{prooftree}.
  \end{equation*}
  }
\noindent So, $\Pi' \colon \Gamma^\Transl \vdash \lnot A^\Transl \lor B^\Transl$ is a derivation in $\Nk \smallsetminus \{\to_\Intro, \forall_\Intro\}$, where $\lnot A^\Transl \lor B^\Transl = (A \to B)^\Transl$.

  \smallskip
  If $\mathsf{r}$ is an instance of $\to_\Elim$, then $\Pi \colon \Gamma \vdash B$ and there are derivations $\pi_1 \colon \Gamma_1 \vdash A \to B$ and $\pi_2 \colon \Gamma_2 \vdash A$ in $\Nk$ such that $\Gamma = \Gamma_1 \cup \Gamma_2$ and
  {\small
  \begin{equation*}
    \Pi = 
    \begin{prooftree}
	\Hypo{}
      \Ellipsis{$\pi_1$}{A \to B}
	\Hypo{}
      \Ellipsis{$\pi_2$}{A}
      \Infer2[\footnotesize$\to_\Elim$]{ B}
    \end{prooftree}.
  \end{equation*}
  }
  
  \noindent By the induction hypothesis, there are derivations $\pi_1' \colon \Gamma_1^\Transl \vdash \lnot A^\Transl \lor B^\Transl$ and $\pi_2' \colon \Gamma_2^\Transl \vdash A^\Transl$ in $\Nk \smallsetminus \{\to_\Intro, \forall_\Intro\}$.
  Let 
  {\small
  \begin{equation*}
    \Pi' = 
    \begin{prooftree}
	\Hypo{}
      \Ellipsis{$\pi_1'$}{\lnot A^\Transl \lor B^\Transl}
	\Hypo{\ulcorner \lnot A^\Transl \urcorner^1}
	  \Hypo{}
	\Ellipsis{$\pi_2'$}{A^\Transl}
	\Infer2[\footnotesize$\lnot_{\Elim}$]{\bot}
      \Infer1[\footnotesize$\Efq$]{B^\Transl}
      \Hypo{\ulcorner B^\Transl \urcorner^1}
      \Infer3[\footnotesize$\lor_\Elim^1$]{B^\Transl}
    \end{prooftree}.
  \end{equation*}
  }
  
  \noindent Thus, $\Pi' \colon \Gamma^\Transl \vdash B^\Transl$ is a derivation in $\Nk \smallsetminus \{\to_\Intro, \forall_\Intro\}$.

  \smallskip
  If $\mathsf{r}$ is an instance of $\forall_\Intro$, then $\Pi \colon \Gamma \vdash \forall x A$ and there is a derivation $\pi \colon \Gamma \vdash A$ in $\Nk$ such that the variable $x$ is not free in any formula of $\Gamma$ and
  {\small
  \begin{equation*}
    \Pi = 
    \begin{prooftree}
	\Hypo{}
      \Ellipsis{$\pi$}{A}
      \Infer1[\footnotesize$\forall_\Intro$]{\forall x A}
    \end{prooftree}.
  \end{equation*}
  }
  
  \noindent By the induction hypothesis, there is a derivation $\pi' \colon \Gamma^\Transl \vdash A^\Transl$ in $\Nk \smallsetminus \{\to_\Intro, \forall_\Intro\}$.
  By Remark~\ref{rmk:translation}.\ref{rmk:translation.free}, the variable $x$ is not free in any formula of $\Gamma^\Transl$.
  Let 
  {\small
  \begin{equation*}
    \Pi' = 
    \begin{prooftree}
      \Hypo{\ulcorner \exists x \lnot A^\Transl \urcorner^2}
	\Hypo{\ulcorner \lnot A^\Transl \urcorner^1}
	  \Hypo{}
	\Ellipsis{$\pi'$}{A^\Transl}
      \Infer2[\footnotesize$\lnot_{\Elim}$]{\bot}
      \Infer2[\footnotesize$\exists_\Elim^1$]{\bot}
      \Infer1[\footnotesize$\lnot_\Intro^2$]{\lnot \exists x \lnot A^\Transl}
    \end{prooftree}.
  \end{equation*}
  }
  
  \noindent Thus, $\Pi' \colon \Gamma^\Transl \vdash \lnot \exists x \lnot A^\Transl$ is a derivation in $\Nk \smallsetminus \{\to_\Intro, \forall_\Intro\}$, where $(\forall x A)^\Transl = \lnot \exists x \lnot A^\Transl$.

  \smallskip
  If $\mathsf{r}$ is an instance of $\forall_\Elim$, then $\Pi \colon \Gamma \vdash A\Sub{t}{x}$ and there is a derivation $\pi \colon \Gamma \vdash \forall x A$ in $\Nk$ such that 
  {\small
  \begin{equation*}
    \Pi = 
    \begin{prooftree}
	\Hypo{}
      \Ellipsis{$\pi$}{\forall x A}
      \Infer1[\footnotesize$\forall_\Elim$]{A\Sub{t}{x}}
    \end{prooftree}.
  \end{equation*}
  }
  
  \noindent By the induction hypothesis, there is a derivation $\pi' \colon \Gamma^\Transl \vdash \lnot \exists x \lnot A^\Transl$ in $\Nk \smallsetminus \{\to_\Intro, \forall_\Intro\}$.
  Let 
  {\small
  \begin{equation*}
    \Pi' = 
    \begin{prooftree}
      \Hypo{}
	\Ellipsis{$\pi'$}{\lnot \exists x \lnot A^\Transl}
	\Hypo{\ulcorner \lnot A^\Transl \Sub{t}{x} \urcorner^1}
      \Infer1[\footnotesize$\exists_{\Intro}$]{\exists x \lnot A^\Transl}
      \Infer2[\footnotesize$\lnot_\Elim$]{\bot}
      \Infer1[\footnotesize$\Raa^1$]{A^\Transl \Sub{t}{x}}
    \end{prooftree}.
  \end{equation*}
  }
  
  \noindent Thus, $\Pi' \colon \Gamma^\Transl \vdash A^\Transl \Sub{t}{x}$ is a derivation in $\Nk \smallsetminus \{\to_\Intro, \forall_\Intro\}$, where $(A\Sub{t}{x})^\Transl = A^\Transl \Sub{t}{x}$ by Remark~\ref{rmk:translation}.\ref{rmk:translation.free}.
  
  \smallskip
  The proof of the existence of a derivation $\Pi'' \colon \Gamma^\Translj \vdash A^\Translj$ in $\Nk \smallsetminus \{\forall_\Intro\}$ is analogous to the proof of the existence of a derivation $\Pi' \colon \Gamma^\Transl \vdash A^\Transl$ in $\Nk \smallsetminus \{\to_\Intro, \forall_\Intro\}$, but the only interesting cases are when $\mathsf{r}$ is an instance of $\forall_\Intro$ or $\forall_\Elim$.
\end{proof}

In general, derivability in $\Nm$ is not preserved via the translation~$(\cdot)^\Transl$: \textit{e.g.} $\vdash_\Nm \Var \to \Var$ but $\not\vdash_\Nm \lnot \Var \lor \Var$, where $(\Var \to \Var)^\Transl = \lnot \Var \lor \Var$.
Also, a formula $A$ in general is not derivably equivalent to $A^\Translj$  in $\Nj$ (resp.~$A^\Transl$ in $\Nm$), since $\forall x A$ is not equivalent to $\lnot\exists x \lnot A$ in intuitionistic (resp.~minimal) logic.

\begin{thm}[Generalized Glivenko's theorem, minimal version]\label{thm:glivenko}\hfill
  \begin{enumerate}
    \item\label{thm:glivenko.translated} If $\Gamma \vdash_\Nk A$, then $\Gamma^\Transl \vdash_\mathsf{D} \lnot \lnot A^\Transl$ and $\Gamma^\Transl, \lnot A^\Transl \vdash_\mathsf{D} \bot$, where $\mathsf{D} = \Nm \smallsetminus \{\to_\Intro, \to_\Elim, \forall_\Intro, \forall_\Elim\}$.

    \item\label{thm:glivenko.nottranslated} If $\to$ and $\forall$ occur neither in $A$ nor in any formula of $\Gamma$, then the following are equivalent: 

    \vspace{-.5\baselineskip}
    \begin{multicols}{3}
    \begin{enumerate}[label=(\alph*),ref=\alph*]
      \item\label{point:Nk} $\Gamma \vdash_\Nk A$,
      \item\label{point:notnotNm} $\Gamma \vdash_\Nm \lnot \lnot A$,
      \item\label{point:botNm} $\Gamma, \lnot A \vdash_\Nm \bot$.
    \end{enumerate}
    \end{multicols}
    
    \vspace{-.5\baselineskip}
    If moreover $A = \lnot B$, then: $\Gamma \vdash_\Nk \lnot B$ if and only if $\Gamma \vdash_\Nm \lnot B$.
  \end{enumerate}
\end{thm}

\begin{proof}\hfill
  \begin{enumerate}
    \item Since $\Gamma \vdash A$ is derivable in $\Nk$, according to Lemma~\ref{lemma:translation}, there exists a derivation $\Pi \colon \Gamma^\Transl \vdash A^\Transl$ in $\Nk \smallsetminus \{\to_\Intro, \forall_\Intro\}$, and by Theorem~\ref{thm:postponement}.\ref{thm:postponement.minimal}, there exists a derivation $\Pi' \colon \Gamma^\Transl \vdash A^\Transl$ in $\Nk \smallsetminus \{\to_\Intro, \forall_\Intro\}$ with at most one instance of the rule $\Raa$: this instance, if any, is the last rule of $\Pi'$, the rest of $\Pi'$ being a derivation in $\Nm$.
    Only two cases are possible:
    \begin{itemize}
      \item either the last rule of $\Pi'$ is not an instance of $\Raa$, and thus $\Pi'$ is a derivation in $\Nm$, so that $\Pi'' \colon \Gamma^\Transl \vdash \lnot \lnot A^\Transl$ and $\Pi''' \colon \Gamma^\Transl, \lnot A^\Transl \vdash \bot$ are derivations in $\Nm$, where:
      {\small
      \begin{align*}
	\Pi'' \ &= \ 
	\begin{prooftree}
	  \Hypo{}
	  \Ellipsis{$\Pi'$}{A^\Transl}
	  \Hypo{\ulcorner \lnot A^\Transl \urcorner^1}
	  \Infer2[\small$\lnot_\Elim$]{\; \bot \;}
	  \Infer1[\small$\lnot_\Intro^1$]{\lnot\lnot A^\Transl}
	\end{prooftree}
	&&\text{and}&
	\Pi''' \ &= \ 
	\begin{prooftree}
	  \Hypo{}
	  \Ellipsis{$\Pi'$}{A^\Transl}
	  \Hypo{\lnot A^\Transl}
	  \Infer2[\small$\lnot_\Elim$]{\; \bot \;}
	\end{prooftree}\ ;
      \end{align*}
      }
        
        \item or the last rule of $\Pi'$ is an instance of $\Raa$, \Ie
        \begin{equation*}
	  \Pi' =
	  \begin{prooftree}
	    \Hypo{\ulcorner \lnot A^\Transl \urcorner^1}
	    \Ellipsis{$\pi$}{\; \bot \;}
	    \Infer1[\footnotesize$\Raa^1$]{A^\Transl}
	  \end{prooftree}
	\end{equation*}
	where $\pi \colon \Gamma^\Transl, \lnot A^\Transl \vdash \bot$ is a derivation in $\Nm$. 
	So, $\Pi'' \colon \Gamma^\Transl \vdash \lnot \lnot A^\Transl$ is a derivation in $\Nm$ where $\Pi''$ is obtained from $\Pi'$ by replacing the instance of $\Raa$ with an instance of $\lnot_\Intro$ discharging the same assumptions, \Ie
        \begin{equation*}
	  \Pi'' =
	  \begin{prooftree}
	    \Hypo{\ulcorner \lnot A^\Transl \urcorner^1}
	    \Ellipsis{$\pi$}{\; \bot \;}
	    \Infer1[\footnotesize$\lnot_\Intro^1$]{\lnot \lnot A^\Transl}
	  \end{prooftree} \ .
	\end{equation*}
      \end{itemize}
      
      We have thus proved that $\Gamma^\Transl \vdash \lnot \lnot A^\Transl$ and $\Gamma^\Transl, \lnot A^\Transl \vdash \bot$ are derivable in $\Nm$.
      According to Remark~\ref{rmk:translation}.\ref{rmk:translation.nobadconnectives}, neither $A^\Transl$ nor any formula in $\Gamma^\Transl$ contain  occurrences of $\to$ and $\forall$; hence, according to the normalization theorem and the subformula property for $\Nm$ \cite[p.~53]{Prawitz65}, $\Gamma^\Transl \vdash \lnot \lnot A^\Transl$ and $\Gamma^\Transl, \lnot A^\Transl \vdash \bot$ are derivable in $\Nm \smallsetminus \{\to_\Intro, \to_\Elim, \forall_\Intro, \forall_\Elim\}$.
    
    \item 
    \begin{description}
      \item [\eqref{point:Nk} implies \eqref{point:notnotNm}:] By Theorem~\ref{thm:glivenko}.\ref{thm:glivenko.translated}, since $\Gamma \vdash A$ is derivable in $\Nk$, there is a derivation $\Pi \colon \Gamma^\Transl \vdash \lnot \lnot A^\Transl$ in $\Nm$.
      According to Remark~\ref{rmk:translation}.\ref{rmk:translation.identity}, $\Gamma^\Transl = \Gamma$ and $A^\Transl = A$.
      So, $\Pi \colon \Gamma \vdash \lnot \lnot A$ (in $\Nm$).

      \item [\eqref{point:notnotNm} implies \eqref{point:botNm}:] If $\Pi \colon \Gamma \vdash \lnot \lnot A$ is a derivation in $\Nm$, then $\Pi' \colon \Gamma, \lnot A \vdash \bot$ is a derivation in $\Nm$, where
      {\small
      \begin{align*}
	\Pi' \ &= \ 
	\begin{prooftree}
	  \Hypo{}
	  \Ellipsis{$\Pi$}{\lnot \lnot A}
	  \Hypo{\lnot A}
	  \Infer2[\small$\lnot_\Elim$]{\; \bot \;}
	\end{prooftree}\ .
      \end{align*}
      }
      
      \item [\eqref{point:botNm} implies \eqref{point:Nk}:] Since $\Nm \subseteq \Nk$, if $\Pi \colon \Gamma, \lnot A \vdash \bot$ is a derivation in $\Nm$, then $\Pi$ is a derivation in $\Nk$. 
      Therefore, $\Pi' \colon \Gamma \vdash A$ is a derivation in $\Nk$, where
      {\small
      \begin{equation*}
	\Pi' \ = \ 
	\begin{prooftree}
	  \Hypo{\ulcorner \lnot A \urcorner^1}
	  \Ellipsis{$\Pi$}{\; \bot \;}
	  \Infer1[\footnotesize$\Raa^1$]{A}
	\end{prooftree}.
      \end{equation*}
      }
    \end{description}

    This proves the equivalences: \eqref{point:Nk} iff \eqref{point:notnotNm} iff \eqref{point:botNm}.
    
    Suppose now that moreover $A = \lnot B$.
	We show that  $\Gamma \vdash_\Nk \lnot B$ if and only if $\Gamma \vdash_\Nm \lnot B$.
    \begin{description}
      \item 
      [if:]
      Every derivation $\pi \colon \Gamma \vdash_\Nm \lnot B$ is also a derivation in $\Nk$ because $\Nm \subseteq \Nk$.
      \item 
      [only if:]
      Since $\Gamma \vdash_\Nk \lnot B$, there exists $\pi \colon \Gamma \vdash_\Nm \lnot \lnot \lnot B$ according to the implication \eqref{point:Nk}$\Rightarrow$\eqref{point:notnotNm} we have just proved.
      Therefore, $\pi' \colon \Gamma \vdash_\Nm \lnot B$ where 
      {\small
      \begin{equation*}
	\pi' \ = \ 
	\begin{prooftree}
	  \Hypo{}
	  \Ellipsis{$\pi$}{\lnot \lnot \lnot B}
	  \Hypo{\ulcorner \lnot\lnot\lnot B \urcorner^3}
	  \Hypo{\ulcorner \lnot B \urcorner^1}
	  \Hypo{\ulcorner B \urcorner^2}
	  \Infer2[\footnotesize$\lnot_\Elim$]{\bot}
	  \Infer1[\footnotesize$\lnot_\Intro^1$]{\lnot\lnot B}
	  \Infer2[\footnotesize$\lnot_\Elim$]{\bot}
	  \Infer1[\footnotesize$\lnot_\Intro^2$]{\lnot B}
	  \Infer1[\footnotesize$\to_\Intro^3$]{\lnot \lnot \lnot B \to \lnot B}
	  \Infer2[\footnotesize$\to_\Elim$]{\lnot B}
	\end{prooftree}.
      \end{equation*}
      }

    \end{description}
  \end{enumerate}

\vspace{-\baselineskip}
\end{proof}

\begin{thm}[Generalized Glivenko's theorem, intuitionistic version]\label{thm:glivenko-intuitionist}\hfill
  \begin{enumerate}
    \item\label{thm:glivenko-intutionist.translated} If $\Gamma \vdash_\Nk A$, then $\Gamma^\Translj \vdash_\mathsf{D} \lnot \lnot A^\Translj$ and $\Gamma^\Translj, \lnot A^\Translj \vdash_\mathsf{D} \bot$ where $\mathsf{D} = \Nj \smallsetminus \{\forall_\Intro, \forall_\Elim\}$.

    \item\label{thm:glivenko-intuitionist.nottranslated} If $\forall$ occurs neither in $A$ nor in any formula of $\Gamma$, then the following are equivalent: 

    \vspace{-.5\baselineskip}
    \begin{multicols}{3}
    \begin{enumerate}[label=(\alph*),ref=\alph*]
      \item\label{point:Nkj} $\Gamma \vdash_\Nk A$,
      \item\label{point:notnotNj} $\Gamma \vdash_\Nj \lnot \lnot A$,
      \item\label{point:botNj} $\Gamma, \lnot A \vdash_\Nj \bot$.
    \end{enumerate}
    \end{multicols}

    \vspace{-.5\baselineskip}
    If moreover $A = \lnot B$, then: $\Gamma \vdash_\Nk \lnot B$ if and only if $\Gamma \vdash_\Nj \lnot B$.
    \end{enumerate}
\end{thm}

\begin{proof}
  The proof of each part of Theorem~\ref{thm:glivenko-intuitionist} is analogous to the proof of the respective part of Theorem~\ref{thm:glivenko} below, replacing $(\cdot)^\Transl$ with $(\cdot)^\Translj$, $\Nm$ with $\Nj$, $\Nm \smallsetminus \{\to_\Intro, \forall_\Intro\}$ with $\Nj \smallsetminus \{\forall_\Intro\}$, $\Nm \smallsetminus \{\to_\Intro, \to_\Elim, \forall_\Intro, \forall_\Elim\}$ with $\Nj \smallsetminus \{\forall_\Intro, \forall_\Elim\}$.
  To prove Theorem~\ref{thm:glivenko-intuitionist}.\ref{thm:glivenko-intutionist.translated} we use Theorem~\ref{thm:postponement}.\ref{thm:postponement.intuitionistic} instead of Theorem~\ref{thm:postponement}.\ref{thm:postponement.minimal}, and the normalization theorem and the subformula property for $\Nj$ instead of the normalization theorem and the subformula property for $\Nm$ (see \cite[p.~53]{Prawitz65}).
  To prove Theorem~\ref{thm:glivenko-intuitionist}.\ref{thm:glivenko-intuitionist.nottranslated} we use Theorem~\ref{thm:glivenko-intuitionist}.\ref{thm:glivenko-intutionist.translated} instead of Theorem~\ref{thm:glivenko}.\ref{thm:glivenko.translated}.
\end{proof}

An immediate consequence of Theorem~\ref{thm:glivenko-intuitionist}.\ref{thm:glivenko-intuitionist.nottranslated} (resp.~Theorem~\ref{thm:glivenko}.\ref{thm:glivenko.nottranslated}) is the next corollary: in the fragment $\{\bot, \top, \lnot, \lor, \land, \to, \exists\}$ (resp.~$\{\bot, \top, \lnot, \lor, \land, \exists\}$) of first-order logic, the consistency of a set of formulas in classical logic is equivalent to its consistency in intuitionistic (resp.~minimal) logic.

\begin{cor}[Relative consistency of a theory]\label{cor:consistent} Let a theory be a set of formulas $\Gamma$.
  \begin{enumerate}
    \item\label{cor:consistent.intuitionist} If $\forall$ does not occur in any formula of $\Gamma$, then:
  $\Gamma \vdash_\Nk \bot$ if and only if $\Gamma \vdash_\Nj \bot$.

    \item\label{cor:consistent.minimal} If $\to$ and $\forall$ do not occur in any formula of $\Gamma$, then:
  $\Gamma \vdash_\Nk \bot$ if and only if $\Gamma \vdash_\Nm \bot$.
  \end{enumerate}
\end{cor}

\begin{proof}
  The proof of Corollary~\ref{cor:consistent}.\ref{cor:consistent.intuitionist} is analogous to the proof of Corollary~\ref{cor:consistent}.\ref{cor:consistent.minimal}: 
  it is sufficient to replace $\Nm$ with $\Nj$, and use Theorem~\ref{thm:glivenko-intuitionist}.\ref{thm:glivenko-intuitionist.nottranslated} instead of Theorem~\ref{thm:glivenko}.\ref{thm:glivenko.nottranslated}.
  We now prove Corollary~\ref{cor:consistent}.\ref{cor:consistent.minimal}.
  \begin{description}
    \item 
    [only if:]
    Since $\Gamma \vdash \bot$ is a derivable in $\Nk$, then by Theorem~\ref{thm:glivenko}.\ref{thm:glivenko.nottranslated} there is a derivation $\pi \colon \Gamma, \lnot \bot \vdash \bot$ in $\Nm$.
    So, $\Pi \colon \Gamma \vdash \bot$ is a derivation in $\Nm$ where
    {\small
    \begin{equation*}
    \Pi \ = \ 
    \begin{prooftree}
      \Hypo{\ulcorner \bot \urcorner^1}
      \Infer1[\footnotesize$\lnot_\Intro^1$]{\lnot \bot}
      \Ellipsis{$\pi$}{\bot}
    \end{prooftree} \ .
    \end{equation*}
    }
    \item 
    [if:]
    Since $\Nm \subseteq \Nk$, every derivation $\pi \colon \Gamma \vdash \bot$ in $\Nm$ is also 
    in $\Nk$.
    \qedhere
  \end{description}
\end{proof}

The fact that Theorem~\ref{thm:glivenko}.\ref{thm:glivenko.nottranslated} and Corollary~\ref{cor:consistent}.\ref{cor:consistent.minimal} (resp.~Theorem~\ref{thm:glivenko-intuitionist}.\ref{thm:glivenko-intuitionist.nottranslated} and Corollary~\ref{cor:consistent}.\ref{cor:consistent.intuitionist}) are restricted to the fragment $\{\bot, \top, \lnot, \allowbreak\land, \lor, \exists\}$ (resp.~$\{\bot, \top\!, \lnot,\allowbreak \land, \lor, \to, \exists\}$) of the first-order language of classical logic is not a limit because this fragment is equally expressive as full first-order classical logic (with respect to the derivability relation).


For the sake of completeness, (a slightly strengthened version of) the converses of Theorems~\ref{thm:glivenko}.\ref{thm:glivenko.translated} and \ref{thm:glivenko-intuitionist}.\ref{thm:glivenko-intutionist.translated} also hold (the proof is straightforward and left to the reader):

\begin{prop}
\label{prop:inverse}
  Let $A$ be a formula and $\Gamma$ be a set of formulas.
  \begin{enumerate}
    \item\label{prop:inverse-minimal} If $\Gamma^\Transl \vdash_\Nm \lnot\lnot A^\Transl$ (or equivalently $\Gamma^\Transl, \lnot A^\Transl \vdash_\Nm \bot$), then $\Gamma \vdash_\Nk A$.
    \item\label{prop:inverse-intuitionist} If $\Gamma^\Translj \vdash_\Nj \lnot\lnot A^\Translj$ (or equivalently $\Gamma^\Translj, \lnot A^\Translj \vdash_\Nj \bot$), then $\Gamma \vdash_\Nk A$.
  \end{enumerate}
\end{prop}


%
\section{Conclusion}
\label{sect:conclusion}

The literature concerning the connections between classical and constructive logic is extremely rich and prolific. Our aim in this paper was to give a sort of unifying view of some of these results, by adopting a proof-theoretic perspective, and, in particular, by focusing on a very specific technique: that of postponing the application of the rule of \textit{reduction ad absurdum} in the proofs of classical logic.

After having sketched the evolution of this technique starting from the seminal work of Prawitz in his monograph on natural deduction \cite{Prawitz65}, we have focused our attention on a particular strategy of postponement: the one adopted by Seldin \cite{Seldin86}. The interest of this strategy is that it can be characterized in a sort of geometrical way. In this sense, we proposed a modified version of it by reasoning only on the distance from the conclusion of the instances of $\Raa$ present in a certain derivation, and we left aside any consideration on the syntactic structure of the formulas introduced by the instances of $\Raa$. This insensitivity to syntactic considerations makes the technique extensible to logic systems going beyond first-order classical logic, like modal classical logic and, especially, second-order classical logic.%
\footnote{Clearly, in second order classical logic, the postponement of $\Raa$ does not also imply the normalization theorem.} 
We also conjecture that, even if our postponing strategy is a weak one, it is possible to transform it into a strong one, in the sense that the order of application of our reduction steps is not essential: any order of application should allow one to push $\Raa$ downward with respect to all the other rules.

The other aspect on which we focused our attention is the possibility of extracting some constructive content from the postponing strategy that we presented. In particular, we have been able to obtain Glivenko’s theorem in a uniform form, that is, working both for intuitionistic and minimal logic. As for the postponement of $\Raa$, it should not be difficult to extend it to systems that go beyond first-order logic, such as modal logic and second-order logic.

Finally, since the proof of Glivenko's theorem rests on the use of a negative translation, and since negative translation is closely related to the \textit{continuation-passing style} (CPS) transformations in functional programming, it would be interesting to investigate which is the proper computational interpretation (in terms of $\lambda$-calculus) that can be assigned to the negative translation induced by our postponing strategy. In particular, since our translation of classical logic into intuitionistic logic is just a variant of the Kuroda translation, it is reasonable to expect that our translation simulates a call-by-value evaluation strategy in a call-by-name interpreter (see \cite[p. 255]{FerreiraOliva12}, \cite[p. 158 ff.]{Murthy90}). And since we can define also a translation of classical logic into minimal logic, it would be interesting to understand whether this second translation generates a different CPS transformation or not.

\bibliographystyle{sl}
\bibliography{biblio}

\addcontentsline{toc}{section}{References}%

\newpage
\appendix

 	\section{Appendix}

%
\subsection{Hidden negative translations}
\label{sect:hidnegativetrans}

In St{\aa}lmarck's reduction for $\vee$ classical detours -- see (\ref{stalsol}) -- the application of $\Raa$ used for introducing the formula $A \vee B$ disappears, and it is not replaced by any application of $\Raa$ on (one of) the subformulas $A$ and $B$, as it happens, instead, in the case of conjunction. It is thus not possible, in the reduced derivation of (\ref{stalsol}), to operate the substitution of $\Raa$ with a $\neg_\Intro$ in order to obtain $\neg A$ or $\neg B$. It seems, then, that the procedure used for obtaining a negative translation from the normalization of classical logic cannot be applied in the case of disjunction. However, in this reduced derivation, the $\vee_\Elim$ is used in a restricted way, namely with $C = \bot$, \Ie

\begin{center}
\begin{prooftree}
	\Hypo{A \vee B}
	\Hypo{\ulcorner A \urcorner^{1}}
	\Ellipsis{$\pi'$}{\;\bot\;}
	\Hypo{\ulcorner B \urcorner^{1}}
	\Ellipsis{$\pi''$}{\;\bot\;}
	\Infer3[\footnotesize$\vee_\Elim^{1}$]{\;\bot\;}
\end{prooftree}
\end{center}
 
\noindent Now, from the two subderivations $\pi'$ and $\pi''$ it is possible to obtain the formula $\neg A \wedge \neg B$:

\begin{center}
\begin{prooftree}
	\Hypo{\ulcorner A \urcorner^{1}}
	\Ellipsis{$\pi'$}{\;\bot\;}
	\Infer1[\footnotesize$\neg_\Intro^{1}$]{\neg A}
	\Hypo{\ulcorner B \urcorner^{2}}
	\Ellipsis{$\pi''$}{\;\bot\;}
	\Infer1[\footnotesize$\neg_\Intro^{2}$]{\neg B}
	\Infer2[\footnotesize$\wedge_\Intro$]{\neg A \wedge \neg B}
\end{prooftree}
\end{center}

\noindent And since $A \vee B$ and $\neg A \wedge \neg B$ are contradictory in minimal logic, one can infer $\neg (\neg A \wedge \neg B)$. This means that, in the process of normalization, $A \vee B$ behaves like $\neg (\neg A \wedge \neg B)$. There is thus no substantial difference to Prawitz's treatment of disjunction, which is defined from the beginning by means of negation and conjunction (see p.~\pageref{statrans}, \textit{supra}).

It is worth noting that this is not a peculiarity of the $\vee_\Elim$ rule. It can be claimed that all general elimination rules hide a sort of negative translation. According to Schroeder-Heister and Olkhovikov \cite{Schroeder-HeisterOlkhovikov14}, it is possible to flatten the general elimination rules used by von Plato and Siders (see p.~\pageref{genelim}, \textit{supra}) by translating them into formulas of second-order propositional logic. 
In particular, the $\wedge_\Elim$ rule becomes the formula $\forall{X}(((A \wedge B) \to X) \to X)$, the $\vee_\Elim$ rule becomes $\forall{X}(((A \to X) \wedge (B \to X)) \to X)$, and the $\to_\Elim$ rule becomes $\forall{X}((A \wedge (B \to X)) \to X)$. Then, by instantiating $X$ with $\bot$, the $\wedge_\Elim$ rule gets associated with $\neg \neg (A \wedge B)$, the $\vee_\Elim$ rule with $\neg (\neg A \wedge \neg B)$, and the $\to_\Elim$ rule with $\neg (A \wedge \neg B)$. Disjunction and implication are then treated like in the Gödel-Gentzen's translation.


%
 \subsection{Tennant's proof of the Glivenko's theorem}
\label{sect:tennant}

In \cite[pp. 266--274]{Tennant87}, Tennant gives a proof of Glivenko's theorem for minimal and intuitionistic logic by using a double-negation translation that allows him to transform classical rules into derivable rules of minimal and intuitionistic logic. 

What he proves is the following theorem:

\begin{thm}\label{thm:tennant}
\begin{enumerate}
Let $\Gamma$ be a set of formulas, and $(\cdot)^{*}$ be a function that, when applied to $\Gamma$, adds two negations ($\neg \neg$) in front of any formula of $\Gamma$.
\item Every derivation $\Pi : \Gamma \vdash A$ in $\mathsf{D_c}$ can be converted into a derivation $\Pi' : (\Gamma)^{*} \vdash (A)^{*}$ in $\mathsf{D_m}$, where $\mathsf{D_c} = \Nk \smallsetminus \{\to_\Intro, \forall_\Intro\}$ and $\mathsf{D_m} = \Nm \smallsetminus \{\to_\Intro, \forall_\Intro\}$.
\item Every derivation $\Pi : \Gamma \vdash A$ in $\mathsf{D_c}$ can be converted into a derivation $\Pi' : (\Gamma)^{*} \vdash (A)^{*}$ in $\mathsf{D_i}$, where $\mathsf{D_c} = \Nk \smallsetminus \{\forall_\Intro\}$ and $\mathsf{D_m} = \Nj \smallsetminus \{\forall_\Intro\}$.
\end{enumerate}
\end{thm}

In order to better explain how Tennant's idea works, it is easiest to consider a natural deduction setting presented in a sequent calculus style. 
Consider then a rule of $\mathsf{D_c}$ of the form

\begin{center}
\begin{prooftree}
	\Hypo{\Gamma_{1} \vdash A_{1}}
	\Hypo{\ldots}
	\Hypo{\Gamma_{n} \vdash A_{n}}
	\Infer3[\footnotesize$\mathsf{s}$]{\Gamma \vdash A}
\end{prooftree}
\end{center}

\noindent What is first shown by Tennant is that the rule 

\begin{equation}
\begin{prooftree}\label{trans:rule}
	\Hypo{({\Gamma}_{1})^{*} \vdash (A_{1})^{*}}
	\Hypo{\ldots}
	\Hypo{({\Gamma}_{n})^{*} \vdash (A_{n})^{*}}
	\Infer3[\footnotesize$\mathsf{s}^{*}$]{(\Gamma)^{*} \vdash (A)^{*}}
\end{prooftree}
\end{equation}

\noindent is derivable either in $\mathsf{D_m}$ or in $\mathsf{D_i}$.

Consider, for example, the rule

\begin{center}
\begin{prooftree}
	\Hypo{\Gamma \vdash A}
	\Hypo{\Delta \vdash B}
	\Infer2[\footnotesize$\land_\Intro$]{\Gamma, \Delta \vdash A \land B}
\end{prooftree}
\end{center}

\noindent In order to show that the rule 

\begin{center}
\begin{prooftree}
	\Hypo{(\Gamma)^{*} \vdash (A)^{*}}
	\Hypo{(\Delta)^{*} \vdash (B)^{*}}
	\Infer2[\footnotesize$\land^{*}_\Intro$]{(\Gamma)^{*}, (\Delta)^{*} \vdash (A \land B)^{*}}
\end{prooftree}
\end{center}

\noindent is derivable in $\mathsf{D_m}$ or $\mathsf{D_i}$, it is sufficient to consider the following derivation 

\begin{center}
\begin{prooftree}[label separation = 0.2em]
	\Hypo{(\Gamma)^{*} \vdash \neg \neg A}
	\Hypo{(\Delta)^{*} \vdash \neg \neg B}
	\Hypo{}
	\Infer1[\footnotesize$\Ax$]{\neg(A \land B) \vdash \neg(A \land B)}
	\Hypo{}
	\Infer1[\footnotesize$\Ax$]{A \vdash A}
	\Hypo{}
	\Infer1[\footnotesize$\Ax$]{B \vdash B}
	\Infer2[\footnotesize$\land_\Intro$]{A, B \vdash A \land B}
	\Infer2[\footnotesize$\neg_\Elim$]{\neg(A \land B), A, B \vdash \bot}
	\Infer1[\footnotesize$\neg_\Intro$]{\neg(A \land B), A \vdash \neg B}
	\Infer[separation = 0.8em]2[\footnotesize$\neg_\Elim$]{\neg (A \land B), A, (\Delta)^{*} \vdash \bot}
	\Infer1[\footnotesize$\neg_\Intro$]{\neg (A \land B), (\Delta)^{*} \vdash \neg A}
	\Infer[separation = 0.8em]2[\footnotesize$\neg_\Elim$]{\neg(A \land B), (\Gamma)^{*}, (\Delta)^{*} \vdash \bot}
	\Infer1[\footnotesize$\neg_\Intro$]{(\Gamma)^{*}, (\Delta)^{*} \vdash \neg\neg (A \land B)}
\end{prooftree}
\end{center} 

The situation is a little bit more complicated when one has to deal with rules involving a discharge of hypothesis. The problem is that if $A$ is one of the hypotheses to be discharged, then it is not sufficient to translate it into $(A)^{*}$, \Ie~$\neg \neg A$. Otherwise, trying to discharge $\neg \neg A$ will perturb the structure of the derivation, and one will be unable to get the desired conclusion. It then becomes necessary to compose with a proof of $A \vdash \neg \neg A$ in order to carry out the right discharge on $A$.\label{disch}

Consider, for example, the case of the implication introduction rule:

\begin{prooftree*}
	\Hypo{A, \Gamma \vdash B}
	\Infer1[\footnotesize$\to_\Intro$]{\Gamma \vdash A \to B}
\end{prooftree*}

\noindent In order to show that the rule 

\begin{prooftree*}
	\Hypo{(A)^{*}, (\Gamma)^{*} \vdash (B)^{*}}
	\Infer1[\footnotesize$\to^{*}_\Intro$]{(\Gamma)^{*} \vdash (A \to B)^{*}}
\end{prooftree*}

\noindent is derivable, one has to consider the following derivation

\begin{footnotesize}
\begin{prooftree*}[label separation = 0.2em]
	\Infer0[\scriptsize$\Ax$]{\neg(A \to B) \vdash \neg(A \to B)}
	\Infer0[\scriptsize$\Ax$]{\neg A \vdash \neg A}
	\Infer0[\scriptsize$\Ax$]{A \vdash A}
	\Infer[separation = 0.8em]2[\scriptsize$\neg_\Elim$]{\neg A, A \vdash \bot}
	\Infer1[\scriptsize$\neg_\Intro$]{A \vdash \neg \neg A}
	\Hypo{\neg \neg A, (\Gamma)^{*} \vdash \neg \neg B}
	\Infer[separation = -1em]2[{\scriptsize{\textsf{comp}}}]{A, (\Gamma)^{*} \vdash \neg \neg B}
	\Infer0[\scriptsize$\Ax$]{\neg(A \to B) \vdash \neg(A \to B)}
	\Infer0[\scriptsize$\Ax$]{B \vdash B}
	\Infer1[\scriptsize$\to_\Intro$]{B \vdash A \to B}
	\Infer[separation = 0.8em]2[\scriptsize$\neg_\Elim$]{B, \neg(A \to B) \vdash \bot}
	\Infer1[\scriptsize$\neg_\Intro$]{\neg(A \to B) \vdash \neg B}
	\Infer[separation = -0.8em]2[\scriptsize$\neg_\Elim$]{A, \neg(A \to B), (\Gamma)^{*} \vdash \bot}
	\Infer1[\scriptsize$\Efq$]{A, \neg(A \to B), (\Gamma)^{*} \vdash B}
	\Infer1[\scriptsize$\to_\Intro$]{\neg(A \to B), (\Gamma)^{*} \vdash A \to B}
	\Infer[separation = -9em]2[\scriptsize$\neg_\Elim$]{\neg(A \to B), (\Gamma)^{*} \vdash \bot}
	\Infer1[\scriptsize$\neg_\Intro$]{(\Gamma)^{*} \vdash \neg \neg (A \to B)}
\end{prooftree*}
\end{footnotesize}

\noindent where \textsf{comp} is a rule for explicit composition (of derivations), like the one used in \cite{vonPlato16}.

Note that the step of $\Efq$ cannot be avoided. It is for this reason that the $\to_\Intro$ rule has to be dropped in order to obtain Glivenko's theorem for minimal logic (cf. Remark \ref{rmk:tointro}). 

The rule of $\neg_\Intro$ is treated similarly to $\to_\Intro$. However, Tennant works modulo the minimal equivalences $\neg \neg \neg A \leftrightarrow \neg A$ and $\neg \neg \bot \leftrightarrow \bot$; he can thus simplify the form taken by the conclusion  of the translated rule, obtaining

\begin{center}
\begin{prooftree}
	\Hypo{\neg \neg A, (\Gamma)^{*} \vdash \bot}
	\Infer1[\footnotesize$\neg^{*}_{i}$]{(\Gamma)^{*} \vdash \neg A}
\end{prooftree}
\end{center}

Moreover, the translation of the $\Raa$ rule can be simply obtained as a special case of the previous rule, namely by taking $A = \neg B$, \Ie

\begin{center}
\begin{prooftree}
	\Hypo{\neg \neg \neg B, (\Gamma)^{*} \vdash \bot}
	\Infer1[\footnotesize$\Raa^{*}$]{(\Gamma)^{*} \vdash \neg \neg B}
\end{prooftree}
\end{center}

Once the rules of classical logic have been translated into derivable rules of minimal or of intuitionistic logic of the form we just described, it is easy to show that each derivation $\Pi : \Gamma \vdash A$ in classical logic can be translated into a derivation $\Pi' : (\Gamma)^{*} \vdash (A)^{*}$ in minimal or intuitionistic logic, by a step-by-step translation of each rule of $\Pi$. More precisely, the proof of Theorem \ref{thm:tennant} is given by induction on the length of the proof $\Pi : \Gamma \vdash A$ (\cite[p. 273]{Tennant87}). 
The base step is that of the axiom rule, which is trivial. The inductive step depends on the last rule $\mathsf{s}$ applied in $\Pi$. The inductive hypothesis guarantees, for each subderivation $\Pi_{i}$ of $\Pi$, a derivation $\Pi'_{i}$ of the desired form. Then, in order to obtain $\Pi'$, it is sufficient to apply the translation of the rule $\textsf{s}$ to these derivations $\Pi'_{i}$.

In some particular situations, $\Pi$ could already contain a subderivation $\Pi_{i}$ of the same conclusion $\neg \neg A$ (\Ie~$(A)^{*}$) of the desired derivation $\Pi'$, but from a set of assumptions which is only a subset of $(\Gamma)^{*}$. The two derivations $\Pi'_{i}$ and $\Pi'$ are thus the same, modulo some applications of the weakening rule. One could then simply take the subderivation $\Pi'_{i}$, instead of reconstructing $\Pi'$ step-by-step from $\Pi$. It is for this reason that Tennant formulates Theorem \ref{thm:tennant} by considering a derivation $\Pi'$ of $(\Gamma')^{*} \vdash (A)^{*}$, where $\Gamma' \subseteq \Gamma$.\footnote{Note also that Tennant proves Glivenko's theorem for the relevant versions of minimal and intuitionistic logic. However, we will not discuss this point here, since our article simply focuses on standard logical systems, and not on relevant ones.} 

Note that if in the case of minimal logic we restrict to the fragment $\{\neg, \land, \lor, \bot, \exists\}$, and in the case of intuitionistic logic to the fragment $\{\neg, \land, \lor, \allowbreak \to, \bot, \exists\}$, then the translation function $(\cdot)^{*}$ is nothing but the Kuroda translation (see \S\ref{sect:Glivenko}). More precisely, in the case of minimal logic we have that $(A)^{*} = \neg \neg (A)^{\Transl}$, while in the case of intuitionistic logic $(A)^{*} = \neg \neg (A)^{\Translj}$. In particular, when the translations $(\cdot)^{\Transl}$ and $(\cdot)^{\Translj}$ are applied to the two aforementioned fragments, they behave like the identity function, keeping the formulas on which they are applied invariant (see Definition \ref{def:translation}). 
Thus, in order to obtain a result analogue to Theorems \ref{thm:glivenko}.\ref{thm:glivenko.nottranslated} and \ref{thm:glivenko-intuitionist}.\ref{thm:glivenko-intuitionist.nottranslated}., it would be sufficient to apply Theorem \ref{thm:tennant} to a certain derivation $\Pi : \Gamma \vdash A$, obtain then a derivation $\Pi' : (\Gamma)^{*} \vdash (A)^{*}$, and finally compose it with the derivation

\begin{equation}
\label{doubleint}
\begin{prooftree}
	\Hypo{}
	\Infer1[\footnotesize Ax]{\neg C \vdash \neg C}
	\Hypo{}
	\Infer1[\footnotesize Ax]{C \vdash C}
	\Infer2[\footnotesize$\neg_\Elim$]{\neg C, C \vdash \bot}
	\Infer1[\footnotesize$\neg_\Intro$]{C \vdash \neg \neg C}
\end{prooftree}
\end{equation}

\noindent for every formula $\neg \neg C$ present in $(\Gamma)^{*}$. In this way a new derivation $\Pi'' : \Gamma \vdash (A)^{*}$ is obtained (see \cite[p.~274]{Tennant87}).\footnote{Note that this is very similar to the procedure we already discussed in order to deal with the translation of rules that discharge hypothesis (see p.~\pageref{disch}, \textit{supra}).} 

Tennant remarks that if we do not want to drop the universal quantifier -- both in the case of minimal and intuitionistic logic -- then we have to proceed in the following way: for the $\forall_\Elim$ rule, one can simply translate it according to the the schema (\ref{trans:rule}), while for the $\forall_\Intro$ rule, one has to proceed inductively (see \cite[p.~247]{Tennant87}). Consider a derivation $\Pi : \Gamma \vdash \forall{x}A$ in $\Nk$, ending with $\forall_\Intro$. By the induction hypothesis, we can apply Theorem \ref{thm:tennant} to the subderivation $\Pi_{1} : \Gamma \vdash A$ and obtain the derivation $\Pi'_{1} : (\Gamma)^{*} \vdash (A)^{*}$. Apply then the $\forall_\Intro$ rule, and finally add a double negation using (\ref{doubleint}). In other words, the derivation

\begin{equation*}
\Pi \ = \ 
\begin{prooftree}
	\Hypo{}
	\Ellipsis{$\Pi_{1}$}{\Gamma \vdash A}
	\Infer1[\footnotesize$\forall_\Intro$]{\Gamma \vdash \forall{x}A}
\end{prooftree}
\end{equation*}

\noindent of $\Nk$ is transformed into the derivation

\begin{equation*}
\Pi' \ = \ 
\begin{prooftree}
	\Hypo{}
	\Infer1{\neg \forall{x}(A)^{*} \vdash \neg \forall{x}(A)^{*}}
	\Hypo{}
	\Ellipsis{$\Pi'_{1}$}{(\Gamma)^{*} \vdash (A)^{*}}
	\Infer1[\footnotesize$\forall_\Intro$]{(\Gamma)^{*} \vdash \forall{x}(A)^{*}}
	\Infer2[\footnotesize$\neg_\Elim$]{(\Gamma)^{*}, \neg \forall{x}(A)^{*} \vdash \bot}
	\Infer1[\footnotesize$\neg_\Intro$]{(\Gamma)^{*} \vdash \neg \neg \forall{x}(A)^{*}}
\end{prooftree}
\end{equation*}

\noindent of either $\Nm \smallsetminus \{\to_\Intro, \to_\Elim\}$ or $\Nj$.

This seems to suggest that universal formulas can be translated in the following way: $(\forall{x}A)^{*} = \neg \neg \forall{x}(A)^{*}$, which corresponds to the original Kuroda translation, since $\neg \neg \forall{x}(A)^{*}$ corresponds either to $\neg \neg \forall{x}(A)^{\Transl'}$ or to $\neg \neg \forall{x}(A)^{\Translj'}$ (for the definition of $(\cdot)^{\Transl'}$ and $(\cdot)^{\Translj'}$ see p.~\pageref{kurodaoriginal}). 

However, this is not the case. The reason is that, by following Tennant's idea, the way in which universal formulas are translated is not a uniform one, but it depends on their position inside a derivation. In particular, it depends on whether they are in the position of hypothesis or are the conclusion of a $\forall_\Intro$. In the first case, only a double negation is put in front of them, while in the second case a double negation is also put in front of the quantified formula. This means that when the universal quantifier is added to the language, the translation set up by Tennant is no longer functional at the level of formulas. 

This same situation would occur if one tried to adapt Tennant's treatment of the universal quantifier in order to deal with implication in minimal logic: it would be possible to prevent the loss of implication, but the price to pay would be to have a translation from $\Nk$ to $\Nm$ which is not functional at the level of formulas.

\subsection{An alternative reduction strategy}
\label{sect:alternative}

As we already noticed (see Remark~\ref{rmk:tointro}), the reduction steps that we defined in \S\ref{sect:rewrite} make essential use of intuitionistic logic only in the $\to_\Intro$ case\,---\,where the $\Efq$ rule is explicitly introduced for defining the reduction step \eqref{eq:tointro} at p.~\pageref{tointroefq}\,---\,while in all the other cases the appeal to minimal logic is already sufficient for defining the subderivations preceding the application of the $\Raa$ rule. It is for this reason that in order to prove Glivenko’s theorem for minimal logic, we had to give up implication, and define a translation $(\cdot)^\Transl$ such that $(A \to B)^\Transl =  \lnot A^\Transl \lor B^\Transl$. 

We show here that, in fact, implication could be preserved, but the price to pay is the imposition of important restrictions on the form of the consequent of an implication, so that a certain uniformity and generality of the translation is lost.

Consider the following reduction:

{\small
  \begin{align}\label{eq:tointronew}
    \Pi = 
    \begin{prooftree}
      \Hypo{\ulcorner A \urcorner^2, \ulcorner \lnot B \urcorner^1}
      \Ellipsis{$\pi'$}{\; \bot \;}
      \Infer1[\footnotesize$\Raa^1$]{B}
      \Infer1[\footnotesize$\to_\Intro^2$]{A \to B}
      \Ellipsis{$\pi$}{}
    \end{prooftree}
      \quad
      &\leadsto
      \quad
    \begin{prooftree}[label separation = 0.2em]
      \Hypo{\ulcorner \lnot (A \to B) \urcorner^3}
	\Hypo{\ulcorner A \urcorner^2,}
	  \Hypo{\ulcorner \lnot (A \to B) \urcorner^3}
	    \Hypo{\ulcorner B \urcorner^1}
	    \Infer1[\footnotesize$\to_\Intro^0$]{A \to B}
	  \Infer[separation = 0.8em]2[\footnotesize$\lnot_\Elim$]{\bot}
	\Infer1[\footnotesize$\lnot_\Intro^1$]{\lnot B}
	\Infer[no rule,separation=-1.0em]2{}
	\Ellipsis{$\pi'$}{\; \bot \;}
	\Infer1[\footnotesize$\lnot_\Intro^0$]{\lnot \lnot B}
	\Infer[dashed]1[\footnotesize\eqref{eq:dn}]{B}
	\Infer1[\footnotesize$\to_\Intro^2$]{A \to B}
      \Infer[separation = -0.5em]2[\footnotesize$\lnot_\Elim$]{\bot}
      \Infer1[\footnotesize$\Raa^3$]{A \to B}
      \Ellipsis{$\pi$}{}
    \end{prooftree}
    \!\!\!\!\!\! = \Pi'
  \end{align}
}

\noindent The rule (\ref{eq:dn}) is a derivable rule, obtained by appealing to the following theorem of minimal logic:
\begin{equation}
\label{eq:dn}
\vdash_\Nm B \leftrightarrow \lnot\lnot B
\end{equation}
where $B$ is a \textit{negative formula}, \Ie~atomic formulas occur only negated in $B$, and $B$ does not contain $\vee$ nor $\exists$ (see \cite[p. 48]{TroelstraSchwichtenberg00}). 

According to this new reduction, before applying the $\Raa$ rule, only inferential steps coming from minimal logic are used. 
However, such a reduction can be applied only when the consequent of the implication is a negative formula. 
The consequence is that, even if we can extract a translation from classical to minimal logic from this new reduction, this translation will be neither uniform nor general like the one given in Definition~\ref{def:translation}.

More precisely, when we restrict to the propositional case, the translation induced by replacing the reduction rule \eqref{eq:tointro} at p.~\pageref{tointroefq} with this new one corresponds to a translation $(\cdot)^{\Transl^{*}}$ which behaves like $(\cdot)^\Transl$, except for the implication, which is defined as follows:
\begin{align*}
  (A \to B)^{\Transl^{*}} =  A^{\Transl^{*}} \to \neg\neg B^{\Transl^{*}} & &\textup{(where $B$ is a negative formula)} 
\end{align*}

The same considerations can be made in the case of the $\forall_\Intro$ rule. In particular, we could use the following reduction:

{\small
  \begin{align}\label{eq:forallintronew}
    \Pi = 
    \begin{prooftree}
      \Hypo{\ulcorner \neg A \urcorner^1}
      \Ellipsis{$\pi'$}{\; \bot \;}
      \Infer1[\footnotesize$\Raa^1$]{A}
      \Infer1[\footnotesize$\forall_\Intro$]{\forall{x}A}
      \Ellipsis{$\pi$}{}
    \end{prooftree}
      \qquad
      &\leadsto
      \qquad
    \begin{prooftree}
	\Hypo{\ulcorner \neg A \urcorner^1}
	\Ellipsis{$\pi'$}{\; \bot \;}
	\Infer1[\footnotesize$\lnot_\Intro^1$]{\lnot \lnot A}
	\Infer[dashed]1[\footnotesize\eqref{eq:dn}]{A}
	\Infer1[\footnotesize$\forall_\Intro$]{\forall{x}A}
      \Ellipsis{$\pi$}{}
    \end{prooftree}
    = \Pi'
  \end{align}
}

\noindent and then define $(\cdot)^{\Transl^{*}}$ for the universal quantifier as follows:
\begin{align*}
  (\forall{x}A)^{\Transl^{*}} =  \forall{x}(A)^{\Transl^{*}} & &\textup{(where $A$ is a negative formula)} 
\end{align*}

This new translation $(\cdot)^{\Transl^{*}}$ is the same as the one defined in \cite[p.~249]{FerreiraOliva12},\footnote{In fact, this translation was already present in \cite[p. 159]{Murthy90}, but only for the intuitionistic case.} except for the restriction about negative formulas. The need of this restriction seems to be explained by the fact the our translation is \textit{directly} defined from classical to minimal logic (at the level of proof reduction), while in \cite[\S6]{FerreiraOliva12} it is obtained by making an intermediary step through intuitionistic logic (at the level of formulas): first, classical logic is embedded into intuitionistic logic via the Kolmogorov translation; secondly, the translated formulas obtained in this way are embedded into minimal logic via a set of simplification rules definable in minimal logic itself that reduce the number of negations present in a formula. In particular, passing through the Kolmogorov translation allows one to put negations in front of atomic formulas and to eliminate the occurrences of $\vee$ or $\exists$ by using equivalences provable in minimal logic, like $\neg (\neg \neg A \vee \neg \neg B) \leftrightarrow \neg A \wedge \neg B$ or $\neg \exists{x}\neg \neg A \leftrightarrow \forall{x}\neg A$. In this way, the condition on \eqref{eq:dn} is always respected, and thus no restrictions have to be explicitly imposed.


\AFFILIATIONS 

\end{document}